\documentclass[12pt, a4paper, oneside]{book}

\usepackage{amsmath,amssymb,amsthm,graphicx,mathrsfs,bbm,url}
\usepackage{amsthm}
\usepackage{wrapfig}
\usepackage{enumitem}
\usepackage{mathtools}
\usepackage[all]{xy}
\usepackage[utf8]{inputenc}

\usepackage{comment}

\usepackage{tikz}
\usepackage{tikz-cd} 
\usetikzlibrary{positioning}
\usetikzlibrary{matrix}
\usepackage[toc]{appendix}
\usepackage[usenames,dvipsnames]{xcolor}
\usepackage[colorlinks=true,linkcolor=Red,citecolor=Green]{hyperref}
\usepackage[super]{nth}
\usepackage[open, openlevel=2, depth=3, atend]{bookmark}
\hypersetup{pdfstartview=XYZ}
\usepackage[font=footnotesize]{caption}
\usepackage{a4wide}
\usepackage{subcaption}
\usepackage{caption}
\usepackage{a4wide}
\usepackage{tikz}
\usepackage{setspace}
\usepackage[normalem]{ ulem }
\usepackage{soul}

\usetikzlibrary{decorations.pathreplacing}
\usetikzlibrary{arrows}
\usetikzlibrary{shapes.misc}
\usetikzlibrary{shapes.symbols}
\usetikzlibrary{patterns}
\usepackage[bottom=40mm, top=35mm]{geometry}

\usepackage[utf8]{inputenc}
\usepackage[english]{babel}
\usepackage{indentfirst} 

\usepackage[nocompress]{cite}

\usepackage{adjustbox}
\usepackage{wrapfig}

\usepackage{mathtools}
\usepackage{amsthm}
\usepackage{amssymb}
\usepackage{amsmath}
\usepackage[mathscr]{urwchancal}
\usepackage{mathptmx}
\usepackage{latexsym}
\usepackage{graphicx}
\DeclareGraphicsExtensions{.pdf,.jpg}

\usepackage{amsfonts}
\usepackage{bm}
\usepackage{msc}
\usepackage{afterpage}
\usepackage{csquotes}

\usepackage{url}
\usepackage{verbatim}
\usepackage[justification=centering]{caption}
\usepackage{color,xcolor}
\usepackage{ifpdf}
\usepackage{psfrag}
\usepackage{xparse}
\usepackage{bigints}
\usepackage{hyperref}
\usepackage{url}
\usepackage{ulem} 
\usepackage{makeidx}

\usepackage{cleveref}
\usepackage{setspace}
\usetikzlibrary{arrows}
\usetikzlibrary{shapes.misc}
\usetikzlibrary{shapes.symbols}
\usetikzlibrary{patterns}
\usepackage{tikz}
\usetikzlibrary{positioning}
\usetikzlibrary{matrix}

\usepackage{tikz}
\usetikzlibrary{arrows,chains,matrix,positioning,scopes}

\usetikzlibrary{decorations.pathreplacing}
\usetikzlibrary{fadings}
\usepackage{bbm}

\usepackage{emptypage}

\usepackage{stmaryrd} 

\usepackage{titlesec}
\usepackage{epigraph}

\usepackage{pgfplots}
\usepackage{pgfplotstable}
\usetikzlibrary{fillbetween}
\usepgfplotslibrary{fillbetween}
\usetikzlibrary{shapes.misc}
\usetikzlibrary{math} 

\usepackage{algpseudocode}
\usepackage{mathrsfs} 
\usepackage{multirow}

\theoremstyle{plain}
\newtheorem{teo}{Theorem}[section]
\newtheorem{theorem}[teo]{Theorem}
\newtheorem{lemma}[teo]{Lemma}

\newtheorem{corollary}[teo]{Corollary}
\newtheorem{example}[teo]{Example}
\newtheorem*{question}{Question}
\newtheorem*{conjecture}{Conjecture}

\theoremstyle{definition}
\newtheorem{defn}{Definition}[chapter]
\newtheorem{definition}[defn]{Definition}
\theoremstyle{remark}

\newtheorem{remark}[teo]{Remark}

\newcommand{\R}{\mathbb{R}}
\newcommand{\N}{\mathbb{N}}

\usepackage{titlesec} 
\usepackage{microtype} 
\usepackage{textcase} 

\colorlet{halfgray}{black!45} 

\DeclareFixedFont{\chapterNumber}{T1}{pplj}{m}{n}{70} 

\DeclareRobustCommand{\spacedallcaps}[1]{\textls[100]{\noindent\bfseries\LARGE\MakeTextUppercase{#1}}} 

\titleformat{\chapter}[display]%
{\relax}{\mbox{} \marginpar{\vspace*{-3\baselineskip}\color{halfgray}\chapterNumber\thechapter}}{0pt}%
{\raggedright\spacedallcaps}[\normalsize\vspace*{.8\baselineskip}\titlerule]

\usepackage{fancyhdr}
\fancypagestyle{plain}{%
	\fancyhead{}
	
}

\makeindex

\begin{document}


\begin{titlepage}
\begin{center}
\Large
{\textsc{Two Instances of Chaos in Deterministic and Quantum Dynamical Systems}}
\end{center}
\vspace{5mm}
%
\vspace{10mm}
\large \begin{center}
{\textsc{Dissertation}}\\
\vspace{2mm}
{\textsc{zur}}\\
\vspace{2mm}
{\textsc{Erlangung der naturwissenschaftlichen Doktorw\"urde}}\\
\vspace{4mm}
{\textsc{(Dr. sc. nat.)}}\\
\vspace{4mm}
{\textsc{vorgelegt der}}\\
\vspace{2mm}
{\textsc{Mathematisch--naturwissenschaftlichen Fakult\"at}}\\
\vspace{2mm}
{\textsc{der}}\\
\vspace{2mm}
{\textsc{Universit\"at Z\"urich}}\\
\vspace{10mm}
{\textsc{von}}\\
\vspace{2mm}
{\textsc{Andrea Ulliana}}\\
\vspace{4mm}
{\textsc{aus}}\\
\vspace{2mm}
{\textsc{Italien}}\\
\vspace{10mm}
{Promotionskommision}\\
\vspace{4mm}
{Prof. Dr. Artur Avila (Vorsitz und Leitung)}\\
{Prof. Dr. Alexander Gorodnik}\\
{Prof. Dr. Corinna Ulcigrai}\\

\vspace{10mm} {Z\"urich, 2026}
\end{center}
\end{titlepage}


\frontmatter
\chapter*{Abstract}\markboth{ABSTRACT}{} \markright{ABSTRACT}{} 
This thesis consists of two distinct projects situated in the areas of smooth dynamics and spectral theory, respectively. They are united by a common interest in mechanisms of chaos and statistical behavior in classical and quantum dynamical systems.\newline
The first concerns smooth dynamics. We prove that all ergodic linear automorphisms of the $N$-dimensional torus with two-dimensional center are stably ergodic, including all ergodic automorphisms in dimensions $N\leq 5$ and $N=7$. This generalizes a previous result of Rodriguez-Hertz, which required an additional algebraic condition on the characteristic polynomial of the linear automorphism.\newline
The second project deals with spectral theory of Schr\"odinger operators. We prove that delocalization of most eigenvectors is topologically common in the space of deterministic Schr\"odinger Operators on a given large finite graph, provided that the IDS satisfies a suitable regularity condition. This result generalizes a recent theorem of Avila and Damanik. We also describe a flexible family of graphs satisfying our criterion, by proving a variant of the Thouless formula.


\setcounter{tocdepth}{1}
\tableofcontents




\mainmatter
\chapter*{Introduction} 
\phantomsection
\addcontentsline{toc}{chapter}{Introduction}
This thesis collects the results of two projects.
The first one, joint with Fernando Argentieri, deals with the problem of the stability of ergodicity of toral linear automorphisms, in the context of Smooth Dynamical Systems.  The second one, joint with Artur Avila, concerns the study of the delocalization of eigenvectors for discrete Schr\"odinger Operators on large graphs, with applications to Quantum Dynamics.

\paragraph*{Dynamical Systems} The field of Dynamical Systems studies the asymptotical or eventual properties of an iterated procedure. In a more general sense, a dynamical system is the action of a semigroup $G$ on a set $X$. The semigroup may represent time, discrete if $G=\mathbb{N}$ and continuous if $G=\mathbb{R}^+$. In the first case, we deal with the iteration of a map $f:X\rightarrow X$ and, in the second one, with a flow $\{f_t\}$ on $X$. This framework includes a great variety of systems, ranging from PDEs to Markov chains.\newline
Classical questions deal with the qualitative study of orbits, e.g. the sequences $(f^n(x))_n$, for discrete time systems. Depending on the structure of $X$, one can analyse properties of topological nature, as  recurrence, or statistical ones, as equidistribution (ergodicity).\newline
By Smooth Dynamics we mean the study of the dynamics of a diffeomorphism $f$ on a smooth manifold $M$.

\paragraph*{Quantum Dynamics}

Quantum Dynamics is the study of how quantum systems evolve over time. The state of such a system is encoded by a wave function $\phi$, which is an element of a suitable Hilbert space $\mathcal{H}$. Its time-evolution obeys the Schr\"odinger Equation
\begin{equation*}
    i\frac{d\phi}{dt} = -H\phi,
\end{equation*}
where $H:\mathcal{H}\rightarrow\mathcal{H}$ is the Hamiltonian of the system. Qualitative features of the solutions are deeply linked to the Spectral Theory of such operators.\newline
A remarkable class of Hamiltonians is give by the Schr\"odinger Operators, which have the form
\begin{equation*}
    H = \Delta + V,
\end{equation*}
where $\Delta$ is the Laplacian (possibly the combinatorial one) and $V$ is a diagonal operator, called potential. For the class of the Ergodic Schr\"odinger Operators, the potential is defined by sampling an observable along the orbits of an ergodic dynamical system. This wide framework links the study of many physical models, such as the Anderson model or quasicystals, to a dynamical system analysis.

\paragraph*{Outline}

Chapter \ref{Chapter - Preliminaries} provides a general introduction to both Dynamical Systems and the Spectral Theory of Schr\"odinger Operators, highlighting the many points of contact. These can be found, for example, in the first two sections: "Ergodic Theory" and "Cocycles and Lyapunov Exponents". Moreover, we describe the specific context and background of the works collected in the following two chapters.\newline
Chapter \ref{Chapter - Torus} consists of the submitted version of the manuscript "Advances on Stable Ergodicity of Toral Automorphisms", coauthored with Fernando Argentieri. We prove that all ergodic linear automorphisms with $2$-dimensional center of the $N$-dimensional Torus are, in fact, stably ergodic. This result generalizes a theorem of Rodriguez-Hertz and is related to a program of Pugh and Shub, which is central in smooth dynamics.\newline
Chapter \ref{Chapter - Schreodinger} contains a preliminary version of the work "Quantitative Delocalization For Schr\"odinger Operators on Certaian Families of Graphs", joint with Artur Avila. We consider Schr\"odinger Operators on finite graphs and we show that, under some geometric assumptions, delocalization of most eigenfunctions  can be achieved with a small perturbation of the potential, building on a previous work of Avila and Damanik.

\chapter{Preliminaries}\label{Chapter - Preliminaries}

\section{Ergodic Theory}
Ergodic Theory has its roots in Boltzmann's Ergodic Hypotesis, according to which, over time, the state of a chaotic system visits every region of the phase space with frequency proportional to their volume. In the modern understanding, this is a feature that characterizes all irreducible measurable dynamical systems.\newline

We consider a map $T:X\rightarrow X$ and we endow $X$ with a $\sigma$-algebra $\mathcal{B}$. We say that a measure $\mu$ on $(X,\mathcal{B})$ is invariant if $T_*\mu=\mu$, i.e. $\mu(T^{-1}(B))=\mu(B)$ for any $B\in\mathcal{B}$.

The existence of invariant measures alone provides rich information about the dynamics of the system. The following theorem is one of the most prominent examples.

\begin{theorem}[Poincar\'e recurrence theorem, \cite{poincare1890probleme}] Let $(X,\mathcal{B},\mu,T)$ be a measure preserving system and $A\subset X$ a measurable set. If $\mu(A) >0$ then for $\mu$-a.e. $x\in A$ there exists $n>0$ such that $T^n(x)\in A$.
\end{theorem}

The same system $(X,\mathcal{B},T)$ might have several different invariant measures, which detect very different aspects of its dynamics. For instance one can consider a diffeomorphism $f$ of a Riemmanian manifold $M$, that satisfies $\det(\text{Jac}(f))=1$ and that has a periodic point $x$ of period $p$. Both the measures
\begin{equation}
    \mu = \text{vol}_M \quad \text{ and } \quad \nu=\sum_{i=0}^p \delta_{f^i(x)}
\end{equation}
are invariant (together with their linear combinations).

An invariant probability measure is ergodic if it is "dynamically indecomposible".

\begin{definition}
    Let $(X,\mathcal{B},\mu,T)$ be a measure preserving system. A measurable set $A\subset X$ is invariant if $T^{-1}(A)=A$. The system is ergodic if there are no non-trivial invariant sets; i.e. every invariant set has measure $0$ or $1$.
\end{definition}

\begin{remark}
    The following are equivalent conditions to ergodicity:
    \begin{enumerate}
        \item If $T^{-1}(A)=A$ mod $\mu$ then $\mu(A)=0$ or $1$;
        \item Every $f\in L^p(\mu)$ with $f\circ T=f$ a.e.  is constant a.e. (for any $1\leq p \leq \infty$).
    \end{enumerate}
\end{remark}

The ergodic measures are in fact the building blocks of all invariant measures, as we now state precisely. Let us denote the family of all invariant probability measure of $(X,\mathcal{B},T)$ by $P_T(X)$. We notice that $P_T(X)$ is a convex set.

\begin{theorem}
    The ergodic invariant measures are the extreme points of $P_T(X)$.\newline
    Moreover, if $\mu, \nu$ are ergodic measures, then either $\mu= \nu$ or $\mu\perp\nu$.
\end{theorem}

We now connect the definition of ergodicity to the intuition we first described.\newline
Given a probability space $(X,\mathcal{I},\mu)$ and a function $f\in L^1(\mu)$ we denote by $E[f|\mathcal{I}]$ its conditional expectation. In particular, if $f$ is invariant $E[f|\mathcal{I}]=f$ and if the system is ergodic $E[f|\mathcal{I}]=\int_fd\mu$. We recall that if $f\in L^2(\mu)$, then $E[f|\mathcal{I}]$ coincides with the orthogonal projection of $f$ to the subspace of $\mu$-a.e. $T$-invariant functions.

\begin{theorem}[Ergodic theorem] Let $(X,\mathcal{B},\mu,T)$ be a measure-preserving system, let $\mathcal{I}$ denote the $\sigma$-algebra of invariant sets. Then for any $f\in L^2(\mu)$,
    \begin{equation}
    \frac{1}{N} \sum_{n=0}^{N-1}T^nf \rightarrow E[f|\mathcal{I}] \quad\text{ in }L^2(\mu),
    \end{equation}
and for any $f\in L^1(\mu)$,
    \begin{equation}
    \frac{1}{N} \sum_{n=0}^{N-1}T^n f \rightarrow E[f|\mathcal{I}] \quad\text{ a.e. and in }L^1(\mu).
    \end{equation}
\end{theorem}

The first part of the statement is known as Mean Ergodic Theorem, due to Von Neumann \cite{neumann1932proof}, while the second one is the Pointwise Ergodic Theorem, due to Birkohoff \cite{birkhoff1931proof}. They state that whenever a measure is invariant, its typical points' orbits have a stationary long time statistical behaviour. Moreover, if the system is ergodic, this behaviour is the same for a.e. point and the average of observables over time (i.e. along orbits) coincides with the ones over space (i.e. with respect to the invariant measure $\mu$). Indeed, in view of our previous remarks on $E[f|\mathcal{I}]$, we have the following carachterization of ergodicity, that also recovers Boltzmann original formulation.

\begin{corollary}
    Let $(X,\mathcal{B},\mu,T)$ be a dynamical system. Let $\mathcal{F}\subset L^2(\mu)\cap L^1(\mu)$ be a dense set of functions. Then the system is ergodic if and only if the limit of the Birkhoff averages is constant a.e. for all $f\in\mathcal{F}$, i.e.
    \begin{equation}
    \frac{1}{N} \sum_{n=0}^{N-1}T^n f \rightarrow \int_Xfd\mu \quad\text{ a.e.}.
    \end{equation}
\end{corollary}

\begin{remark}
    Whenever $T$ is an invertible map one can consider the backwards Birkhoff averages. Since the $\sigma$-algebra of invariant sets is the same for $T$ or $T^{-1}$, we have that, for $f\in L^1(\mu)$,
    \begin{equation}
        \frac{1}{N} \sum_{n=0}^{N-1}T^n f(x) = \frac{1}{N} \sum_{n=0}^{N-1}T^{-n} f(x) \quad \text{ for a.e. }x.
    \end{equation}
    
\end{remark}

\section{Cocycles and Lyapunov Exponents}

Let $(X,T)$ be a dynamical system. A cocycle over $(X,T)$ is the data of a vector bundle $V$ over $X$ and a linear bundle map $A:V\rightarrow V$ that lifts $T$. In other words, $A$ is defined by linear maps $A_x:V_x\rightarrow V_{T(x)}$ for every $x\in X$, giving the following diagram.
\[ \begin{tikzcd}
V \arrow{r}{A} \arrow[swap]{d}{\pi} & V \arrow{d}{\pi} \\%
X \arrow{r}{T}& X
\end{tikzcd}
\]

If $V$ is a trivial bundle, i.e. $V=A\times\mathbb{R}^d$, we can identify $V$ with a map $V:A\rightarrow \text{Mat}(d)$, or to some more specific group of matrices. When possible, this description is typically preferred.\newline

The underlying dynamics may be just measurable, continuous or differentiable. This flexibility allows to treat very different systems under the cocycle formalisms. We describe some examples. For more details one can see \cite{viana2014lectures}.

\paragraph*{Derivative Cocycle} Let $M$ be a smooth manifold and $f:M\rightarrow M$. Denoting by $TM$ the tangent space of $M$ and by $Df$ the differential of $f$ we have that $Df:TM\rightarrow TM$ is a cocycle over $(M,f)$.

\paragraph*{Random Matrices Product} We consider a probability space $(\Sigma,\nu)$. We define $\Omega=\Sigma^{\mathbb{Z}}$, $\mu=\nu^{\otimes \mathbb{Z}}$ and we let $T$ be the shift map on $\Omega$. $(\Omega,T)$ is a measure preserving dynamical system. A map $a:\Sigma\rightarrow\text{GL}(d)$ defines the cocycle $A:\Omega\rightarrow\text{GL}(d)$ by the formula $A(\omega)=a(\omega_0)$. The iteration $A^n$ produces the random product of $n$ matrices indipendently identically distributed on $\text{GL}(d)$ according to $a_*\nu$, in fact
\begin{equation}
    A^n({\omega})=a(\omega_0)a(\omega_1) \cdots   a(\omega_{n-1}).
\end{equation}

\paragraph*{Schr\"odinger Cocycle} We consider the Schr\"odinger operator $H_x$ on $\ell^2(\mathbb{Z})$ with potential $V_x$ defined by the dynamical system $(X,T)$ and the observable $v:X\rightarrow\mathbb{R}$. Given $E\in\mathbb{C}$ one can consider the eigenvector equation
\begin{equation}\label{ev eq}
    Eu(n) = u(n-1)+V_x(n)u(n)+u(n+1)
\end{equation}
Its formal solutions must satisfy
\begin{equation}
    \begin{pmatrix}
        u(n+1)\\
        u(n)\\
    \end{pmatrix}=
    \begin{pmatrix}
        E - V(T^n(x)) & -1\\
        1& 0\\
    \end{pmatrix}
    \begin{pmatrix}
        u(n)\\
        u(n-1)\\
    \end{pmatrix}.
\end{equation}
Therefore the cocycle $A:X\rightarrow SL(2,\mathbb{R})$ over $(X,T)$, defined by
\begin{equation}
    A(x) = \begin{pmatrix}
        E - V(x) & -1\\
        1& 0\\
    \end{pmatrix},
\end{equation}
describes the solutions of \ref{ev eq}, since $(u(n+1),u(n))^T=A^n(x)(u(1),u(0))^T$ for any $n\in\mathbb{Z}$.\newline

In the following we assume the fibers of $V$ to be endowed with a norm $\|\cdot\|$, that may vary over $X$ with measurable, continuous of smooth regularity.\newline
We are interested in the asymptoptic behaviour of the iterates of the $A$ over the fibers of $V$, specifically on the growth of $\|A^Nv\|$ for $v\in V_x$ for a point $x\in X$.\newline
In the concrete, among many examples, this gives information about the hyperbolicity of the diffeomorphism $f$ in the case of the derivative cocycle and the existence of integrable solutions to \ref{ev eq} in the case of the Schr\"odinger Cocycle.\newline

The growth of $\|A^Nv\|$ need not have a definite exponential rate in general. However, if the base system $(X,T)$ preserves a measure $\mu$, as another manifestation of the ergodic theorem, that is true for a $\mu$-generic point $x\in X$ and every $v\in V_x$. We also have a good description on the dependance in $v$. We remark the earlier fundational of Furstenberg and Kensten \cite{furstenberg1960products} regarding the growth of $\|A^N\|$.

\begin{theorem}[Oseledets Multiplicative Ergodic Theorem, \cite{oceledec1968multiplicative}]
Suppose $(V,A)$ is a cocycle over a probability measure
preserving system $(X,\mathcal{B},\mu,T)$. Suppose that
\begin{equation}
\int_X \log \|A_x\| d\mu(x) <+ \infty,
\end{equation}
Then, for $\mu$-a.e. $x\in X$,  there exist $A$-invariant and measurable $k(x)>0$, real numbers $\lambda_1(x) > \lambda_2(x) >\dots> \lambda_{k(x)}(x)$ and subspaces $V_x= V_1(x)\supset \dots \supset V_{k(x)}(x)\supset V_{k(x)+1}(x)=\{0\}$, such that:
\begin{equation}
    \lim_{N\rightarrow+\infty} \frac 1N \log\|A^Nv\| = \lambda^i(x) \quad \text{ for any } v\in V_i(x)\setminus V_{i+1}(x).
\end{equation}
\end{theorem}

The numbers $\lambda_1(x) > \lambda_2(x) >\dots> \lambda_{k(x)}(x)$ take the name of Lyapunov exponents, and $\dim(V^i_x)-\dim(V_x^{i+1})$ are their multiplicities.

\begin{remark}
    When the base system $(X,\mathcal{B},\mu,T)$ is ergodic, the Lyapunov exponents and their multiplicities do not depend on $x$.
\end{remark}

Whenever the map $A:V\rightarrow V$ is invertible (and hence $T$ as well), we have the following strengthening to Oseledets theorem.

\begin{theorem}[Oseledets]
Suppose $(V,A)$ is an invertible cocycle over an invertible probability measure
preserving system $(X,\mathcal{B},\mu,T)$. Suppose that
\begin{equation}
\int_X \log \|A_x\| d\mu(x) <+ \infty,
\end{equation}
Then, for $\mu$-a.e. $x\in X$,  there exist $A$-invariant and measurable $k(x)>0$, real numbers $\lambda_1(x) > \lambda_2(x) >\dots> \lambda_{k(x)}(x)$ and subspaces $E_1(x)\oplus \dots \oplus E_{k(x)}(x)=V_x$, such that:
\begin{equation}
    \lim_{N\rightarrow\pm\infty} \frac 1N \log\|A^Nv\| = \lambda^i(x) \quad \text{ for any } v\in E_i(x)\setminus\{0\}.
\end{equation}
\end{theorem}

\begin{definition}
    A cocycle $(V,A)$ over a measure preserving system $(X,\mathcal{B},\mu,T)$ is said to be non uniformly hyperbolic if, for $\mu$-a.e. $x$, no Lyapunov exponent is $0$.
\end{definition}

\section{Smooth Dynamics and Toral Automorphisms}

\subsection{Linear Automorphisms of the Torus}

We define the $N$-dimensional Torus as the quotient $\mathbb{T}^N=\mathbb{R}^N/\mathbb{Z}^N$ and we denote by $\pi$ the covering map from $\mathbb{R}^N$. Given an element of $SL(N,\mathbb{Z})$, by composition with $\pi$, it induces a group homomorphism $A:\mathbb{R}^N\rightarrow\mathbb{R}^N/\mathbb{Z}^N$. Since its kernel coincides with $\mathbb{Z}^N$, the induced map $A:\mathbb{R}^N/\mathbb{Z}^N\rightarrow \mathbb{R}^N/\mathbb{Z}^N$ is a well-defined isomorphism of the Torus.
\[ \begin{tikzcd}
\mathbb{R}^N \arrow{r}{A} \arrow[swap]{d}{\pi} & \mathbb{R}^N \arrow{d}{\pi} \\%
\mathbb{T}^N \arrow{r}{A}& \mathbb{T}^N
\end{tikzcd}
\]

\begin{remark}
    Since the Lebesgue measure on $\mathbb{R}^N$ is invariant under the action of $\mathbb{Z}^N$ it descends to the base space $\mathbb{T}^N$. The induced measure is, in fact, a probability measure.\newline
    Moreover $A$ preserves the induced Lebesgue measure on $\mathbb{T}^N$.
\end{remark}

\begin{theorem}[Halmos, \cite{halmos1943automorphisms}]
    $A$ is ergodic if and only if no eigenvalue of $A$ is a root of unity.
\end{theorem}

\begin{proof}
    We consider an $A$-invariant function $f\in L^2(\mathbb{T}^N)$.\newline
    The Fourier series $\hat f\in \ell^2(\mathbb{Z}^N)$ satisfies
    \begin{equation}
        f(x) = \sum_{n\in\mathbb{Z}^N} \hat f(n) e^{2\pi i x\cdot n}.
    \end{equation}
    We also have that
    \begin{align}
        f(Ax) & = \sum_{n\in\mathbb{Z}^N} \hat f(n) e^{2\pi i Ax\cdot n}\\
        &= \sum_{n\in\mathbb{Z}^N} \hat f(n) e^{2\pi i x\cdot A^Tn}\\
        &= \sum_{n\in\mathbb{Z}^N} \hat f((A^T)^{-1}n) e^{2\pi i x\cdot n}
    \end{align}
    Equating $f(x)$ and $f(Ax)$ and calling $B=(A^T)^{-1}$, we have that $\hat f$ must be constant along the orbits $(B^kn)_{k\in\mathbb{Z}}$ for every $n\in\mathbb{Z}^N$. If all such orbits are infinite, $\hat f(n)$ must vanish away from 0 not to contradict its square-integrability. On the other hand, a finite non-trivial orbit, if set as support for $\hat f$, generates a non-constant invariant function $f\in L^2(\mathbb{T}^N)$.\newline
    In conclusion, $A$ is ergodic if and only if the equation $B^kn=n$ has no positive solutions $(k,n)$. This translates to $B$ (and equivalently $A$) not having any root of unity as eigenvalue.
\end{proof}

The the Harmonic Analysis approach is, although very effective, not robust. In the study of $\mathcal{C}^k$ perturbations of $A$, geometrical methods have been more successful.\newline

We recall that differential of $A$ is constant and coincides with $A$ itself. We have the $A$-invariant decomposition
\begin{equation}
    \mathbb{R}^N = E^s \oplus E^c \oplus E^u,
\end{equation}
where $E^s$, $E^c$ and $E^u$ are the sums of the generalized eigenspaces of $A$ corresponding to the eigenvalues of modulus less than 1, exactly $1$ and more than 1, respectively.\newline

The characteristic polynomial $p_A(x)$ of $A$ belongs to $\mathbb{Z}[x]$. As a first consequences, the generalized eigenvalues corresponding to $E^c$ must be paired by complex conjugacy, implying that $\dim(E^c)$ is even.  Moreover, if all root of $p_A(x)$ are unitary, they must be roots of unity. This implies, that if $A$ is ergodic $E^s$ and $E^u$ are always non trivial (recall that $\det(A)=1$), providing some degree of hyperbolicity.\newline

Depending on whether $E^c$ is trivial or not, $A$ is uniformly or partially hyperbolic. We introduce the general theory of such maps in the following subsections, with particular attention to the stability of the properties in question.

\subsection{Uniform Hyperbolicity}

Let $M$ be a Riemmanian $\mathcal{C}^k$ closed manifold and $f:M\rightarrow M$ a $\mathcal{C}^k$ diffeomorphism.

\begin{definition}
    An $f$-invariant closed set $\Lambda\subset M$ is hyperbolic if there exist numbers $\lambda<1$, $C>0$ and a $Df$-invariant splitting $T_{\Lambda}M=E^s\oplus E^u$ into measurable bundles, such that, for any $x\in\Lambda$,
    \begin{align}
        \|Df_x^n v\|&\leq C\lambda^n \|v\| \quad \text{ for every } v\in E^s_x,\\
        \|Df_x^{-n} v\|&\leq C\lambda^n \|v\| \quad \text{ for every } v\in E^u_x.
    \end{align}
\end{definition}
The regularity for the bundles $E^s$ and $E^u$ automatically upgrades to continuous.\newline
The famous Smale's horseshoe gives an example of hyperbolic set.

\begin{theorem}[Hadamard, Perron, \cite{hadamard1902classe},\cite{perron1929stabilitat}]
    Let $\Lambda\subset M$ be an hyperbolic set for a $\mathcal{C}^1$ diffeomorphism $f:M\rightarrow M$. For every $x\in\Lambda$, the sets
    \begin{gather}
        W^s(x)=\{y\in M : \lim_{n\rightarrow +\infty} d(f^n(x),f^n(y)) \rightarrow0\},\\
        W^u(x)=\{y\in M : \lim_{n\rightarrow +\infty} d(f^{-n}(x),f^{-n}(y)) \rightarrow0\}.
    \end{gather}
    are injectively immersed submanifolds, which are as smooth as $f$ and depend continuously in $x$. Moreover:
    \begin{equation}
        T_xW^*(x)=E^*(x)\quad \text{ and } \quad f(W^*(x))=W^*(f(x)), \quad \text{ for }*=s,u.
    \end{equation}
\end{theorem}

We refer to $W^s(x)$ and $W^u(x)$ as the stable and unstable manifolds of $x$.

 \begin{definition}
     If $\Lambda=M$ we say that $f$ is Anosov, or uniformly hyperbolic.
 \end{definition}

\begin{example}
    Any toral automorphism $A:\mathbb{T}^N\rightarrow\mathbb{T}^N$ with no unitary eigenvalue is an Anosov diffeomorphism. Among these we find the famous Arnold's cat map, i.e. the toral automorphism induced by the matrix
    \begin{equation}
        A=\begin{pmatrix}
            2 & 1\\
            1 & 1
        \end{pmatrix}.
    \end{equation}
\end{example}

\begin{theorem}
    The set of Anosov diffeomorphisms is open in $\mathcal{C}^1(M)$.
\end{theorem}

For $f$ Anosov, the stable and unstable manifolds exist for every point, producing two transverse $Df$-invariant foliations, $\mathcal{F}^s$ and $\mathcal{F}^u$ that integrate $E^s$ and $E^u$, respectively.

\begin{remark}
    If $f$ is a $\mathcal{C}^{1+\alpha}$ Anosov diffeomorphism, then the bundles $E^s$ and $E^u$ are automatically H\"older-continuous.
\end{remark} 

Given a set $V\subset M$, we denote by $W_V^*(x)$ the connected component of $W^*(x)\cap V$ containing $x$. We say that $V$ has a product structure if, for any $x,y\in V$ the intersection beween $W_V^s(x)\cap W_V^u(y)$ contains exactly one point. If $f$ is Anosov, any point $x\in M$ has an open neighbourhood $V(x)$ with a product structure.\newline
Given $V$ with a product structure, for any $x,y\in V$ we can define the stable holonomy $h^s:W_V^u(x)\rightarrow W_V^u(y)$ and the unstable one $h^u:W_V^s(x)\rightarrow W_V^s(y)$, respectively, by
\begin{equation}
    h^s(z) = W^s_V(z)\cap W^u_V(y) \quad \text{ and } \quad h^u(z) = W^u_V(z)\cap W^s_V(y).
\end{equation}

\begin{remark}
    The holonomies $h^s$ and $h^u$ are H\"older-continuous homeomorphisms.
\end{remark}

\subsection{The Hopf Argument}

\begin{theorem}[Anosov, \cite{ansov1969geodesic}]
    Let $M$ be a smooth closed Riemmanian manifold. Any $\mathcal{C}^{1+\alpha}$ Anosov diffeomorphism $f: M\rightarrow M$ that preserves a smooth measure is ergodic.
\end{theorem}

\begin{proof}
    The proof is based on the Hopf argument \cite{MR1464}. Fix a continuous function $\phi:M\rightarrow\mathbb{R}$. We consider its forward and backward Birkhoff averages
    \begin{equation}
        \phi^+(x) = \frac 1N \sum_{n=0}^{N-1} \phi(f^n(x)) \quad \text{ and } \quad \phi^-(x) = \frac 1N \sum_{n=0}^{N-1} \phi(f^{-n}(x)),
    \end{equation}
    which are well defined and coincide on a full measure set $M'\subset M$.\newline
    Our goal is to prove that they are a.e. constant.\newline
    We start by noticing that, if $x'\in W^s(s)$, since $d(f^n(x'),f^n(x))\rightarrow0$, one has that $\phi^+(x')=\phi^+(x)$, whenever one of the two is defined. In conclusion, $\phi^+$ is constant along the leaves of $\mathcal{F}^s$, and $\phi^-$ is constant along the ones of $\mathcal{F}^u$, provided they are defined for at least one point of those leaves. By connectedness of $M$, is enough to show that $\phi^{\pm}$ are a.e. constant on any open set $U$ with product structure.\newline
    One could try to argue as follows. Take $x,y\in U\cap M'$ and consider $z=W^s_U(x)\cap W^u_U(y)$. If we had that $z\in M'$, we could conclude that
    \begin{equation}
        \phi^-(x)=\phi^+(x)=\phi^+(z)=\phi^-(z)=\phi^-(y)=\phi^+(y),
    \end{equation}
    as one wished. But there is no gurantee, a priori, that $z$ actually belongs to $M'$, and neither that this happens with high probability, even though $M'$ has full measure in $U$. This issue is related with the fact the the invariant foliations are only H\"older regular, and is not enogh to allow disintegration along them. The next theorem fixes the problem.

    \begin{theorem}
        Suppose $f:M\rightarrow M$ is a $\mathcal{C}^{1+\alpha}$ Anosov diffeomorphism.
        Then the invariant foliations $\mathcal{F}^s$ and $\mathcal{F}^u$ are absolutely continuous.
    \end{theorem}
    
        In our particular situation this translates to:
        \begin{itemize}
            \item the holonomy $h^u$ is an absolutely continuous maps for any $x,y\in U$;
            \item if $N\subset M$ has $\text{vol}_M(N)=0$, then $\text{vol}_{W^s(x)}(N\cap W^s_U(x))=0 $ for a.e. $x\in U$.
        \end{itemize}
    
    The second part of the statement ensures that, for a full measure set $U'\subset U$, every $x\in U'$, $M'$ has full measure in $W^s_U(x).$ We take $x,y\in U'$ and show that $\phi^+(x)=\phi^{+}(y)$. Given a set $B$, we call $B_x=B\cap W^s_U(x)$ its $W^s(x)$ section in $U$. We then have that
    \begin{equation}
        h^u\big(M'_x\big) \cap M'_y \neq \emptyset
    \end{equation}
    since $M'_x$ and $M'_y$ have  full measure in $W^s_U(x)$ and $W^s_U(y)$ (since $x,y\in U'$) and $h^s$ is absolutely continuous. Then we can find $z\in h^u\big(M'_x\big) \cap M'_y$ and call $w=(h^u)^{-1}(z)$. Finally, we have that $w\in W^s(x)$, $w\in M'$, $z\in W^u(w)$, $z\in M'$ and $y\in W^s(z)$, implying that
    \begin{equation}
        \phi^+(x)=\phi^+(w)=\phi^-(w)=\phi^-(z)=\phi^+(z)=\phi^+(y).
    \end{equation}

\end{proof}
    
\begin{corollary}
    All Anosov diffeomorphisms $f:M\rightarrow M$ are stably ergodic in the $\mathcal{C}^{1+\alpha}$ topology.
\end{corollary}
    
\begin{corollary}
    Any toral automorphism $A:\mathbb{T}^N\rightarrow\mathbb{T}^N$ with no unitary eigenvalue is stably ergodic in the $\mathcal{C}^{1+\alpha}$ topology.
\end{corollary}

\subsection{Partial Hyperbolicity and Accessibility}

Let $M$ be a closed Riemmanian manifold.

\begin{definition}
    We say that a diffeomorphism $f:M\rightarrow M$ is partially hyperbolic if there exists a $Df$-invariant splitting
\begin{equation*}
    TM = E_f^s \oplus E_f^c\oplus E_f^u,
\end{equation*}
where $E_f^s$ and $E_f^u$ are non trivial bundles, and for $*=s,c,u$ there exist continuous functions $\mu_*,\lambda_*:M\rightarrow\mathbb{R}_+$ such that
\begin{equation*}
    \mu_*(x) < \frac{\|D_xf(v^*)\|}{\|v\|} < \lambda_*(x), \quad \text {for every } v^*\in E_f^*(x)\setminus\{0\},
\end{equation*}
and
\begin{equation*}
    \lambda_s(x) < \min\{1,\mu_c(x)\}, \quad \quad \max\{1,\lambda_c(x)\}<\mu_u(x).
\end{equation*}
\end{definition}

In other words, $E^s_f$ and $E^u_f$ are hyperbolic bundles.In fact, we have that\newline $\lambda=\max_{x\in M}\max\{\lambda_s(x),1/\mu_u(x)\}<1$ is a constant of uniform hyperbolicity.\newline
The vectors of $E^c_f$ are contracted or expanded less than the ones of $E^s_f$ and $E^u_f$, but, in general the control, is only pointwise (i.e. not by $\lambda$).

\begin{example}
    Any toral automorphism $A:\mathbb{T}^N\rightarrow\mathbb{T}^N$ with unitary eigenvalue is a partially hyperbolic diffeomorphism.
\end{example}

\begin{theorem}
    The set of partially hyperbolic diffeomorphisms is open in $\mathcal{C}^1(M)$.
\end{theorem}


We always have that $E^s_f$ and $E^u_f$ are uniquely integrable to invariant H\"older continuous foliations $\mathcal{F}_f^s$ and $\mathcal{F}_f^u$, whose leaves are as smooth as $f$. These foliations are not transverse, because of the lack of dimensions caused by the existence of the center.\newline
The distributions $E^c_f$, $E_f^{cs}=E^s_f\oplus E^c_f$ and $E_f^{cu}=E^u_f\oplus E^c_f$ are not always uniquely integrable. If they are, we say that $f$ is dynamically coherent and we denote by $\mathcal{F}_f^c$, $\mathcal{F}_f^{cs}$ and $\mathcal{F}_f^{cu}$ the corresponding tangent foliations. Dynamical coherence is a $\mathcal{C}^1$-stable property, under some additional hypotesis on the uniformity of the partial hyperbolicity.\newline
We denote by $W_f^*(x)$ the leaf of $\mathcal{F}_f^*$ containing $x$, if existing, for $*=s,c,u,cs,cu$.\newline

Unlike the uniformly hyperbolic case, not every volume-preserving partially hyperbolic diffeomorphism is ergodic, since the center behavior may produce invariance.

\begin{example}\label{non ergodic PH}
    The map $T:\mathbb{T}^3\rightarrow\mathbb{T}^3$, defined by $T=A\times\text{id}$, where $A:\mathbb{T}^2\rightarrow\mathbb{T}^2$ is the Arnold's cat map and $\text{id}:\mathbb{T}^1\rightarrow\mathbb{T}^1$ is the identity, is volume preserving but not ergodic.
\end{example}

The first example of an ergodic partially hyperbolic diffeomorphism which is not uniformly hyperbolic was given by Grayson, Pugh and Shub, who exploited an additional property of such system, called accessibility.

\begin{theorem}[Grayson-Pugh-Shub, \cite{MR1298715}]
    The time-one map of the geodesic flow on the unit tangent bundle of a surface of constant negative curvature is ergodic, and stably
ergodic in the $\mathcal{C}_{\text{vol}}^2$-topology.
\end{theorem}

We say that a curve $\gamma:I\rightarrow M$ is an $su$-path, if it is piecewise contained in the leaves of $\mathcal{F}^s$ and $\mathcal{F}^u$. A point $y$ is said to be $su$-accessible from $x$, if there exists a $su$-path joining the two and we define the accessibility class of $x$ as $AC(x)=\{y\in M: y \text{ is accessible from } x\}$. We say that $f$ is accessible if $AC(x)=M$ for every $x\in M$.\newline
Euristically, accessibility should imply ergodicty, by allowing one to replicate the Hopf argument. However, to make this work, high regularity of the invariant foliations or some additional geometrical assumptions are required, given the potentially more complex geometry of the $su$-paths taken into consideration.

Pugh and Shub conjectured that ergodicity is generic among partially hyperbolic diffeomorphisms, proposing accessibility as the cornerstone of this phenomenon.

\begin{conjecture}[Pugh-Shub, \cite{MR1449765}]
    An open and dense subset of the class of the $\mathcal{C}^r$ volume-preserving partially hyperbolic diffeomorphisms are accessible. 
\end{conjecture}

\begin{conjecture}[Pugh-Shub, \cite{MR1449765}]
    Every $\mathcal{C}^2$ accessible volume-preserving partially hyperbolic diffeomorphism is ergodic.
\end{conjecture}

The first conjecture has been proved only in the $\mathcal{C}^1$ topology by Dolgopyatt and Wilkinson \cite{MR2039999} (or when the center is one dimensional).\newline
The second one holds, under a mild additiona assumption called center bunching, and a weakened version of the accessibility property, called essential accessibility.

\begin{theorem}[Burns, Wilkinson, \cite{burns2010ergodicity}]
    Let $f:M\rightarrow M$ be a $\mathcal{C}^2$ volume-preserving, partially hyperbolic and center bunched diffeomorphism on a closed Riemmanian manifold $M$. If $f$ is essentially accessible, then $f$ is ergodic.
\end{theorem}

Here, a set $E\subset M$ is said $su$-saturated if, for every $x\in E$, we have that $AC(x)\subset E$ and $f$ is said essentially accessible if every measuable $su$-saturated $E\subset M$ has either full or zero volume.\newline

The missmatch between the regularity assumptions of the just mentioned results leaves Pugh and Shub's program widely open.\newline

\begin{example}
    A linear automorphism of the torus $A:\mathbb{T}^N\rightarrow\mathbb{T}^N$, with non-trivial center, is not accessible, since the invariant foliations are linear, but it is essentially accessible.
\end{example}

Accessibility is a $\mathcal{C}^1$-stable property for center of dimension one \cite{DIDIER_2003} and two \cite{MR4142463}, and it is conjectured to be in any dimension. On the oether hand, essential accessibility might not be stable, making particularly hard to establish stable ergodicity for essentially accessible, but not accessible, systems, such as Toral automorphisms.\newline
This was achieved by Rodriguez Hertz \cite{Hertz2005} for a certain class of such maps.

\begin{theorem}[Rodriguez Hertz, \cite{Hertz2005}]
    Any pseudo-Anosov linear automorphism $A:\mathbb{T}^N\rightarrow\mathbb{T}^N$ with $\dim(E^c)=2$ is stably ergodic in the $\mathcal{C}^{22}_{\text{vol}}$ topology.
\end{theorem}

The pseudo-Anosov assumption is an algebraic condition imposed on the carachterisitic polynomial of the toral automorphism.\newline
The proof is based on the fact that essential accessibility is actually stable for these specific maps. The first step to establish essential accessibility is to investigate the structure of the accessibility classes.

\begin{conjecture}[Rodriguez-Hertz, Rodriguez-Hertz, Ures, \cite{zbMATH05287021}]
    The accessibility classes are topological manifolds that vary semi-continuously, as well as their dimensions. Under center bunching, they are indeed smooth manifolds.
\end{conjecture}

The conjecture holds true in the specific setting we are concerning ourselves with. More generally, a positive answer was given (except for the semi-continuity of the dimensions) for dynamical coherence diffeomorphisms with two-dimensional center \cite{rodriguez2017structure}.

The next step of the proof is to show that all accessibility classes must have the same dimension, for any $\mathcal{C}^1$-small perturbation of $A$. Then, one shows that codimension one classes cannot occur. Finally, either one has the accessibility property, or it is possible to $\mathcal{C}^1$-conjugate the accessibility classes to the ones of $A$, by exploiting KAM theory. This transfers the essential accessibility of $A$ to the perturbed map.\newline

We highlight that KAM theory usually provides the persistence of some invariant regions under perturbations  (e.g. invariant tori for integrable systems), which, in several contexts, represents a robust obstruction to ergodicity. Rodriguez-Hertz remarkably used it, in the setting of partially hyperbolic diffeomorphisms, to show that the persistence of some structure may instead lead to ergodicity in other settings.\newline

Finally, we remark that, in general, a diffeomorphism might simultaneously have accessibility classes of different dimensions.

\begin{example}
    Consider the map $T:\mathbb{T}^3\rightarrow\mathbb{T}^3$ from example \ref{non ergodic PH}. We have that $AC(0)=\mathbb{T}^2\times\{0\}$. Let us fix $x\in\mathbb{T}^3$ with $x\neq0$. One can show that a generic perturbation $T+\Phi$, with $\Phi$ supported away from $\mathbb{T}^2\times\{0\}$, breaks the joint integrability of $\mathcal{F}^s$ and $\mathcal{F}^u$ near $x$. This ensures that $AC(x)$ is open, since the center is one dimensional. On the other hand, the perturbation does not affect the orbit of any point in $\mathbb{T}^2\times\{0\}$, hence preserving the invariant foliations of those points. In conclusion we still have that $AC(0)=\mathbb{T}^2\times\{0\}$ has dimension 2. 
\end{example}






\section{Spectral Theory of Schr\"odinger Operators}

\subsection{Discrete Schr\"odinger Operators}

Let $\mathcal{G}=(\mathcal{V},\mathcal{E})$ be a graph, either finite or infinite, endowed with shortest path distance $\text{d}_{\mathcal{G}}$. We fix $\lambda>0$ and consider a potential $V:\mathcal{V}\rightarrow[-\lambda,\lambda]$. We define the discrete Schr\"odinger Operator $H_{\mathcal{G},V}:l^2(\mathcal{V})\rightarrow l^2(\mathcal{V})$ by
\begin{equation*}
    \left(H_{\mathcal{G},V}\phi\right)(v) =  \sum_{v'\sim v} [\phi(v')-\phi(v)] + V(v)\phi(v),
\end{equation*}
where $\phi\in l^2(\mathcal{V})$ and $v\sim v'$ means that $v$ and $v'$ are neighbours.\newline
When the potential is bounded by $\lambda$ and the degree of every vertex of $\mathcal{G}$ is bounded by $k$, $H_{\mathcal{G},V}$ is a bounded self-adjoint operator and $\sigma(H_{\mathcal{G},V})\subset[-\lambda-2k,\lambda+2k]$.\newline

A discrete Schr\"odinger operator can be written as $H=\Delta + V$, where $\Delta$ is the combinatorial Laplacian and $V$ is a diagonal operator, corresponding to the potential. Considering a Riemmanian manifold $(M,g)$, one can define the continuous Schr\"odinger Operator $H:L^2(M,g)\rightarrow L^2(M,g)$ again by $H=\Delta+V$, where, this time, $\Delta$ is the Beltrami-Laplace operator and $V$ is a multiplication operator that corresponds to a potential $V:M\rightarrow\mathbb{R}$. The continuous Schr\"odinger Operator is typically a non-bounded operator.

\begin{remark}
    A Schr\"odinger Operator $H$ with identically zero potential is precisely the combinatorial Laplacian $\Delta$ on the graph $\mathcal{G}$.
\end{remark}

When $\mathcal{G}$ is the Cayley graph of a group $G$, the potential $V$ may be dynamically defined. More specifically, let $K$ be a set, $\Phi:K\rightarrow[-\lambda,\lambda]$ be a function and let $(f^{\bold{g}})_{\bold{g}\in G}$ a $G$-action on $K$. Often the space $K$ is endowed with a topology or a measure and $\Phi$ and $(f^{\bold{g})}$ have the compatible regularity. For every $x\in K$, we define the Schr\"odinger operator $H_x$, corresponding to the potential $V_x$, obtained by sampling $\Phi$ along the $G$-orbit of $x$, i.e.
\begin{equation}
    V_x(\bold g) = \Phi(f^{\bold g}(x)).
\end{equation}
The main examples are given by $G=\mathbb{Z}$ or $G=\mathbb{Z}^d$.\newline
If $K$ is endowed with a measure $\mu$ and the action of $G$ is ergodic, we obtain a family of Ergodic Schr\"odinger operators $(H_x)_{x\in K}$. Operators of an ergodic family tend to show common regular behaviours for $\mu$-a.e. $x\in K$, as the next theorem showcases.

\begin{theorem}[\cite{pastur1980spectral}]
    Let $(H_x)_{x\in K}$ be a family of Ergodic Schr\"odinger Operators over the Ergodic system $(K,\mu,T)$. Then exists a set $\Sigma\subset\mathbb{R}$ such that
    \begin{equation}
        \sigma(H_x)=\Sigma \quad\quad \text{ for }\mu\text{-a.e. }x\in K.
    \end{equation}
\end{theorem}

Moreover, Ergodic Schr\"odinger Operators provide a common framework for several different models arising naturally and separately from physics.

\begin{example}[Random potential] Let $P$ be a probability distribution on $[-\lambda,\lambda]$. Now consider a probability space $(\Sigma,\nu)$ and a random variable $X:\Sigma\rightarrow[-\lambda,\lambda]$ with law $P$. We set $\Omega=\Sigma^{\mathbb{Z}}$, and $\mu=P^{\otimes\mathbb{Z}}$ and the shift map $T:\Omega\rightarrow \Omega$ that gives an ergodic $\mathbb{Z}$-action on $\Omega$. Considering the sampling function $\Phi(\omega)=X(\omega_0)$, for $\omega\in\Omega$, we define the operators $H_{\omega}:\ell^2(\mathbb{Z})\rightarrow\ell^2(\mathbb{Z})$, corresponding to the potential
\begin{equation}
    V_{\omega}(n) = \Phi(T^n\omega)=X(\omega_n), \quad \text{ for } n\in\mathbb{Z}.
\end{equation}
The values of the potential are indipendently randomly distributed according to $P$. This is the celebrated Anderson Model \cite{anderson1958absence}, that arises in the study of material conductivity.
\end{example}

\begin{example}[Quasi-periodic potential]
    We consider $\mathbb{S}^1=[0,1]/_\sim$ and the irrational rotation $f:\mathbb{S}^1\rightarrow \mathbb{S}^1$ given by $f(x)=x+\alpha$, for $\alpha\in\mathbb{R}\setminus\mathbb{Q}$, which provides an ergodic $\mathbb{Z}$-action on $\mathbb{S}^1$. Finally, we fix a continuous $\Phi:\mathbb{S}^1\rightarrow[-\lambda,\lambda]$ as sampling function. Then, for $x\in\mathbb{S}^1$, we obtain the operators $H_x:\ell^2(\mathbb{Z})\rightarrow\ell^2(\mathbb{Z})$, corresponding to the potential
    \begin{equation}
        V_x(n) = \Phi(T^nx) = \Phi(x + n\alpha).
    \end{equation}
    They take the name of quasi-periodic Schr\"odinger operators, which comes from the study of quasi-crystals. With the specific choice of $\Phi(x)=2\lambda\cos(2\pi x))$, one recovers the Almost Mathieu Operator.
\end{example} 

Given a graph $\mathcal{G}=(\mathcal{V},\mathcal{E})$, a subset of vertices $Y\subset\mathcal{V}$ induces a subgraph $\mathcal{G}_Y\subset\mathcal{G}$. Its vertices are the elements of $Y$ and two of them are connected by an edge if they are connected by an edge of $\mathcal{G}$. Given a Shr\"odinger operator $H_{\mathcal{G},V}$ on $\mathcal{G}$ with potential $V$, $Y$ induces also a truncated Schr\"odinger operator $H_{\mathcal{G}_Y,V|_Y}$ on $\mathcal{G}_Y$ with potential $V|_{Y}$. Finite truncation of infinite graphs play a fundamental role in the interplay between the theory of infinite and finite dimensional Schr\"odinger operators. A prime example is given by the box Schr\"odinger operators, defined on the subgraph of $\mathbb{Z}^d$ induced by the set $[N]^d=\{0,\dots,N-1\}^d$.\newline

Finally we mention the closely related family of finite range Schr\"odinger operators. Given a graph $\mathcal{G}=(\mathcal{V},\mathcal{E})$, a potential $V:\mathcal{V}\rightarrow[-\lambda,\lambda]$ and a range $r\in\mathbb{Z}^+$, the corresponding $r$-range Schr\"odinger operator $H_{\mathcal{G},V,r}:\ell^2(\mathcal{V})\rightarrow\ell^2(\mathcal{V})$ is defined by

\begin{equation}
	\left(H_{\mathcal{G},V,r}\phi\right)(v) =  \sum_{\text{d}_{\mathcal{G}}(v',v)\leq r} [\phi(v')-\phi(v)] + V(v)\phi(v).
\end{equation}

These can be seen as regular Schr\"odinger operators, provided that the underlying graph $\mathcal{G}$ is replaced with the graph $\mathcal{G}_r$, which has the same vertices as $\mathcal{G}$ and where two vertices are connected by an edge if their distance is less than $r$ in $\mathcal{G}$.

\subsection{Quantum Dynamics and Spectral Measures}

Given an Hilbert space $\mathcal{H}$ and a bounded self-adjoint operator $H:\mathcal{H}\rightarrow\mathcal{H}$ the asociated Schr\"odinger equation is given by
\begin{equation}
	\frac d{dt} \phi(t) = -iH\phi(t), \quad \text{ for }t\in\mathbb{R},
\end{equation}
where $\phi$ is a function from $\mathbb{R}$ to $\mathcal{H}$, and goes by the name of wave function. We are interested in long-time qualitative behaviour of $\phi(t)$, in particular in determining whether its bulk remains trapped in a compact region (dynamical localization) or some escape of mass occurs (dynamical delocalization). The solutions to the Schr\"odinger Equation are given by 
\begin{equation}
    \phi(t)=e^{-itH}\phi(0)
\end{equation}
and their asymptotical properties depend on the spectral type of $H$.

\begin{theorem}
	Let $\mathcal{H}$ be an Hilbert space and $H:\mathcal{H}\rightarrow\mathcal{H}$ be a bounded self-adjoint linear operator. For any $\phi\in\mathcal{H}$ there exists a unique measure $\sigma_{\phi}$, supported on $\sigma(H)$ such that, for every continuous $f:\sigma(H)\rightarrow\mathbb{R}$,
	\begin{equation}
	\langle \phi, f(H)\phi \rangle = \int_{\sigma(H)}f(E)d\sigma_{\phi}(E).
	\end{equation}
The measure $\mu_{\phi}$ is said to be the spectral measure of $H$ corresponding to $\phi$.
\end{theorem}

We recall that any Borel measure $\sigma$ on $\mathbb{R}$ admits a unique Lebesgue decomposition
\begin{equation}
	\sigma=\sigma^{pp}+\sigma^{sc}+\sigma^{ac},
\end{equation}
where $\sigma^{pp}$ is supported on a countable set, $\sigma^{sc}$ is singular continuous and $\sigma^{ac}$ is absolutely continuous. We also define $\sigma^{s}=\sigma^{pp}+\sigma^{sc}$ and $\sigma^c=\sigma^{sc}+\sigma^{ac}$.\newline

We say that $H$ has pure point spectrum if $\sigma_{\phi}=\sigma_{\phi,pp}$ for every $\phi\in\mathcal{H}$, purely singular continuous spectrum if $\sigma_{\phi}=\sigma_{\phi,ac}$ for every $\phi\in\mathcal{H}$ and purely absolutely continuous spectrum if $\sigma_{\phi}=\sigma_{\phi,ac}$ for every $\phi\in\mathcal{H}$. If $H$ is pure point we have spectral localization, if $H$ is absolutely continuous we have spectral delocalization.\newline

For Schr\"odinger Operators on $\ell^2(\mathbb{Z})$, the RAGE theorem connects the notion of spectral localization (or delocalization) to phenomena of dynamical localization (or delocalization) of the solutions of the Schr\"odinger Equation. In the following, $\delta_n$ denotes the characteristic function of $\{n\}$.

\begin{theorem}[RAGE Theorem, \cite{ruelle1969remark}, \cite{amrein1974characterization}, \cite{enss1978asymptotic}]
    Suppose the operator $A:\ell^2(\mathbb{Z})\rightarrow\ell^2(\mathbb{Z})$ to be self-adjoint,
$\phi\in\ell^2(\mathbb{Z})$, and $\phi(t)$ to be a solution of the Schr\"odinger equation with initial data $\phi$.
\begin{enumerate}
    \item We have that $\sigma_{\phi}=\sigma_{\phi,pp}$ if and only if for every $\varepsilon>0$, there exists $N\in\mathbb{Z}^+$ such that
    \begin{equation}
        \sum_{|n|> N}|\langle\delta_n,\phi(t)\rangle|^2 < \varepsilon \quad \text{ for every }t\in\mathbb{R}
    \end{equation}.
    \item We have that $\sigma_{\phi}=\sigma_{\phi,c}$ if and only if for every $N\in\mathbb{Z}^+$,
    \begin{equation}
        \lim_{T\rightarrow\infty}\frac 1{2T}\int_{-T}^T \sum_{|n|\leq N}|\langle\delta_n,\phi(t)\rangle|^2dt = 0
    \end{equation}.
    \item If $\sigma_{\phi}=\sigma_{\phi,ac}$, then for every $N\in\mathbb{Z}^+$,
    \begin{equation}
        \lim_{|t|\rightarrow\infty} \sum_{|n|\leq N}|\langle\delta_n,\phi(t)\rangle|^2dt = 0
    \end{equation}.
\end{enumerate}
\end{theorem}

\begin{example}
	 The combinatorial Laplacian $\Delta$ on $\mathbb{Z}^d$ has purely absolutely continuous spectrum.
\end{example}

\begin{theorem}[Anderson Localization, \cite{goldsheid1977random}, \cite{kunz1980spectre}]
	Let $\nu$ be an absolutely continuous probability measure on $[-\lambda,\lambda]$. Let $H:\ell^2(\mathbb{Z})\rightarrow\ell^2(\mathbb{Z})$ be a Schr\"odinger Operator on $\mathbb{Z}$ where the values of the potential are indipendently sampled from $[-\lambda,\lambda]$ according to $\nu$. Then $H$ is almost surely pure point.
\end{theorem}

The last two examples suggest that a material with a perfect crystalline structure is a conductor, and the presence of random impurities makes it an insulator. The first formulation of this phenomenon \cite{anderson1958absence} was worth the Nobel prize to Anderson in 1977.

\subsection{Eigenfunctions Delocalization and Quantum Ergodicity}

In some settings, like for continuous Schr\"odinger Operators on compact manifolds or discrete Schr\"odinger Operators on finite graphs, an escape of mass is trivially impossible, making a purely qualitative analysis of the Schr\"odinger Equation solutions unsatisfactory. At the same time, in such scenarios, the existence of a basis of eigenfunctions is often guaranteed. By studying the concetration or diffusion of such eigenfunctions, a finer analysis of the long-time behaviour of the wave function is possible.\newline

This is particularly well showcased in the case of manifolds. Let $(M,g)$ be a compact Riemannian manifold and call $\Delta$ the Laplace-Beltrami operator on $M$. Let $T^1M$ be the unit tangent bundle of $M$, endowed with the Liouville Measure, and let $g_t:T^1M\rightarrow T^1M$ be the geodesic on $M$. We also recall that $\{g_t\}$ is ergodic for all manifolds with negative curvature.

\begin{theorem}[Quantum Ergodic Theorem, \cite{alexander1974shnirel}, \cite{colin1985ergodicite}, \cite{zelditch1987uniform}]
Let $M$ be a closed Riemannian manifold for which the geodesic flow $\{g_t\}$ is ergodic. Let $(\phi_n)_n$ be an orthonormal basis of $L^2(M,g)$ of eigenfunctions of the Laplacian $\Delta$, corresponding to weakly increasing eigenvalues $(\lambda_n)_n$. Then, for any smooth $a\in\mathcal{C}^{\infty}(M)$,
\begin{equation}
\lim_{N\rightarrow\infty} \frac 1N \sum_{n=1}^N \left| \int_M a(x)|\phi_{n}(x)|^2dx - \int_M a(x)dx \right|^2 =0.
\end{equation}
\end{theorem}

This property of the family of the eigenfunctions goes by the name of quantum ergodicity, since, on average, the eigenfunctions tend to spread (in a weak sense) towards a uniform distribution.\newline

We now turn our attention to finite graph. Here, a tree-like structure mimics the hypotesis of negative curvature for Riemmanian manifolds. We begin with a result of weak delocalization of all eigenvectors. More precisely, we say that a graph $\mathcal{G}$ is $(q+1)$-regular if every vertex has degree $(q+1)$ and that a $(q+1)$-regular graph $\mathcal{G}$, with $n$ vertices, has $(c,\delta)$-few short loops if, for any $k< c\log(n)$ and any two vertices $v,w$,
\begin{equation}
	\#\{\text{paths of length $k$ in $\mathcal{G}$ from $v$ to $w$}\} \leq q^{k(\frac{1-\delta}2)}.
\end{equation}

\begin{theorem}[\cite{brooks2013non}]
	Let $\mathcal{G}$ be a $(q+1)$-regular graph with $(c,\delta)$-few short loops for some $c,\delta>0$. Fix $\varepsilon>0$. Then, for any unitary eigenfunction $\phi$ of the discrete laplacian on $\mathcal{G}$ and for any set of vertices $A$,
	\begin{equation}
		\sum_{v\in A} |\phi(v)|^2>\varepsilon \quad \text{ implies that } \quad |A| \geq |\mathcal{G}|^{\alpha},
	\end{equation}
 for $\alpha = \alpha(\varepsilon,c,\delta)>0$.
\end{theorem}

We conclude mentioning a result of quantum ergodicity in the same setting. Given a $(q+1)$-regular graph $\mathcal{G}$ and calling $A_{\mathcal{G}}$ its adjacency matrix, we say that $\mathcal{G}$ is an expander graph if  there exists $\beta>0$ such that $\sigma\big(\frac{A_{\mathcal{G}}}{q+1}\big)$ is contained in $[-1+\beta, 1 -\beta]\cup \{1\}$. The $(q+1)$-regular tree is the only infinite $(q+1)$-regular graph with no loops. Anantharaman and Le Masson \cite{MR3322309} proved that, any sequence of $(q+1)$-regular expander graphs converging in the Benjamini-Schramm sense to the $(q+1)$-regular tree, has quantum ergodicity, when we restric our attention to the eigenfunctions of the combinatorial Laplacian corresponding to the tempered part of the spectrum. 

\subsection{The Integrated Density of States}

In this and the following section we deal with an arbitrary family of Ergodic Schr\"odinger Operators $(H_x)_{x\in K}$ on $\ell^2(\mathbb{Z})$, over a dynamical system $(K,\mu,T)$. The content of this section is mainly due to \cite{benderskiui1970spectrum}, \cite{pastur1980spectral}, \cite{avron1983almost} and is nicely exposed in \cite{damanik2022one}. For $n\in\mathbb{Z}$, we denote by $\sigma_{x,n}$ the spectral measure of $H_x$ corresponding to $\delta_n\in\ell^2(\mathbb{Z})$.

\begin{definition}
    The Density of States Measure (DOSM) is the measure $\sigma$ on $\mathbb{R}$, defined by
    \begin{equation}
        \sigma = \int_K \sigma_{x,n} d\mu(x), \quad \text{for any fixed }n\in\mathbb{Z}.
    \end{equation}
    The Integrated Density of States (IDS) is the function $\mathcal{N}:\mathbb{R}\rightarrow[0,1]$, defined by
    \begin{equation}
        \mathcal{N}(E)=\sigma((-\infty,E]).
    \end{equation}
\end{definition}

As an instance of Birkhoff Ergodic Theorem, we have the following equivalent definition of the Density of States Measure.

\begin{lemma}\label{lemma avrg over sites}
    The Density of States Measure is the measure $\sigma$ on $\mathbb{R}$, defined by
    \begin{equation}\label{average over all sites}
        \sigma = \lim_{N\rightarrow+\infty}\frac 1N \sum_{n=0}^{N-1} \sigma_{x,n}, \quad \text{for } \mu\text{-a.e. }x\in{K},
    \end{equation}
    where the limit is intended in the weak-* topology.
\end{lemma}

\begin{proof}
We have
	\begin{equation}
	\frac 1N \sum_{n=0}^{N-1} \sigma_{x,n} = \frac 1N \sum_{n=0}^{N-1} \sigma_{T^n(x),0},
	\end{equation}
	and Birkhoff Ergodic Theorem let us conclude for a.e. $x$.
\end{proof}

For $N\in\mathbb{Z}^+$, we denote by $H_{x,N}:\ell^2([N])\rightarrow \ell^2([N])$ the operator obtained by truncating $H_x$ on $[N]=\{0,\dots,N-1\}$. We let $E_{1,N,x}\leq \dots \leq E_{N,N,x}$ be its eigenvalues. We have one more alternative definition of the DOSM and the IDS. Here, $\delta_x$ is the Dirac measure concentrated at $x$. We distinguish between the two meanings of our notation $\delta$, depending on whether it stands for a measure or an element of $\ell^2(\mathbb{Z})$.

\begin{lemma}
    The Density of States Measure coincides with
    \begin{equation}\label{average of EV}
        \sigma = \lim_{N\rightarrow+\infty}\frac 1N \sum_{i=0}^{N-1} \delta_{E_{i,N,x}}, \quad \text{for } \mu\text{-a.e. }x\in{K},
    \end{equation}
    where the limit is intended in the weak-* topology.\newline
    The Integrated Density of States coincides with
    \begin{equation}
        \mathcal{N}(E)= \lim_{N\rightarrow+\infty} \frac{\#\{i:E_{i,N,x}\leq E)\}}N, \quad \text{for } \mu\text{-a.e. }x\in{K}.
    \end{equation}
\end{lemma}

\begin{proof}
    We only need to show the compatibility between expression \ref{average over all sites} and \ref{average of EV}.\newline
    For $n\in[N]$, we call by $\sigma_{x,n,N}$ the spectral measure of $H_{x,N}$ with respect to $\delta_n\in\ell^2([N])$. The two measures in question coincide for the truncated opeartor $H_{x,N}$, i.e.
    \begin{equation}\label{truncated IDS}
        \frac 1N\sum_{n=0}^{N-1} \sigma_{x,n,N} = \frac 1N\sum_{i=0}^{N-1} \delta_{E_{i,N,x}}.
    \end{equation}
    In fact, testing both against a bounded measurable $g:\mathbb{R}\rightarrow\mathbb{R}$, one gets $\frac 1N\text{tr}(g(H_{x,N})$. We denote by $\sigma_{x,N}$ the measure defined by expression \ref{truncated IDS}. We want to conclude by the moments method. Given a Borel measure $\nu$ on $\mathbb{R}$, we denote its $k$-th moment by $m_k(\nu)=\int x^kd\nu(x)$. If $\text{d}(n,\partial[N])>k$, then
\begin{equation}
	m_k(\sigma_{x,n,N}) = \langle \delta_n, H_{x,N}^k\delta_n \rangle = \langle \delta_n, H_{x}^k\delta_n \rangle = m_k(\sigma_{x,n}),
\end{equation}
since $H_{x,N}^k$ and $H_{x,N}^k$ have range $k$ and do not see any boundary effect. Averaging over $n\in[N]$, we get
\begin{equation}
	\left|m_k\left( \frac 1N\sum_{n=0}^{N-1} \sigma_{x,n,N} \right) - m_k\left( \frac 1N\sum_{n=0}^{N-1} \sigma_{x,n} \right)\right| \leq C\cdot  \frac{2k}N.
\end{equation}
The cocnclusion follows since, by lemma \ref{lemma avrg over sites}, for a.e. $x$,
\begin{equation}
	\lim_{N\rightarrow +\infty} m_k\left( \frac 1N\sum_{n=0}^{N-1} \sigma_{x,n,N} \right) = m_k(\sigma).
\end{equation}
\end{proof}

\begin{remark}
In the non-ergodic setting the first definition of DOSM and IDS is not available. One could use the second one, however the limit is not guaranteed to exists. Bourgain and Klein \cite{bourgain2013bounds} proposed a Density of States Outer Measure for deterministic Schr\"odinger Operators on $\mathbb{Z}^d$, which is always well-defined. For finite graphs one can simply take expression \ref{truncated IDS} as the definition.
\end{remark}

\subsection{The Thouless Formula}
For $E\in\mathbb{C}$, we denote by $L(E)$ the largest Lyapunov exponent of the Schr\"odinger Cocycle $A_E:K\rightarrow SL(2,\mathbb{C})$, over $(K,\mu,T)$. $L(E)$ does not depend on $x$ a.e., beacuse of the ergodicity of $(K,\mu,T)$, and one can verify that $L(E)$ is sub-harmonic in $E$.

\begin{theorem}[Thouless Formula, \cite{thouless1972relation}, \cite{avron1983almost}]
	For every $E\in\mathbb{C}$,
	\begin{equation}\label{Thouless formula}
		L(E) = \int \log|E-E'| d\sigma(E').
	\end{equation}
\end{theorem}

\begin{proof}
    We present the proof from \cite{MR705040}. We first prove the formula for $E\in\mathbb{C}\setminus\mathbb{R}$. Then one can conlude for $E\in\mathbb{R}$ by sub-harmonicity of both sides. Fixing $x\in K$, one can check by induction that the entries $(A_E^N(x))_{ij}$ are monic polynomials in $E$ of degree $N+2-i-j$, for $N\geq2$ and $i,j\in\{1,2\}$. The proof is based on the fact that we can determine exactly their zeros. In fact, recalling that any formal solution $u$ to the eigenfunction equation $H_xu=Eu$ satisfies
    \begin{equation}
    \begin{pmatrix}
        u(N+1)\\
        u(N)\\
    \end{pmatrix}=
    A_E^N(x)
    \begin{pmatrix}
        u(1)\\
        u(0)\\
    \end{pmatrix},
\end{equation}
we see that, for example $(A_E^N(x))_{11}$ vanishes if and only if there exist a non trivial eigenfunction $u$, corresponding to $E$, with $u(0)=u(N+1)=0$. This is equivalent to $E$ being an eigenvalue of the truncated operator $H_{x,N}$. Since the spectrum of $H_{x,N}$ is simple, its $N$ eigenvalues are exactly the zeros of $(A_E^N(x))_{11}$. Thus
\begin{equation}
    (A_E^N(x))_{11} = \prod_{i=1}^N(E-E_{i,N,x}),
\end{equation}
and
\begin{align}
    \frac1N\log \left|(A_E^N(x))_{11}\right| &= \frac1N \sum_{i+1}^N \log|E-E_{i,N,x}|\\
    &= \int \log|E-E'| d\sigma_{x,N}(E')\\
    &\xrightarrow{N\rightarrow\infty} \int \log|E-E'| d\sigma(E').
\end{align}
One can prove the same exponential rate of growth for all entries $(A_E^N(x))_{ij}$ by comparing to suitable truncated operator (i.e. $H_{x,N-1}$, $H_{Tx,N-1}$ and $H_{Tx,N-2}$), concluding the proof.
\end{proof}

\begin{corollary}
The Lyapunov Exponent $L(E)$ is harmonic in $E$.
\end{corollary}

Since the Schr\"odinger Cocycle is $SL(2,\mathbb{C})$-valued, $L(E)$ must be non-negative. By the Thouless formula, this implies that $\sigma$ cannot be too concentared. This reasoning provides a uniform modulus of continuity for the IDS.

\begin{corollary}[\cite{MR705040}]
The IDS is $\log$-Holder continuous, meaning that there exists a constant $C>0$, such that, for every $x,y\in\mathbb{R}$, with $|x-y|<1/ 2$, we have that
\begin{equation}
	|\mathcal{N}(x) - \mathcal{N}(y)| < C \cdot \frac 1{\log(1/|x-y|)}.
\end{equation}
\end{corollary}

\chapter{Advances on Stable Ergodicity of Toral Automorphisms}\label{Chapter - Torus}

\begin{center}
\large{F. Argentieri, A. Ulliana}
\end{center}

\subsection*{Abstract}
    We prove that all ergodic automorphisms of the $N$-dimensional torus with two dimensional center are stably ergodic. This includes all ergodic automorphisms in dimension $N\leq 5$ or $N=7$.
    This generalizes a previous result of Rodriguez-Hertz, that required an additional algebraic condition on the carachteristic polynomial of the linear automorphism. The core of the proof is a minimality criterion.

\section{Introduction}

Any element $A$ of $SL(N,\mathbb{Z})$ induces a linear automorphism of the torus $\mathbb{T}^N=\mathbb{R}^N/\mathbb{Z}^N$, that we will denote again by $A$. These maps have a remarkable richness of dynamical features and for this reason they have often served as playground for the early study of many fundamental aspects of dynamical systems. Among several examples we find: entropy \cite{adler1967entropy}, Markov partitions \cite{adler1970similarity}, the Bernoulli property \cite{katzenlson1971ergodic}, and partial hyperbolicity \cite{lind1982dynamical}.\newline
The question about their ergodicity has long been answered, since 1932 in \cite{halmos1943automorphisms} by Halmos, but the study of their statistical properties is still active today \cite{le1999limit}, \cite{dedecker2012empirical}. By Fourier analysis, it is easy to see that $A$ is ergodic with respect to the Lebesgue measure on $\mathbb{T}^N$ if and only if none of the eigenvalues of $A$ is a root of unity.\newline
This paper deals with stable ergodicity of linear automorphisms of the torus, aiming to advance towards answering the following question, posed by Hirsch, Pugh and Shub in 1977 \cite{hirsch1977invariant}.

\begin{question}\label{question}
    Are all ergodic linear automorphism of $\mathbb{T}^N$ also stably ergodic?
\end{question}

The stability has to be intended in $\mathcal{C}_{\text{vol}}^k(\mathbb{T}^N)$ for some $k$, ideally $1+\alpha$, where $\mathcal{C}_{\text{vol}}^k(\mathbb{T}^N)$ denotes the space of $\mathcal{C}^k$ volume preserving diffeomorphisms of $\mathbb{T}^N$.\newline
While stable ergodicity in the $\mathcal{C}^1$ category would be highly desirable, so far there are no known examples of maps with such property and all known mechanisms to prove ergodicity require higher regularity, even in the simplest settings \cite{MR2736152}.\newline

The characterization of stable ergodicity of diffeomorphisms on compact manifolds has been a central problem in smooth dynamics (see \cite{MR4198639} for an historical account) and it is strictly linked to the degree of hyperbolicity of the said diffeomorphisms. In our setting, $A$ is partially hyperbolic, in fact, one has the $A$-invariant splitting
\begin{equation*}
    \mathbb{R}^N=E^s\oplus E^c \oplus E^u,
\end{equation*}
where $E^s$, $E^c$ and $E^u$ are the sums of the generalized eigenspaces of $A$ corresponding to the eigenvalues of modulus less than 1, exactly $1$ and more than 1, respectively. $E^s$ and $E^u$ are, respectively, uniformly contracted and expanded
by $A$, while vectors in $E^c$ are neither contracted or expanded as strongly as the
vectors in $E^s$ nor in $E^u$. We also remark that, for algebraic reasons, $\dim(E^c)$ must be even.\newline

The first example of stably ergodic diffeomorphisms was given by the class of Anosov diffeomorphisms (i.e. uniformly hyperbolic). Uniform hyperbolicity is itself an $\mathcal{C}^1$-open condition and it implies ergodicity for any $\mathcal{C}^{1+\alpha}$ conservative diffemorphism, as proved by Anosov \cite{ansov1969geodesic} in 1969, exploiting the the Hopf argument from \cite{MR1464}. When $E^c$ is trivial, $A$ is uniformly hyperbolic, settling question \ref{question} in the affirmative in this case.\newline

Partial hyperbolicity is a $\mathcal{C}^1$-open condition, but it does not imply ergodicity alone. In 1994 in \cite{MR1298715}, the first example of non uniformly hyperbolic stably ergodic diffeomorphism was given, exploiting its partial hyperbolicity and an additional property, called accessibility, which, euristically, allows to replicate the Hopf argument. Soon afterwards, stable ergodicity was conjectured to be prevalent ($\mathcal{C}^r$-dense) among partilly hyperbolic diffeomorphisms by Pugh and Shub \cite{MR1449765}, proposing accessibility as a main tool.\newline
Indeed, Burns and Wilkinson \cite{burns2010ergodicity} proved that essential accessibility (a weakening of accessibility) implies ergodicity of $\mathcal{C}^2$ conservative diffeomorphisms under mild additional assumptions (center bunching). Moreover, accessibility was proved to be $\mathcal{C}^1$-stable when the center is low dimensional \cite{MR4142463}, \cite{DIDIER_2003}, and stable accessibility to be $\mathcal{C}^1$-dense \cite{MR2039999}. However, because of the missmatch between the $\mathcal{C}^{1+\alpha}$ requirement for ergodicity and the only $\mathcal{C}^1$-prevalence of accessibility this was not enough to complete Pugh and Shub program, which remains widely open today. The best results up to date exploit a different approach, based on blenders \cite{MR4198639}.\newline

Out of completeness, we mention that there are examples of non partially hyperbolic diffeomorphisms that are stably ergodic \cite{MR2085722}, but every stably ergodic diffeomorphism must have a weak form of hyperbolicity. In fact, the existence of a dominated splitting is a necessary condition for stable ergodicity \cite{MR2018925}.\newline

Coming back to our setting, when $E^c$ is non trivial, $A$ is partially hyperbolic but not accessible, preventing us from exploiting the afore mentioned results.
At least, the ergodicity of $A$ ensures the essential accessibility property, which, though, is not necessarily a stable property, posing a serious challenge in establish stable ergodicity results.\newline

A major breakthrogh was achieved by Rodriguez-Hertz \cite{Hertz2005}, who proved that $A$ is actually stably essentially accessible in the $\mathcal{C}^{22}$-topology under the assumptions that $\dim(E^c)=2$ and of $A$ being pseudo-Anosov (see section \ref{linear algebra} for the definition). The latter is an algebraic condition of the carachteristic polynomial of $A$, which controls the action of $A$ on $\mathbb{Z}^N$. This class of linear automorphisms includes all non Anosov ones in dimension 4, but, on the other hand, is empty in odd dimension.\newline

Our result removes any algebraic constraint on $A$.

\begin{theorem}\label{stable ergodicity}
    Any ergodic linear automorphism $A$ of the torus $\mathbb{T}^N$ with $\dim(E^c)=2$ is stably ergodic in $\mathcal{C}^{22}_{\text{vol}}$.
\end{theorem}
We remark that in any dimension $N\geq 6$ there exist linear automorphisms of $\mathbb{T}^N$ that have $\dim(E^c)=2$ but are not pseudo-Anosov. This is easily seen, thanks to the fact that any integer coefficient polynomial is the carachteristic polynomial of a suitable element of $SL(N,\mathbb{Z})$.\newline

In dimension $7$ the ergodicity of $A$ implies that $\dim(E^c)=0$ or $\dim(E^c)=2$ (see lemma \ref{dimension 7 polynomial}), giving us the following corollary directly.

\begin{corollary}\label{dimension 7}
    Any ergodic linear automorphism $A$ of $\mathbb{T}^7$ is stably ergodic in $\mathcal{C}^{22}_{\text{vol}}$.
\end{corollary}

\begin{remark}
    The same conclusion holds in dimension $6$ and $9$, provided that no eigenvalue of $A$ is a Salem number.
\end{remark}

Finally, we want to comment about the remaining assumptions of our result.\newline
As in \cite{Hertz2005}, the high regularity of the perturbations is reuqired specifically to be able to exploit KAM theory (see \cite{MR538680} for an exposition) in the final step of the proof. There is no indication for this to be a necessary condition, but a new general strategy of proof would be required to reduce such regularity requirement.\newline
The condition that $\dim(E^c)=2$ is essential for understanding the structure of the accessibile components (called accessibility classes) of $f$ on $\mathbb{T}^N$. They are, in general, conjectured to be topological manifolds \cite{zbMATH05287021} (smooth manifolds under center bunching) and it is proved only for low dimensional center \cite{rodriguez2017structure}. This assumption is also used in the algebraic topological argument and to control the number of possible dimensions for the accessibility classes.

\subsection*{Acknowledgments}

The authors thank their mentor Artur Avila for drawing their attention to this topic, Federico Rodriguez Hertz for the valuable conversations, and Daniele Galli for the encouragements.
F.A. was supported by the SNSF mobility grant.


\section{Background and Organization}

\subsection{Partial Hyperbolicity}

We recall some general theory about partial hyperbolicity.\newline
Let $M$ be a Riemmanian manifold endowed with its volume measure and let $f:M\rightarrow M$ be a diffeomorphism. We say that $f$ is partially hyperbolic if there exists a $Df$-invariant splitting
\begin{equation*}
    TM = E_f^s \oplus E_f^c\oplus E_f^u,
\end{equation*}
where $E_f^s$ and $E_f^u$ are non trivial bundles, and for $*=s,c,u$ there exist continuous functions $\mu_*,\lambda_*:M\rightarrow\mathbb{R}_+$ such that
\begin{equation*}
    \mu_*(x) < \frac{\|D_xf(v^*)\|}{\|v\|} < \lambda_*(x), \quad \text {for every } v^*\in E_f^*(x)\setminus\{0\},
\end{equation*}
and
\begin{equation*}
    \lambda_s(x) < \min\{1,\mu_c(x)\}, \quad \quad \max\{1,\lambda_c(x)\}<\mu_u(x).
\end{equation*}
In other words, $E^s_f$ and $E^u_f$ are hyperbolic bundles, and the vectors of $E^c_f$ are contracted or expanded less than the ones of $E^s_f$ and $E^u_f$.\newline
If $\mu_*$ and $\lambda_*$ do not depend on $x$, for $*=s,c,u$, then we say that $f$ is absolutely partially hyperbolic \cite{brin1974partially}. In addition, we say that $f$ is center bunching if
\begin{equation*}
    \lambda_s(x) < \frac {\|D_xf(v)\|}{\|D_xf(w)\|} <\mu_u(x) \quad \text{for any } v,w\in E^c_f(x)\setminus\{0\}.
\end{equation*}

We always have that $E^s_f$ and $E^u_f$ are uniquely integrable to invariant H\"older continuous foliations $\mathcal{F}_f^s$ and $\mathcal{F}_f^u$. If the distributions $E^c_f$, $E_f^{cs}=E^s_f\oplus E^c_f$ and $E_f^{cu}=E^u_f\oplus E^c_f$ are uniquely integrable to foliations $\mathcal{F}_f^c$, $\mathcal{F}_f^{cs}$ and $\mathcal{F}_f^{cu}$, we say that $f$ is dynamically coherent. We denote by $W_f^*(x)$ the leaf of $\mathcal{F}_f^*$ that contains $x$, whenever this foliation exists and is unique, for $*=s,c,u,cs,cu$. Finally, dynamical coherence is a $\mathcal{C}^1$-stable property, provided that $f$ is absolutely partially hyperbolic (see \cite{hirsch1977invariant} or \cite{pesin2004lectures}).\newline

We say that a curve $\gamma:I\rightarrow M$ is an $su$-path, if it is piecewise contained in the leaves of $\mathcal{F}^s$ and $\mathcal{F}^u$. We say that the minimal number of such pieces is the number of legs of $\gamma$. A point $y$ is said to be $su$-accessible from $x$, if there exists a $su$-path joining the two and this is an quivalence relation. We denote the equivalence class of $x$ by $AC(x)=\{y\in M: y \text{ is accessible from } x\}$ and we refer to it as the accessibility class of $x$. A set $E\subset M$ is said $su$-saturated if, for every $x\in E$, we have that $AC(x)\subset E$. We say that $f$ is accessible if $AC(x)=M$ for every $x\in M$. Instead, $f$ is said essentially accessible (with respect to the volume) if every measuable $su$-saturated $E\subset M$ has either full or zero volume.\newline

As mentioned in the introduction essential accessibility is one of the main tools to prove ergodicity of partially hyperbolic systems with respect to the volume measure. We recall Burns-Wilkinson theorem \cite{burns2010ergodicity}.

\begin{theorem}
    Let $f:M\rightarrow M$ be a $\mathcal{C}^2$ volume-preserving, partially hyperbolic and center bunched diffeomorphism on a closed Riemmanian manifold $M$. If $f$ is essentially accessible, then $f$ is ergodic and in fact has the Kolmogorov property.
\end{theorem}

\subsection{Lifts and Perturbations}

In the rest of the paper $f:\mathbb{T}^N\rightarrow \mathbb{T}^N$ will denote a perturbation of $A$ in $\mathcal{C}^k_{vol}(\mathbb{T}^N)$.\newline
First of all, we notice that $0\in\mathbb{T}^N$ is a fixed point for $A$. Given the ergodicity of $A$, none of its eigenvalues can be equal to 1, implying that this fixed point is stable under small $\mathcal{C}^1$ perturbations, i.e., up to a small translational change of variable, we can assume that $f(0)=0$.\newline

We will often work in $\mathbb{R}^N$, the universal cover of $\mathbb{T}^N$, denoting the projection to the torus by $\pi:\mathbb{R}^N\rightarrow\mathbb{T}^N$. We consider the unique lift $F:\mathbb{R}^N\rightarrow\mathbb{R}^N$ of $f$ that fixes $0$. Finally, we can write $F = A + G$ with $G$ being $\mathbb{Z}^N$-periodic.\newline

We recall the splitting $\mathbb{R}^N=E^s\oplus E^c \oplus E^u$, mentioned in the introduction. Given $v\in\mathbb{R}^N$, we denote the corresponding decomposition by $v = v^s + v^c + v^u$. We also set $E^{cs}=E^c\oplus E^s$, $E^{cu}=E^c\oplus E^u$, $E^{su}=E^s\oplus E^u$ and, analogously, $v^{cs}=v^s+ v^c$, $v^{cu}= v^c + v^u$, $v^{su}=v^s+ v^u$. Finally, we endow $E^s,E^u$ and $E^c$ with norms making $A|_{E^s}$ and
$A^{-1} |_{E^u}$ contractions and $A|_{E^c}$ an isometry. On $\mathbb{R}^N$ we set the norm $|v|=|v^u|+|v^c|+|v^s|$. 
In this way we have that $(E^s)^{\perp}=E^{cu}$, $(E^c)^{\perp}=E^{su}$, $(E^u)^{\perp}=E^{cs}$.\newline 

It is easy to check that the constant bundles $E^*_A(x)=E^*$ for $*=s,c,u$, provide an $A$-invariant splitting of $T\mathbb{T}^N$, which is absolutely partially hyperbolic. Being these distributions constant, $A$ is dynamically coherent. In view of the previous discussion, provided that $f$ is sufficiently $\mathcal{C}^1$-close to $A$, $f$ and $F$ are dynamically coherent partially hyperbolic diffeomorphisms.\newline

We denote by $W^*_f$ and $W^*_F$ their invariant manifolds, for $*=s,c,u,cs,cu$. Very often we will omit the index, since the environment manifold ($\mathbb{T}^N$ or $\mathbb{R}^N$) will distinguish between the two, and we will use a different notation for the invariant manifolds of $A$.\newline

\subsection{Plan of Proof and Structure of the Paper}

Here we describe the general strategy of proof of Theorem \ref {stable ergodicity}, which is similar to the one in \cite{Hertz2005}. In this setting, the accessibility classes of $f$ are injectively immersed manifolds \cite{rodriguez2017structure} and their dimensions vary semi-continuously (section \ref{proof of ergodicity}). It is an open problem to determine if this holds in general \cite{zbMATH05287021}.The first step in the proof is to prove that all accessibility classes of $f$ have the same dimension and in view of the semi-continuity, the following minimality criterion is all we need.

\begin{theorem}\label{minimality} Let $A$ be a linear automorphism of $\mathbb{T}^N$ with $\dim(E)^c=2$.\newline
    If $f$ is sufficiently close to $A$ in $\mathcal{C}_{\text{vol}}^1(\mathbb{T}^N)$, for every $E\subset\mathbb{T}^n$ $su$-saturated, $f$-invariant non-empty closed set, $E=\mathbb{T}^N$.
\end{theorem} 
 
To conclude the proof we can now reason by cases. If the codimension of the classes is 0, they must be open and therefore $f$ is accessible by connectedness. Codimension 1 is ruled out by analyzing the fixed point of $f$ (namely $0$). Finally, if the codimension is 2, one can deduce the joint integrability of stable and unstable foliations. In this situation, exploiting some KAM theory, it is possible to $\mathcal{C}^1$-conjugate these foliation to those of $A$, deducing that $f$ is essentially accessible. Either way, ergodicity is finally obtained thanks to Burns-Wilkinson theorem, establishing Theorem \ref{stable ergodicity}.\newline

The novelty of our work stands in the generality of the minimality criterion (Theorem \ref{minimality}). Rodriguez-Hertz exploits a version which requires $A$ to be pseudo-Anosov, and it is the only part of his work where such assumption appears. The core of this paper is devoted to the proof of this theorem, which we sketch here.\newline

We will work in $\mathbb{R}^N$, considering $F:\mathbb{R}^N\rightarrow\mathbb{R}^N$, the lift of $f$, and $V=\pi^{-1}(E)$.\newline
The first step is to build a subspace $X\subset\mathbb{R}^N$, that contains $E^c$ and where a suitable power of $A$, let us say $A^k$, acts in a pseudo-Anosov way. This is section \ref{linear algebra}.\newline
We remark that this is not enough to reduce ourselves to Rodriguez-Hertz's argument. In fact, $X$ is typically not invariant under $F$, nor under $F^k$.\newline
We then examine a connected component $U$ of $V$, aiming to prove that $U=\mathbb{R}^N$.\newline
By a volume argument, valid for any $su$-satuared open set, we obtain the esistence of $n\in X\cap \mathbb{Z}^N$ such that $U+n=U$. By the algebraic properties of $A^k|_X$, we can extend the transaltional invariance of $U$ to a lattice $\Gamma$ of $X$ (subsection \ref{the structure of U}).\newline
It is also possible to show that $U$ is simply connected (appendix \ref{Simply connected proof}).\newline
By a topological argument, any simply connected set, which is invariant along directions that span $E^c$, must intersect every accessibility class (subsection \ref{main proof}). The $su$-saturation of $U$ let us conclude.\newline

We remark that theorem \ref{minimality}, apart from being the key of our argument, has a relevance on its own. The assumption that $\dim(E^c)=2$ is not restrictive in low dimension and small perturbation of linear automorphisms of the torus are commonly studied maps (see for example \cite{micena2021lyapunov}, \cite{gogolev2025smooth}, \cite{brown2026lyapunov}, \cite{MR4422214}, \cite{Adam_2017}).\newline

Finally, we describe the remaining sections of the paper. In section \ref{geometric preliminaries} we list some properties of the invariant manifolds and the holonomy maps they induce, which are peculiar to this setting. Section \ref{linear algebra} is dedicated to the algebraic preliminaries. In section \ref{Proof of the Minimality Criterion} we prove the minimality criterion (theorem \ref{minimality}) and finally, section \ref{proof of ergodicity} proves that the minimality criterion implies our main result (theorem \ref{stable ergodicity}).

\section{Geometric Preliminaries}\label{geometric preliminaries}

\subsection{Invariant Manifolds and Holonomies}\label{Invariant Manifolds and Holonomies}

In this subsection we collect some properties of the invariant manifolds and the associated holonomies, that are peculiar of our setting. These results are directly taken from \cite{Hertz2005}.\newline

In this setting, the invariant manifolds $W_F^*$ are graphs over the subspaces $E^*$. More precisely, there exist $g^*:\mathbb{R}^N\times E^*\rightarrow (E^{*})^{\perp}$ such that $W^*(x) = x+ \text{graph}(g^*(x,\cdot))$. Therefore we have that the parametrization $\sigma^{*}:\mathbb{R}^N\times E^{*} \rightarrow \mathbb{R}^N$ of the invariant foliation $\mathcal{F}^*$, given by $\sigma^{*}(x,v) = x + v + g^{*}(x,v)$. We also denote by $\sigma_x^*:E^*\rightarrow\mathbb{R}^N$ the parametrization of $W^*(x)$, defined by $\sigma_x^*(v)=\sigma^*(x,v)$. \newline

\begin{lemma}\label{invariant manifolds are Lip graphs}
There exist $\kappa=\kappa(f)$, with $\kappa(f)\rightarrow0$ as $f\rightarrow A$ in $\mathcal{C}^1$,  such that, for $*=s,c,u,cs,cu$,
\begin{equation*}
    |g^{*}(x,v)| \leq \kappa|v|, \quad\text{for every }v\in E^{*}.
\end{equation*}
Moreover, the parametrizations $\sigma_x^*$ are $(1+\kappa)$-biLipschitz continuous.
\end{lemma}

\begin{lemma}\label{unique intersection}
    For any $x,y\in\mathbb{R}^N$
    \begin{equation*}
        \#W^s(x)\cap W^{cu}(y)=1, \quad\text{and}\quad \#W^u(x)\cap W^{cs}(y)=1.
    \end{equation*}
\end{lemma}



For $x\in\mathbb{R}^N$ and $y\in W^{cu}(x)$ we have the holonomy $h^u_{xy}:W^c(x)\rightarrow W^c(y)$, defined by 
$$ h^u_{xy}(z) = W^u(z)\cap W^{cs}(y). $$
For $y\in W^{cs}(x)$ one defines $h^s_{xy}:W^c(x)\rightarrow W^c(y)$ by $h^s_{xy}(z)=W^s(z)\cap W^{cu}(y)$.\newline
If $\gamma$ is a $su$-path connecting $x$ and $y$ we can define $h^{\gamma}:W^c(x)\rightarrow W^c(y)$, by concatenating the the stable and unstable holonomies along the corresponding legs of the path.\newline
Given $x,y\in\mathbb{R}^N$ we have a way to define a preferential holonomy between $W^c(x)$ and $W^c(y)$. Exploiting the uniqueness of the intersection between $W^{s}$ and $W^{cu}$, we define $h^{su}_{x,y}=h^{s}_{x,z}\circ h^{u}_{z,y}$, where $z=W^s(x)\cap W^{cu}(y)$.\newline

Exploiting Lemma \ref{unique intersection} we can define
$\pi^s:\mathbb{R}^N\rightarrow W^{cu}(0)$, by $\pi^s(x)=W^s(x)\cap W^{cu}(0)$  and $\pi^u:\mathbb{R}^N\rightarrow W^{cs}(0)$, by $\pi^u(x)=W^u(x)\cap W^{cs}(0)$. Finally we set
\begin{equation*}
    \pi^{su}:\mathbb{R}^N\rightarrow W^c(0) \quad \text{ by } \quad \pi^{su}=\pi^s\circ\pi^u.
\end{equation*}
Notice that $\pi^{su}$ preserves the accessibility classes, i.e. $\pi^{su}(z)\in AC(z), \forall z\in\mathbb{R}^N$.\newline

Then we set $\hat{T}_n :W^c(0)\rightarrow W^c(0)$ by $\hat{T}_n=h^{su}_{n,0}\circ(id+n)$. Notice that, when projected to $\mathbb{T}^N$, these are holonomies along $su$-loops. To make some properties of this family of maps more evident it will be at times useful to read them in charts, hence defining $T_n:E^c\rightarrow E^c$ by $T_n=(\sigma_0^c)^{-1}\circ \tilde T _n \circ \sigma_0^c$.\newline
Notice that, a priori, it might hold that $T_n\circ T_m\neq T_{n+m}$.

\begin{lemma}\label{Holonomies Lip}
    There exists $\beta=\beta(f)$ with $\beta(f)\rightarrow 0$ as $f\rightarrow A$ in $\mathcal{C}^1$ and $C>0$, that depends only on the size of the $\mathcal{C}^1$ neighbourhood of $A$ such that the following holds.\newline
    For any $su$-path $\gamma$ with at most $K$ legs and length less than $L\geq1$,
    \begin{equation*}
        \text{Lip}(h^{\gamma})\leq C^KL^{K\beta}.
    \end{equation*}
    Moreover, for every $n\in\mathbb{Z}^N$ with $n\neq 0$,
    \begin{equation*}
        \text{Lip}(\hat T_n)\leq C|n|^{\beta}.
    \end{equation*}
\end{lemma}

\begin{lemma}\label{Tn log}
    There exists $C >0$ that only depends on the $\mathcal{C}^1$ size of the neighborhood of $A$ such that, for any $x\in E^c$ and $n\in\mathbb{Z}^N$,
    \begin{equation*}
        |T_n(x)-(x+n^c)| \leq C\log^+(|n|)+C.
    \end{equation*}
\end{lemma}
Here $\log^+(x)=\max\{0,\log(x)\}$.

\begin{lemma}\label{Tn close to linear}
If $f$ is $\mathcal{C}^r$, then $T_n$ is $\mathcal{C}^r$ for all $n\in\mathbb{Z}^N$, moreover, writing
$T_n(z) = z+ n^c + \phi_n(z)$
then for any $\varepsilon>0$ and $R>0$ there is a neighborhood of $A$ in the $\mathcal{C}^r$ topology such that if $f$
is in this neighborhood, then $\|\phi _n\|_{\mathcal{C}^r} <\varepsilon$ whenever $|n| \leq R$.
\end{lemma}

\subsection{Saturation}\label{Su-saturation}

Here we prove futher results about the invariant foliations, in particular how they jointly saturate the environment space.\newline

We set $E^*(x)=x+E^*$. Moreover, for $S\subset\mathbb{R}^n$, we define $E^*(S)=\bigcup_{x\in S} E^*(x)$.\newline
We also define define $W_r^{*}(x)=\sigma_x^*(B_{E^*}(0,r))$ and, similarly, we set $W^*(S)=\bigcup_{x\in S} W^*(x)$ and $W_r^*(S)=\bigcup_{x\in S} W_r^*(x)$.\newline

We define $\Phi_x(V):\mathbb{R}^N\rightarrow \mathbb{R}^N$ by $\Phi_x(v) = \sigma^u(v^u,\sigma^s(v^s,\sigma^c(v^c,x)))$.

\begin{lemma}\label{Phi close to identity}
    There exist $\kappa=\kappa(f)$ with $\kappa\rightarrow0$ when $f\rightarrow A$ in $\mathcal{C}^1$ such that
    $$|\Phi_x(v)-(x+v)|\leq\kappa|v|.$$
    Moreover, whenever $f$ is sufficiently $\mathcal{C}^1$-close to $A$, $\Phi_x$ is proper.
\end{lemma}

\begin{proof}
    We have, thanks to lemma \ref{invariant manifolds are Lip graphs}, that
    \begin{align*}
        |\Phi_x(v)-(x+v)|=&|\sigma^u(v^u,\sigma^s(v^s,\sigma^c(v^c,x)))-(x+v)|\\
        \leq&|\sigma^u(v^u,\sigma^s(v^s,\sigma^c(v^c,x))) - (\sigma^s(v^s,\sigma^c(v^c,x)) + v^u)|\\
        &\quad+|\sigma^s(v^s,\sigma^c(v^c,x))-(\sigma^c(v^c,x)+v^s)| +|\sigma^c(v^c,x)-(x+v^c)|\\
        \leq& \kappa|v^u| + \kappa|v^s| + \kappa|v^c|\\
        =& \kappa|v|,
    \end{align*}
    where $\kappa$ is, of course, the same as in Lemma lemma \ref{invariant manifolds are Lip graphs}.\newline
    Finally, we have that $|\Phi_x(v)|\geq (1-\kappa)|v-x|$, hence $\Phi_x$ is proper, whenever $\kappa<1$.
\end{proof}

\begin{lemma}\label{proper 1}
    If $f$ is sufficiently close to $A$ in the $\mathcal{C}^1$ topology we have that $$W_r^u(W_r^s(W_r^c(x)))\supset B(x,r/2).$$ 
    Moreover $\Phi_x$ is surjective.
\end{lemma}

\begin{proof}
    Suppose that $f$ is sufficiently $\mathcal{C}^1$-close to $A$ to ensure that $\kappa$ (from lemma \ref{invariant manifolds are Lip graphs} and \ref{proper 1}) is less than $1/2$. We will prove that $\Phi_x(B(0,r))\supset B(x,r/2)$ and both conclusions will follow directly from this claim.\newline

    We define $H:\overline{B(0,r)}\times[0,1]\rightarrow\mathbb{R}^N$ by $H(v,t) = (1-t)\Phi_x + t(x+v)$. This is the linear interpolation homotopy between $H_0=\Phi_x$ and $H_1=Id+x$.\newline
    For $v\in\partial B(0,r)$, lemma \ref{proper 1} guarantees that $\Phi(x)\in B(x+v,r/2)$, and, by convexity, it follows that $H_t(v)\in B(x+v,r/2)$ for every $t$. Therefore, $H_t(\partial B(0,r))\cap B(x,r/2)=\emptyset$ for every $t$ and, consequently, for every $y\in B(x,r/2)$
    \begin{equation*}
        \deg(\Phi_x,B(0,r),y) = \deg(Id+x,B(0,r),y) = 1.
    \end{equation*}
    This implies that $y\in \Phi_x(B(0,r))$, which gives the conclusion.

\end{proof}

\begin{lemma}\label{homomorphism}
    If $f$ is sufficiently close to $A$ in the $\mathcal{C}^1$ topology, $\Phi_x$ is an homeomorphism and
    \begin{equation*}
        |\Phi^{-1}_x(v)-(v-x)|\leq\frac{\kappa}{1-\kappa}|v-x|.
    \end{equation*}
\end{lemma}

\begin{proof}
    We require that $f$ is in a $\mathcal{C}^1$ neighbourhood of $A$ where $\kappa<1$.\newline

    We start by proving that $\Phi_x$ is injective.\newline
    Suppose that $\sigma^u(v_i^u,\sigma^s(v_i^s,\sigma^c(v_i^c,x)))=z$ for some values of $v_i^*\in E^*$, for $i=1,2$ and $*=s,u,c$. Since 
    $\sigma^s(v_i^s,\sigma^c(v_i^c,x))\in W^u(z)\cap W^{cs}(x)$, by lemma \ref{unique intersection}, it must be the same point for $i=1,2$. We denote this point by $w$. Now, $\sigma^c(v_i^c,x)\in W^s(w)\cap W^{cu}x$, hence it must be again a uniquely determined point $y$. Finally, the injectivity of $\sigma_x^c$, $\sigma_y^s$, $\sigma_z^u$ inplies that $v_1^*=v_2^*$ for $*=c,s,u$, proving the claim.\newline

    This and Lemma \ref{proper 1} prove that $\Phi_x$ is bijective and therefore $\Phi_x^{-1}$ is well-defined.\newline
    We now concern ourselves with the final estimate, leaving the continuity of $\Phi_x^{-1}$ to the end of the proof. From lemma \ref{Phi close to identity}, as long as $\kappa<1$, we deduce that $|\Phi_x(v)-x|>(1-k)|v|$. Chaning this with lemma \ref{Phi close to identity} itself, we get the following.

    \begin{align*}
        |\Phi_x(v) - (v+x)| < \kappa |v| &< \frac {\kappa}{1-\kappa} |\Phi_x(v) - x|,\\
        |\Phi_x^{-1}(w) - (w-x)| &< \frac{\kappa}{1-\kappa}|w-x|,
    \end{align*}
    where the second line follows from the first one with the substitution $v=\Phi_x^{-1}(w)$.\newline

    Finally, recalling that $\Phi_x:\mathbb{R}^N\rightarrow\mathbb{R}^N$ is bijective, continuos and proper, the continuity of $\Phi_x^{-1}$ follows from the next general topology fact.
\end{proof}

\begin{lemma}
Let $X$, $Y$ be Hausdorff topological spaces, with $Y$ admitting a locally finite compact cover. Let $f:X\rightarrow Y$ a continuous proper bijective function.\newline
Then $f$ is an homeomorphism.
\end{lemma}

\begin{proof}
    Let $(K_i)_i$ the locally finite compact cover of $Y$. Since $f$ is proper, $f^{-1}(K_i)$ is compact and therefore $f|_{f^{-1}(K_i)}$ is closed. This, together with $f^{-1}(K_i)$ being closed, implies that $f^{-1}|_{K_i}$ is continuous. We can finally conclude that $f^{-1}$ is continuous, since it is continuous when restricted to all elements of a locally finite closed cover.\newline
\end{proof}

Defining $\Psi_x(v)=\sigma(v^u,\sigma(v^s,x))+v^c$, we have that Lemma \ref{Phi close to identity} holds for $\Psi_x$ and, arguing as in lemma \ref{proper 1}, we get the following.

\begin{lemma}\label{proper 2}
    If $f$ is sufficiently close to $A$ in the $\mathcal{C}^1$ topology, then $$E^c(W^u_r(W^s_r(x)))\supset B(x,r/2)$$
\end{lemma}

\section{Algebraic Preliminaries}\label{linear algebra}

In this section we develop all needed algebraic tools. In particular we prove the fundamental Lemma \ref{Pseudo Anosov subspace}, about the action of $A$ on $\mathbb{Z}^N$. Secondly, we classify the possible dimensions of $E^c$ in low dimensions.\newline
Finally we recall Diophantine estimates for vectors in $E^c$, that will be crucial in the conclusion of the proof of stable ergodicity, that exploits some KAM theorems.

\subsection{Characteristic Polynomial Properties}

In this section $V$ will denote a real finite dimesional vector space and $L:V\rightarrow V$ a linear map. Furthermore, we will assume the existence of a maximal rank discrete subgroup $S$ of $V$, which is also $L$-invariant.\newline

\begin{remark}\label{change of basis}
    There always exists a linear isomorphism between $(V,S)$ and $(\mathbb{R}^d,\mathbb{Z}^d)$, where $d=\dim(V)=\text{rank}(S)$. This isomorphism is always bi-Lipschiz, for any two norms on $V$ and $\mathbb{R}^d$.
\end{remark}

In view of this remark, in a suitable basis, $L$ can be represented by a matrix in $SL(d,\mathbb{Z})$. In particular, its characteristic polynomial has always integer coefficients.\newline

We say that a vector $v\in V$ is cyclic for $L$ if $\text{span}_{\mathbb{R}}\{v, Lv, \dots,L^{d-1} v\} = V$.

\begin{lemma}
    The characteristic polynomial $p_L$ of $L$ is irreducible in $\mathbb{Z}[x]$ if and only if every non-zero element of $S$ is cyclic for $L$.
\end{lemma}

\begin{proof}\label{cyclic iff irreducible}
    Thanks to remark \ref{change of basis}, we can assume that $(V,S)=(\mathbb{R}^d,\mathbb{Z}^d)$.\newline
    We prove the first implication.\newline
    For the sake of contradiction, suppose there exists $n\in\mathbb{Z}^d\setminus\{0\}$ and $k<d$ such that we can write $\sum_{i=0}^k m_iA^in=0$, for a non-zero choice of integer coefficients $m_i$.\newline 
    Set $q(x)=\sum_{i=0}^k m_ix^i$. Since $p_A(x)$ is irreducible over $\mathbb{Z}$, it is over $\mathbb{Q}$ and therefore $\mathbb{Q}[x]/(p_A(x))$ is a field. Since $q(x)$ is a non-zero element, there exists $r(x)\in\mathbb{Q}[x]/(p_A)$ such that $r(x)q(x)=1$. Evaluating this expression in $A$, we get $r(A)q(A)=\text{id} + k(A)p_A(A)=\text{id}$ for some $k(x)\in\mathbb{Q}(x)$. Since $p_A(A)$, $q(A)$ has to be an invertible matrix, contraddicting $\sum_{i=0}^k m_iA^in=0$ for $n\neq 0$.\newline
    For the reverse implication, suppose that $P_L$ has a non trivial irreducible factor $q$ in $\mathbb{Z}[x]$. Then, working in $\mathbb{Q}^d$, $\ker{q(L)}$ is a non-trivial $L$-invariant $\mathbb{Q}$-subspace, that contains integer vectors which cannot be cyclic.
\end{proof}

\begin{definition}
    We say that $L$ is Pseudo-Anosov on $(V,L)$ if one of the following three equivalent conditions holds:
    \begin{enumerate}
        \item Every non-zero $s\in S$ is cyclic for $L^k$, for every $k>0$;
        \item The characteristic polynomial $p_{L^k}$ of $L^k$ is irreducible in $\mathbb{Z}[x]$, for every $k>0$;
        \item The characteristic polynomial $p_{L}$ of $L$ is irreducible in $\mathbb{Z}[x]$, and it is not a polynomial in $x^m$ for any $m>1$.
    \end{enumerate}
\end{definition}

The first two conditions are equivalent because of Lemma \ref{cyclic iff irreducible} and are the only two we are going to work with in this paper. The third was included to facilitate comparison with \cite{Hertz2005}, as that is their original definition. The equivalence between that and the second condition is, in fact, proved in \cite{Hertz2005} (Lemma A.9).

\begin{lemma}\label{Pseudo Anosov subspace}
    Let $A\in SL(N,\mathbb{Z})$ have no root of unity as eigenvalue and $\dim(E^c)=2$. Then there exists $k\in\mathbb{N}$, and $X$ real subspace of $\mathbb{R}^N$, such that:
    \begin{itemize}
        \item $E^c\subset X$;
        \item $X$ is $A^k$-invariant;
        \item $\Lambda=\mathbb{Z}^d\cap X$ has full rank in $X$;
        \item $A^k$ is Pseudo-Anosov on $(X,\Lambda)$.
    \end{itemize}
\end{lemma}

\begin{proof}
    Let $p_k(x)$ be the irreducible factor over $\mathbb{Q}$ of the charecteristic polynomial $p_{A^k}(x)$ of $A^k$ that has both unitary roots of $p_{A^k}(x)$. (If it has one it has to have the other one, since it has real coefficients and the two roots are complex conjugates).\newline
    Let us set $d_k=\deg(p_k(x))$.\newline
    We now work in $\mathbb{Q}^N$. Let us set $X_k=\ker(p_k(A^k))$. $X_k$ is an $A^k$-invariant space and the minimal polynomial of $A^k|_{X_k}$ is $p_k(x)$. It has to be the charcacteristic polynomial as well since it must divide $p_{A^k}(x)$ and $p_k(x)$ is an irreducible factor of multiplicity one (since its complex unitary roots have multiplicity one). It follows that $A^k|_{X_k}$ has dimension equal to $d_k$ and irreducible characteristic polynomial over $\mathbb{Z}$. In view of the previous lemma, every non-zero vector of $X_k$ is then cyclic with respect to $A^k$ (considering a suitable integer multiple). This implies that $X_k\cap\mathbb{Z}^N$ has full rank and that $X_k$ has no non-trivial $A^k$-invariant subspaces (and we stress that we are talking about $\mathbb{Q}$-subspaces).\newline
    Set $\tilde{X_k}=\text{span}_{\mathbb{R}}(X_k)$. We have $\tilde{X_k}=\ker(p_k(A^k))$, this time seen as an linear endomorphism of $\mathbb{R}^d$. Let us denote by $q_k(x)$ the real coefficient polynomial that has as roots the unitary roots of $p_{A^k}(x)$. Since $q_k(x)|p_k(x)$, we have that $E^c=\ker(q_k(A^k))\subset \ker(p_k(A^k))=\tilde{X_k}$.\newline
    Since $d_k$ are natural numbers we can pick $k$ such that $d_k$ is minimized. We fix such $k$ for the rest of the proof and set $X=\tilde{X_k}$.\newline
    We claim that $X_{kl}=X_k$ for any $l\geq 1$.\newline
    Since $\tilde{X_k}$ and $\tilde{X_{kl}}$ cointain $E^c$ their intersection is non trivial and hence $V=X_k\cap X_{kl}$ is as well. Since $X_k$ is $A^k$ invariant, $V$ is an $A^{kl}$ invariant non-trivial subspace of $X_{kl}$. Hence $V=X_{kl}$, implying $X_{kl}\subset X_k$, but since $d_k$ was minimal $X_{kl}=X_k$.\newline
    Since $X=X_{kl}$ we have that any non-zero integer vector of $X$ is $A^{kl}$-cyclic as well, hence $A^k$ is Pseudo-Anosov on $X$.\newline
\end{proof}

\begin{lemma}
    Let $p(x)$ be an irreducible monic polynomial in $\mathbb{Z}[x]$ with constant term 1. Suppose that $p(x)$ has a unitary root that is not a root of unity. Then, $\deg(p(x))$ is even and, in particular, $p(x)$ is a reciprocal polynomial. Furthermore, $p(x)$ has at least a root of modulus less then 1 and one of modulus greater than 1.
\end{lemma}

\begin{proof}
    This is essentially lemma A.3 in \cite{Hertz2005}, we present an alternative proof.\newline
    Let $\xi$ be the non-real unitary complex root of $p(x)$. Then $1/\xi=\bar{\xi}$ is a root of $p(x)$. Setting $n=\deg(p(x))$, we have that $x^np(1/x)$ is a monic irreducible polynomial in $\mathbb{Z}[x]$ that vanishes in $1/\xi$. By uniqueness of the carachteristic polynomial we must have that $p(x)=x^np(1/x)$, meaning that it is reciprocal. Therefore, the roots of $p(x)$, all distinct beacuse it is irreducible, comes in pairs of inverses, since $\pm1$ are not roots of $p(x)$. This implies that $\deg(p(x))$ is even.\newline
    Finally if all roots of a monic integer coefficients polynomial are unitary they must be roots of unity (Kroenecker's theorem, see \cite{Greiter1978}).\newline
\end{proof}

\begin{lemma}
    We have that $\text{dim(X)}$ must be even and $\dim(X)\geq 4$.
\end{lemma}

\begin{proof}
    By construction of $X$, the carachteristic polynomial $p_{A^k|_X}(x)$ of $A^k|_X$ is irreducible over $\mathbb{Z}$. Since $\dim(X)=\deg(p_{A^k|_X}(x))$, it must be even because of the previous lemma, applied to $p_{A^k|_X}(x)$. The same lemma rules out dimension 2, because $p_{A^k|_X}(x)$ has two unitary roots.
\end{proof}

\begin{lemma}\label{dimension 7 polynomial}
    Let $A\in SL(7,\mathbb{Z})$ have no roots of unity as eigenvalues.\newline
    Then, either $\dim(E^c)=0$ or $\dim(E^c)=2$.
\end{lemma}

\begin{proof}
    We have that $\dim(E^c)$ equals the number of unitary roots of the carachteristic polynomial $p_A(x)$ of $x$. Let $q(x)$ be an irreducible factor of $p_A(x)$ in $\mathbb{Z}[x]$ that has a unitary root. Therefore, $\deg(q(x))$ is even and at least 4. This proves that $p_A(x)$ can have at most one such factor, hence $q(x)$ has all the unitary roots of $p_A(x)$. If $\deg(q(x))=6$, then $p_A(x)$ would have a rational root, which is impossible. The degree of $q(x)$ must be 4, and consequentely it has at most 2 unitary roots.
    \newline
\end{proof}

\subsection{Diophantine Estimates}

Up to a bi-Lipschitz change of coordinates we can assume that $(X,\Gamma)$ from the previous lemma is $(\mathbb{R}^d,\mathbb{Z}^d)$ for a suitable $d$. Therefore, many results about Psuedo-Anosov linear maps of $\mathbb{R}^N$ from \cite{Hertz2005} can be directly transposed to our setting. Here we also exploit the fact that the following statements are invariant if one swaps $A$ with $A^k$, where $k$ comes from Lemma \ref{Pseudo Anosov subspace}.\newline

\begin{lemma}\label{Diphantine center}
    There exists a constant $c'=c'(A)>0$, such that, calling $r=\frac{\text{dim}X}2$,
    \begin{equation*}
        |n^c| \geq\frac {c'}{|n|^{r}} \quad\quad \text{for any }n\in\Lambda, n\neq0.
    \end{equation*}
\end{lemma}

\begin{proof}
    This is lemma A.10 in \cite{Hertz2005} with $\delta=1/2$.\newline
\end{proof}

We now fix a $\mathbb{Z}$-basis of $\Lambda$, namely $\{e_1,\dots, e_{\dim(X)}\}$.\newline
We also consider the linear trasformation $R:E^c\rightarrow \mathbb{R}^2$ defined by $R(e_1^c)=(1,0)$ and $R(e_2^c)=(0,1)$. The next two lemmas are Lemma A.11 and Lemma A.12 in \cite{Hertz2005}.

\begin{lemma}\label{Diphantine condition 1}
    If $\dim(X)\geq6$ there exists $n\in\Lambda$ such that if we call $R(n^c)=\alpha$, then for any $\delta>0$, there exists a constant $c'>0$, such that
    \begin{equation*}
        ||| k\alpha||| \geq \frac c{k^{2+\delta}}, \quad\quad \forall k\in\mathbb{Z},k\neq0.
    \end{equation*}
\end{lemma}

Unfortunately this lemma is not true in dimension 4, hence, in that case, we will have to rely on the following alternative one.

\begin{lemma}\label{Diphantine condition 2}
    If $\dim(X)=4$ there exist $n_1,n_2\in\Lambda$ such that if we call $R(n_1^c)=\alpha_1$ and $R(n_2^c)=\alpha_2$, then there exists a constant $c'>0$, such that
    \begin{equation*}
        \max_{i=1,2}|||k\cdot\alpha_i||| \geq \frac c{k^2}, \quad\quad \forall k\in\mathbb{Z}^2,k\neq0.
    \end{equation*}
\end{lemma}

\section{Proof of the Minimality Criterion}\label{Proof of the Minimality Criterion}

We focus on a connected component $U$ of $V$, with the goal of showing that $U=\mathbb{R}^N$.

\subsection{The structure of U}\label{the structure of U}

Our first goal is to prove that $U$ is invariant by a discrete subgroup of directions that spans $E^c$ and that $U$ is simply connected. Here we follow ideas from \cite{Hertz2005}.\newline

Given $\varepsilon>0$ we set $L(\varepsilon)=\varepsilon^{-2}$, then define $W_{\varepsilon}(x)=W^s_{\varepsilon}(W^u_{L(\varepsilon)+\varepsilon}(W^s_{L(\varepsilon)}(W^c_{\varepsilon}(x))))$.

\begin{lemma}\label{infinite volume}
    Provided that $f$ is sufficiently $\mathcal{C}^1$-close to $A$ the following holds.\newline
    For every $x\in\mathbb{R}^N$, we have that $\text{Vol}(W_{\varepsilon}(x)\cap(X+x))\rightarrow +\infty$ when $\varepsilon\rightarrow0$.
\end{lemma}

\begin{proof}
    The space $X/E^c$ inherits naturally a distance from $X$, and so does $X/E^c+x$. We set $2R=L(\varepsilon)$ and consider the ball $B=B(x,R)$ in $X/E^c+x$. Setting $2\delta=\varepsilon L(\varepsilon)^{-2\beta}$, we can find a $2\delta$-separated subset $\{x_1,\dots,x_N\}$ of $B$, where $N\geq C (R/\delta)^{\text{dim}(X)-2}$, where $C$ depends only on $\text{dim}(X)$.\newline
    By lemma \ref{proper 2}, there exist points $y_1,\dots,y_N$, such that, for every $i$, $y_i\in E^c(x_i)$ and there exists a 2-pieces $su$-curve $\gamma_i$, that joins $x_i$ and $y_i$ and satisfies $\ell(\gamma_i)<L(\varepsilon)$.\newline
    We notice that $\{y_1,\dots, y_N\}$ is a $2\delta$-separated set in $X$.\newline
    In addition, we have that $y_i \in W^u_{L(\varepsilon)}(W^s_{L(\varepsilon)}(x))$. Lemma \ref{Holonomies Lip} implies that $W^c_{\delta}(y_i)\subset W^u_{L(\varepsilon)}(W^s_{L(\varepsilon)}(W^c_{\varepsilon}(x)))$, therefore $W^u_{\delta}(W^s_{\delta}(W^c_{\delta}(y_i)))\subset W_{\varepsilon}(x)$. Now, by lemma \ref{proper 1}, we get that $B(y_i,\delta)\subset W_{\varepsilon}(x)$. Since the balls $B(y_i,\delta)$ are disjoint we get that
    \begin{align*}
        \text{Vol}(W_{\varepsilon}(x)\cap(X+x)) &\geq \sum_{i=1}^N \text{Vol}(B_{X+x}(y_i,\delta)\\
        & = C \delta^{\text{dim}(X)}\frac{L(\varepsilon)^{\text{dim}(X)-2}}{\delta^{\text{dim}(X)-2}}\\
        &= C\varepsilon^{2-2(\text{dim}(X)-2-4\beta)}.
    \end{align*}
    If $\beta<1/4$ (i.e. $f$ is sufficinetly close to $A$), the exponent is negative and this concludes the proof.

\end{proof}

\begin{lemma}\label{n epsilon}
    For every $\varepsilon>0$ small enough and $x\in\mathbb{R}^N$, there exists $n=n_{\varepsilon}\in \Lambda\subset X$ such that $W_{\varepsilon}(x)\cap(W_{\varepsilon}(x)+n_{\varepsilon})\neq\emptyset$. Moreover, if $B(x,\varepsilon)\subset U$, then  $U+n_{\varepsilon}=U$.\newline
    Finally, $|n_{\varepsilon}|\leq 5(1+\kappa) L(\varepsilon)$
\end{lemma}

\begin{proof}
    Consider $\varepsilon$ small enough to ensure that $\text{Vol}(W_{\varepsilon}(x)\cap(X+x))>\text{Vol}(X/\Lambda)$, which is finite since $\Lambda$ has full rank. We now consider the projection $\pi:(X+x)\rightarrow X/\Lambda$. If the fist part of the statement did not hold,we would then have that $\pi|_{W_{\varepsilon}(x)\cap(X+x)}$ is injective, but this would contraddict the fact that $\pi$ is a local isometry and that $\text{Vol}(W_{\varepsilon}(x)\cap(X+x))>\text{Vol}(X/\Lambda)$.\newline
    Now we notice that $U+n_{\varepsilon}$ is a connected component of $V$ (since $V$ is $\mathbb{Z}^n$ invariant) and that the first part ensures that $U+n_{\varepsilon}\cap U\neq \emptyset$, giving that $U+n_{\varepsilon}=U$.\newline
    Finally we must have $|n_{\varepsilon}|\leq \text{diam}(W_{\varepsilon}(x))$, hence the last statement.
\end{proof}

\begin{lemma}
    $U+\Gamma = U$ for a suitable $\Gamma\subset\Lambda$, full rank subgroup of $X$.
\end{lemma}

\begin{proof}
    (Same as in [RH], expanded). The non-wandering set of $f$ is the whole $\mathbb{T}^N$ and $U$ is open. This implies the existence of $l\in\mathbb{N}$ and $h\in\mathbb{Z}^N$ such that $(F^l(U)+h)\cap U\neq\emptyset$. Since $V$ is $F$-invariant and $\mathbb{Z}^N$-invariant, $F^l(U)+h$ is a connected component of $V$, and therefore $U$, since their intersection is non-empty. To recap $U= F^l(U)+h$.\newline
    Using the fact that $F=A+G$, and that $A$ is linear and $G$ is $\mathbb{Z}^N$-periodic we can show by induction that $F^l(U+n)=F^l(U)+A^ln$.\newline
    Combining the two and lemma \ref{n epsilon}, we can write
    \begin{align*}
        U &= F^l(U) + h\\
        &= F^l(U+n) + h\\
        &= F^l(U) + A^ln + h\\
        &= U + A^ln.
    \end{align*}
By induction we get that $U$ is invariant by $A^{sl}n$ for every $s\in\mathbb{N}$.\newline
We set $\Gamma=\langle n, A^{kl}n, \dots, A^{kl(\dim(X)-1)}n\rangle$, which, by lemma \ref{Pseudo Anosov subspace}, is a full rank subset of $\Lambda$, since $n$ is cyclic for $A^{kl}$ and, by our proof, $U+\Gamma=U$. ($k$ is from Lemma \ref{Pseudo Anosov subspace}.)
\end{proof}

\begin{remark}
    Here we used Poincar\'e recurrence (and hence the volume preserving assumption) to see that the non-wandering set is full.
\end{remark}

\begin{lemma}
    $U$ is simply connected.
\end{lemma}

\begin{proof}
    A detailed proof is given in appendix \ref{Simply connected proof}.\newline
    The proof will mostly follow the one in \cite{Hertz2005}, with the only changes occurring in the proof of their lemma 4.6. Their proof exploits an object very similar, yet different, to $W_{\varepsilon}(x)$. The modifications we introduce in the proof make up for the slight differences between these two objects.

\end{proof}

\subsection{Conclusion of the proof}\label{main proof}

From the properties proved in the previous subsection, we conclude the proof with a topological argument. We actually prove the following statement, from which Theorem \ref{minimality} follows directly.

\begin{proof}
\begin{lemma}
    If $f$ is sufficiently close to $A$ in $\mathcal{C}_{\text{vol}}^1(\mathbb{T}^N)$, for every $V\subset\mathbb{R}^n$ $su$-saturated, $F$-invariant $\mathbb{Z}^N$-invariant non-empty open set, $V=\mathbb{R}^N$.
\end{lemma}

In view of lemma \ref{proper 2} and the $su$-saturation of $U$, to conclude the proof it is enough to show that $U$ is $E^c$-saturated.\newline

Fix $x\in U$ and $y\in E^c(x)$. We want to conclude that $y\in U$.\newline
The main idea is to combine the connectedness of $U$ and its $\Gamma$-invariance to build a curve in $U$ that wraps around $W^s(W^u(y))$. Then exploit the simple connectedness of $U$ to deduce that it must intersect $W^s(W^u(y))$, which yields the conclusion, by its $su$-saturation.\newline

\noindent We first need a preliminary lemma.\newline
Let us define $C_{\varepsilon}=\{z\in \mathbb{R}^n: |z^{su}|<\varepsilon|z|\}$ and $C_{y,\varepsilon}=y+C_{\varepsilon}$.

\begin{lemma}\label{piecewise linear curve}
    Given $x\in U$ and $y\in E^c(x)$ For every $\varepsilon>0$, there exist, $n_1,n_2,n_3\in \Gamma$ such that, for every $R>0$ there exist $M\in\mathbb{N}$ and a closed piecewise linear curve $\gamma:[0,M]\rightarrow \mathbb{R}^n\setminus E^{su}(y)$, such that:
    \begin{enumerate}
        \item $\gamma(i)\in x+\Gamma$, for $i=0,\dots,M$;
        \item $\gamma'|_{(i,i+1)}\equiv n_{j(i)}$ for a suitable $j(i)\in\{1,2,3\}$, for $i=0,\dots,M$;
        \item $\pi_1(\mathbb{R}^n\setminus E^{su}(y))=(\gamma)$;
        \item $\gamma(t)\in C_{y,\varepsilon}$ and $|\gamma(t)-y|>R$ for every $t$.
    \end{enumerate}
\end{lemma}

\begin{proof}
    We can suppose that $x=0$.\newline
    Define $d_{\Gamma}=\text{diam}(X/\Gamma)$. Take three points $z_1,z_2,z_3\in E^c$, such that $B(0,3d_{\Gamma}/\varepsilon)$ lies inside their convex hull. We can find $v_1,v_2,v_3\in\Gamma$ such that $|v_i-z_i|<d_{\Gamma}$. Therefore, $B(0,2d_{\Gamma}/\varepsilon)$ lies in the convex hull of $v_1^c, v_2^c,v_3^c$. This and the fact that $|v_i^{su}|<d_{\Gamma}$ imply that $\overline{v_iv_{i+1}}\subset C_{\varepsilon/2}$ (indices must be intended modulo 3).\newline
    Set $n_i=v_{i+1}-v_i$ (this does not depend on $R$).\newline
    We now take $K\in\mathbb{N}$, whose size will be specified later.\newline
    Consider $w_i=Kv_i$, and define the closed curve $\gamma:[0,3K]\rightarrow\mathbb{R}^n$, such that $\gamma(i)=w_i$, and $\gamma$ linear between these points.\newline
    Conditions 1 and 2 hold for any $K$.\newline
    For $K>(2d_{\Gamma}/\varepsilon)/|x-y|$, we have that $y$ lies in the convex hull of $w_1^c,w_2^c,w_3^c$, hence condition 3.\newline
    We notice that $\gamma(t)\in C_{\varepsilon/2}$ and $|\gamma(t)|>K(2d_{\Gamma}/\varepsilon)$ for every $t\in[0,3K]$.\newline
    For $K>|x-y|/d_{\Gamma}$, we have that $\gamma(t)\in C_{\varepsilon/2}\cap B(0,2|x-y|/\varepsilon)^c\subset C_{y,\varepsilon}$.\newline
    Finally, we pick $K$ large enough to ensure that $|\gamma(t)|>K(2d_{\Gamma}/\varepsilon)>R+|y|$, which is what we need to meet condition 4.\newline
    This concludes the proof of the lemma.\newline
\end{proof}

For the rest of the proof, we can suppose that $x\in E^c$ and $y=0$, for the sake of notation. Since $U$ is $su$-saturated, we assume, by contradiction, that $U\cap W^u(W^s(0)))=\emptyset$.\newline
We apply Lemma \ref{piecewise linear curve}.
Being $U$ connected there exist a path $\tilde\gamma_j:[0,1]\rightarrow U$ that joins $x$ and $x+n_j$ for $j=1,2,3$. By compactness, there exists a constant $C_{n_1,n_2,n_3}>0$ (that depends only on $n_1,n_2,n_3$) such that $|(\tilde\gamma_j(t)-x)-t\gamma_j)|<C$ for every $t\in[0,1]$. We set $R=2C_{n_1,n_2,n_3}/\varepsilon$. We can now define a new curve $\tilde\gamma:[0,M]\rightarrow \mathbb{R}^n$ by replacing its linear pieces with $\Gamma$-translates of the curves $\tilde\gamma_j$. More precisely
    \begin{equation}
        \gamma(t) = \gamma(i) + \tilde\gamma_{j(i)}(t)-x   \qquad \text{for }t\in[i,i+1).
    \end{equation}
Since $U$ is $\Gamma$-invariant, $\tilde\gamma:[0,M]\rightarrow U$. Moreover $\|\gamma-\tilde\gamma\|_{\infty}<C_{n_1,n_2,n_3}$. This implies that $\tilde\gamma(t)\in C_{2\varepsilon}$ for every $t$, and at the same time that, $\gamma$ and $\tilde\gamma$ are homotopic in $\mathbb{R}^n\setminus E^{su}$, implying that $\pi_1(\mathbb{R}^n\setminus E^{su})=(\tilde\gamma)$.\newline
Recall the map $\Phi_y$ from section \ref{Su-saturation}, and set $\Phi=\Phi_0$. By lemma \ref{homomorphism}, we can apply $\Phi^{-1}.$ Now, $\Phi^{-1}(W^u(W^s(0)))=E^{su}$ and we define $\hat\gamma=\Phi^{-1}\circ\tilde\gamma$. If $\kappa$ and $\varepsilon$ are small enough we have that $|\hat\gamma(t)-\tilde\gamma(t)|<|\tilde\gamma(t)^c|$ for every $t$. This implies that $\hat\gamma$ is homotopic to $\tilde\gamma$ in $\mathbb{R}^n\setminus E^{us}$ and therefore its homotopy type is nontrivial in $\mathbb{R}^n\setminus E^{us}$. Finally, we have that $\Phi^{-1}(U)\cap E^{su}=\emptyset$ and $\text{Im}(\hat\gamma)\subset \Phi^{-1}(U)$. Since $U$ has trivial fundamental group and $\Phi$ is an homomorphism, $\hat\gamma$ should be contractible in $\Phi^{-1}(U)$, but then, a fortiori, it should be contractible in $\mathbb{R}^n\setminus E^{su}$, providing the desired contradiction.
\end{proof}

\section{Proof of Stable Ergodicity}\label{proof of ergodicity}

This section is devoted to the proof of theorem \ref{stable ergodicity}.\newline
As mentioned, the proof follows ideas from \cite{Hertz2005}, with our theorem \ref{minimality} playing the role of theorem 4.1 in \cite{Hertz2005}. Our exposition takes advantage of more recent results about accessibility classes.\newline

Our goal is to prove the ergodicity of $f$ by showing that it is either accessible or essentially accessible. Hence we start by describing the structure of the accessibility classes.\newline

For $x\in\mathbb{T}^N$ we denote its accessibility class with respect to $f$ by $AC_{f}(x)$, while for $\tilde x\in \mathbb{R}^N$ we denote the one with respect to $F$ by $AC_F(\tilde x)$. When no confusion arises, we will drop this indices. If $\tilde x$ is a lift of $x$ we have that $\pi^{-1}(AC_{f}( x)) = AC_F(\tilde x)+\mathbb{Z}^N$ and, in particular, $AC_F(\tilde x)$ is the arc-connected component of $\pi^{-1}(AC_{f}( x))$ containing $\tilde x$.\newline

We remark that if $f$ is close enough to $A$, $f$ is dynamically coherent and center bunching. Therefore, we can apply theorem B from \cite{rodriguez2017structure}.

\begin{theorem}
Let $f : M \rightarrow M$ be a partially hyperbolic $\mathcal{C}^2$ diﬀeomorphism on a closed Riemmaninan manifold. Suppose $f$ to have a two-dimensional center bundle, and to be dynamically coherent and center bunching.\newline
Then all accessibility classes are injectively immersed $\mathcal{C}^1$-submanifolds.
\end{theorem}

We can then consider the quantity $\text{codim}(AC(x))$. By our previous consideretions we also have that $\dim(AC_F(\tilde x))=\dim AC_f(x)$.\newline We investigate how it varies in $x$, with the goal of proving that it must be constant.\newline
The first remark is that $\text{codim}(AC(x))=\text{codim}_{W^c(0)}(AC(x)\cap W^c(0))$.

\begin{lemma}\label{AC invariant homotopy}
    Working in $\mathbb{R}^N$, consider $x\in W^c(0)$ and $y\in W^c(0)\cap AC(x)$.\newline
    There exists a continuous homotopy $H:W^c(0)\times[0,1]\rightarrow W^c(0)$, such that $H(x,0) = x$, $H(x,1) = y$, $H(z,[0,1]) \subset AC(z)$ for any $z\in W^c(0)$. 
\end{lemma}

\begin{proof}
    Let $\gamma:[0,1]\rightarrow \mathbb{R}^N$ be a $su$-curve such that $\gamma(0)=x$ and $\gamma(1)=y$. We define $H(t,z)=\pi^{su}(h^{\gamma_{|[0,t]}}(z))$. First notice that $\Im(H)\subset\Im(\pi^{su})=W^c(0)$. Then, we have that $H(x,0)=\pi^{su}(x)=x$ and that $H(1,x)=\pi^{su}(h^{\gamma}(x))=\pi^{su}(y)=y$. Finally, the last requirement is met since both $h^{\gamma|_{[0,t]}}$ and $\pi^{su}$ preserve the accessibility classes.\newline
\end{proof}

This implies that $AC_F(x)\cap W^c(0)$ is arc-connected for every $x\in W^c(0)$.\newline
In particular, if $\dim(AC(x)\cap W^c(0))=0$, we must have that $\#AC(x)\cap W^c(0)=1$. 

\begin{lemma}
    The map $x \mapsto \text{codim}(AC(x))$ is upper semi-continuous on $\mathbb{T}^N$ and $\mathbb{R}^N$.
\end{lemma}

\begin{proof}
    The set $\{x\in\mathbb{R}^N:\text{codim}(AC(x))=0\}$ is open by definition.\newline
    We are left with proving that $\{x\in\mathbb{R}^N:\text{codim}(AC(x))\leq 1\}$ is open as well. For $x\in W^c(0)$ belonging to such set, we can find $y\neq x$ in $AC(x)\cap W^c(0)$. We can then apply lemma \ref{AC invariant homotopy}, and obtain the described homotopy $H$. Since $H(x,0)\neq H(x,1)$, we have that there exists $\varepsilon>0$ such that $H(z,0)\neq H(z,1)$ for every $z\in W^c_{\varepsilon}(x)$. Therefore, for every such $z$, $\text{codim}(AC(z))\leq 1$. This ends the proof.\newline
    The statement for $\mathbb{T}^N$ follows since $\pi:\mathbb{R}^N\rightarrow \mathbb{T}^N$ is open.
\end{proof}

\begin{corollary}
    All accessibility classes have the same dimension.
\end{corollary}

\begin{proof}
    The sublevels $\{x\in\mathbb{R}^N:\text{codim}(AC(x))\leq k\}$ are open, $su$-saturated, $F$-invariant and $\mathbb{Z}^N$-invariant. By theorem \ref{minimality}, each of these sets must be either empty or full, hence the conclusion.
\end{proof}

\begin{lemma}
    If $f$ is sufficiently $\mathcal{C}^1$-close to $A$, then $\text{codim}(AC(0))\neq 1$.
\end{lemma}

\begin{proof}
    Since $E^c$ contains no real eigenvector of $A$, if $f$ is sufficiently close to $A$, $D_0F$ is close enough to $A$ (in the space of normed linear operators) to guarantee that there are no real eigenvectors of $D_0F$ in a suitable open cone $U$ around $E^c$. Again, if $f$ is sufficiently close to $A$, $T_0W^c(0)$ is so close to $E^c$ to be contained in $U$. Therefore, no line of $T_0W^c(0)$ can be invariant under $D_0F$ contradicting the fact that $W^c(0)\cap AC(0)$ is $F$-invariant and has dimension 1.\newline
\end{proof}

We can now conclude the proof of Theorem \ref{stable ergodicity}.

\begin{itemize}
    \item If $\text{codim}(AC(x))=0$ for every $x$, all accessibility classes are open and by connectedness there is only one of them. Being $f$ accessible, it is also ergodic by Burns Wilkinson theorem.

    \item If $\text{codim}(AC(x))=2$ for every $x$, one can work exactly as in section $6$ of \cite{Hertz2005}.\newline
    We just describe the general strategy. The joint integrability $\mathcal{F}^s$ and $\mathcal{F}^u$ imply that $T_n\circ T_m=T_{n+m}$, i.e. $(T_n)_n$ induces a $\mathbb{Z}^N$ action on $E^c$.  Lemma \ref{Tn close to linear} says that the maps $T_n$ are close to traslations in $\mathcal{C}^{22}$ and lemmas \ref{Diphantine condition 1} and \ref{Diphantine condition 2} guarantee that for suitable choices of $n$ these translations are Diophantine. By KAM theory it is possible to linearize this whole $\mathbb{Z}^N$ action. This implies a $\mathcal{C}^1$ conjugacy between the stable and unstable foliations of $f$ and $A$. This implies that the essential accessibilty of A is inherited by $f$, which, in turn, implies the ergodicity.
\end{itemize}

\newpage

\section*{Appendix A - $U$ is simply connected}\label{Simply connected proof}


Recall the map $\pi^{su}:\mathbb{R}^N\rightarrow W^c(0)$, from subsection \ref{Invariant Manifolds and Holonomies}. It is a fibration (lemma 4.3 in \cite{Hertz2005}), and, as a consequence, we have the following (lemma 4.4 in \cite{Hertz2005}).

\begin{lemma}
    Given any open and connected su-saturated set $E$,
    \begin{equation*}
        \pi_1 (E) = \pi_1 (E \cap W^c(0)).
    \end{equation*}
\end{lemma}

The proof then reduces to showing that $U\cap W^c(0)$ is simply connected. We want to exploit the following algebraic topolgy lemma.

\begin{lemma}
    Let $E\subset\mathbb{R}^2$ be a connected open set such that all connected components of $\mathbb{R}^2\setminus E$ are unbounded. Then $E$ is simply connected.
\end{lemma}



\begin{proof}
    Suppose that $E$ is not simply connected. Since $E$ is an open subset of $\mathbb{R}^2$ we can find  Jordan curve $\gamma:[0,1]\rightarrow E$, which is non-contractible in $E$.
    By Jordan curve theorem there exists a disk $D\subset\mathbb{R}^2$ such that $\Im(\gamma)=\partial D$. If $D\subset E$, then $\gamma$ would be contractible in $E$. Hence a point of $D$ must belong to $\mathbb{R}^2\setminus E$, but then its connected component must be contained in $D$ and therefore bounded. This is the desired contradiction. \newline
\end{proof}

Let us consider $C=E^c\setminus (\sigma_0^c)^{-1}(U)$. Recalling that $\sigma_0^c$ is a bi-Lipschitz homeomorphism and identifying $E^c$ with $\mathbb{R}^2$, in view of the previous lemma, to conclude the proof we only need to prove that all connected components of $C$ are unbounded.

\begin{lemma}
    For any $ x\in E^c$ and $\delta>0$ there exist $n\in\Lambda$ with $U+n=U$, $k\in\mathbb{N}$, and curves $\eta_i:[0,1]\rightarrow E^c$, for $i=1,\dots,k$, such that, calling $\hat\eta_i=\sigma_0^c\circ\eta_i$ and $\hat x=\sigma_0^c( x)$:
    \begin{enumerate}
        \item $\hat\eta_i([0,1]) \subset (AC(\hat x)+\Lambda)\cap W^c(0)$;
        \item $|\eta_i(0)-T_{(i-1)n}( x)|<\delta$ and $\hat\eta_i(1) = T_{in}(\hat x)$;
        \item $|T_{kn}(x)- x|\rightarrow +\infty$, when $\delta\rightarrow 0$.
    \end{enumerate}
\end{lemma}

\begin{proof}
    We recall the constants $\beta$ from lemma \ref{Holonomies Lip} and $r$ from lemma \ref{Diphantine center}. We set $s=2r+1$ and $\gamma=1-\beta(s+14)$.\newline
    If $f$ is sufficiently $\mathcal{C}^1$-close to $A$, $\beta$ approaches $0$ and $\gamma$ is positive.\newline
    Given $\delta>0$, we consider $\varepsilon>0$ such that $\varepsilon^{\gamma}<\delta$ (notice that $\varepsilon\rightarrow0$, whenever $\delta\rightarrow0$). We obtain $n=n_{\varepsilon}$ from lemma \ref{n epsilon} (with $|n_\varepsilon|<5(1+\kappa)L(\varepsilon)$).\newline
    Finally we pick $k\in\mathbb{N}$ such that $\varepsilon^{-s}/2<k<\varepsilon^{-s}$.\newline

    We start by proving item 3.\newline
    Exploiting lemma \ref{Tn log} at the first line, and lemma \ref{Diphantine center} at the second one we have the following. (In the next computations we denote all universal constants with $C$, which might assume a diffent value, even within the same line.)
    \begin{align*}
        |T_{kn}( x) - (x + n^c)| &\geq (|kn^c|) - C\log|kn| - C\\
        &\geq k\frac{C}{|n|^r} - C\log(k)  - C\log(|n|) - C\\
        &\geq C \frac{\varepsilon^{-s}}{\varepsilon^{2r}} - sC\log(\varepsilon^{-1}) - 2C\log(\varepsilon^{-1})-2C\log(5C) -C\\
        &\geq C \varepsilon^{-1} - (2C + sC)\log(\varepsilon^{-1}) - C\\
        &\rightarrow+\infty.
    \end{align*}

    In order to construct the curves $\eta_i$, we first construct their starting points $\eta_i(0)= y_i$. We aim to find points $ y_i\in E^c$ such that, again denoting $\hat y_i=\sigma_0^c( y_i)$:
    \begin{itemize}
        \item[a)] $|T_{(i-1)n}( x)- y_i|<\delta$;
        \item [b)] $ \hat y_i\in AC ( \hat x+in)$.\newline
    \end{itemize}

    Recalling that, $W_{\varepsilon}(\hat x+n) = W_{\varepsilon}(\hat x)+n$, by lemma \ref{n epsilon}, we have that $W_{\varepsilon}(\hat x +n ) \cap W_{\varepsilon}(\hat x)\neq\emptyset$. This implies the existence of a 5-legged $su$-path $\gamma:[0,1]\rightarrow \mathbb{R}^N$ such that $\gamma(0)\in W_{\varepsilon}^c(\hat x + n)\subset W^c(n) $ and $\gamma(1)\in W_{\varepsilon}^c(\hat x)\subset W^c(0)$.\newline
    We call $\hat S=h^{\gamma}:W^c(n)\rightarrow W^c(0)$ the $su$-holonomy along $\gamma$. Notice that $\gamma$ is shorter that $10L(\varepsilon)$ and consequently, by lemma \ref{Holonomies Lip}, $\text{Lip}(\hat S)\leq C(L(\varepsilon))^{5\beta}$.\newline
    We call $\hat T_{m} = \sigma_0^c\circ T_m \circ (\sigma_0^c)^{-1}: W^c(0)\rightarrow W^c(0)$. By lemma \ref{Holonomies Lip}, $\text{Lip}(\hat T_m)\leq C|m|^{\beta}$.\newline

    We define $\hat y_i= \hat T_{(i-1)n}\circ \hat S(\hat x +n)$.\newline

    By definition of $\hat S$, we have that $\hat S (\hat x +n)\in AC(\hat x +n)$. Moreover, for any $\hat z \in W^c(0)$ and any $m\in\mathbb{Z}^N$, we have that $\hat T_m(\hat z) \in AC(\hat z + m)$. We conclude that $\hat y_i= \hat T_{(i-1)n}\circ \hat S(\hat x +n)\in AC(\hat x + (i-1)n +n)=AC(\hat x +in)$, satisfying the second condition we required on $ y_i$.\newline

    We now focus on the first condition on $y_i$.\newline
    Indeed, for $\varepsilon$ small enough, we have the following computation.

    \begin{align*}
        |y_i-T_{(i-1)n}(x)| & \leq |\hat y_i-\hat T_{(i-1)n}(\hat x)|\\
        & \leq \text{Lip}(\hat T_{(i-1)n}) | \hat S (\hat x +n) - \hat x|\\
        & \leq \text{Lip}(\hat T_{(i-1)n}) (| \hat S (\hat x +n) - \hat S(\gamma(0))| + | \gamma(1) - \hat x|)\\
        & \leq \text{Lip}(\hat T_{(i-1)n}) (\text{Lip}(\hat S)\varepsilon + C\varepsilon)\\
        & \leq C |kn|^{\beta} L(\varepsilon)^{5\beta} \varepsilon\\
        & \leq C \varepsilon^{-s\beta}\varepsilon^{-2\beta}\varepsilon^{-10\beta}\varepsilon\\
        &\leq \varepsilon^{1-\beta(14+s)} = \varepsilon^{\gamma} < \delta.
    \end{align*}

Now we can build the curves $\hat \eta_i$.\newline
Since $\hat y_i\in AC(\hat x + in)$, there exists a $su$-curve $\tilde \eta_i$, such that $\tilde\eta_i(0)=y_i$ and $\tilde\eta_i(1)=\hat x +in$. Of course, $\tilde \eta([0,1)]\subset AC(\hat x+in)$.\newline
We set $\hat \eta_i = \pi^{su} \circ \tilde \eta_i$.\newline
Since we have that $\pi^{su}(\mathbb{R}^N)=W^c(0)$ and $\pi^{su}$ preserves the accessibility classes, $\hat\eta_i$ satisfies item 1.\newline 
Moreover, we have that $\hat\eta_i(1)=\pi^{su}(\tilde\eta_i(1))=\pi^{su}(\hat x +in)=\hat T_{in}(\hat x)$, hence $\eta_i(1)=T_{in}(x)$. Finally, since $\hat y_i\in W^c(0)$, we have that $\hat\eta_i(0)= \pi^{su}(\tilde\eta_i(0))=\pi^{su}(\hat y_i)=\hat y_i$, and, therefore, $\eta_i(1)=y_i$ and $\eta_i$ satisfies item 2.\newline
\end{proof}

We can now conclude the proof following \cite{Hertz2005} (see the proof of Corollary 4.7).\newline
Let $x$ be a point in $C$. Suppose, by contradiction that its component of $C$ is contained in $B(x,R)$ for some $R>0$. For $\delta>0$ we apply the previous lemma to $x$.\newline
Since $U$ is $n$-invariant, $T_{in}(x)\in C$ for every $i=0,\dots,k$. In addidion, the $su$-saturation of $U$ implies that $\eta_i([0,1])\subset C$ for $i=0,\dots,k$.\newline
We define
\begin{equation*}
    \tilde K_\delta = \bigcup_{i=0}^k \overline{B(T_{in}(x),\delta)} \cup \bigcup_{i=1}^k \eta_i([0,1]),
\end{equation*}
and $K_{\delta}$ to be the connected component of $\tilde K_{\delta}\cap \overline{B(x,R)}$ that contains $x$.\newline
Since $\tilde K_{\delta}$ is connected and $|T_{kn}(x)-x|>R$ for $\delta$ small enough, we have that
\begin{equation*}
    K_{\delta}\cap \partial B(x,R)\neq \emptyset \quad \text{and that} \quad K_{\delta} \subset C+\overline{B(0,\delta)}.
\end{equation*}
By the compactness of the space of the compact subsets of $\overline{B(x,R)}$ with respect to the Hausdorff distance, we have the existence of a sequence $\delta_m\rightarrow 0$ and a compact subset $K$ of $\overline{B(x,R)}$ such that $K_{\delta_n}\rightarrow K$ in the Hausdorff metric.\newline
Passing to the limit, $K$ must be connected, contain $x$ and intersect $\partial B(x,R)$. At the same time we have that $K\subset C +\overline{B(0,\delta)}$ for every $\delta>0$, implying that $K\subset C$, in contradiction with our initial assumption.

\chapter{Delocalization for Schr\"odinger Operators on Some Large Graphs}\label{Chapter - Schreodinger}
\newcommand{\va}{\varepsilon}
\newcommand{\n}{\bold{n}}

\begin{center}
\large{A. Avila, A. Ulliana}
\end{center}

\subsection*{Abstract}

We prove that delocalization of most of eigenvectors is topologically common in the space of deterministic Schr\"odinger Operators on a given large graph, provided that their IDS has a certain degree of regularity. This generalizes a recent theorem of Avila and Damanik. The proof is based on an iterative scheme of perturbations of the potential and a study of the stability of delocalization. We also describe a flexible family of graphs that fall under our criterion, by proving a variant of the Thouless formula.

\section{Introduction}

\subsection{Context}

This paper focuses on deterministic Schr\"odinger operators on large graphs. Our purpose is to investigate for what family of graphs delocalization of eigenvectors is a common property from a topological point of view.\newline

Let $\mathcal{G}=(\mathcal{V},\mathcal{E})$ be a graph. Let us fix $\lambda>0$ and let $V:\mathcal{V}\rightarrow[-\lambda,\lambda]$ be a potential. We define the Schr\"odinger Operator $H_{\mathcal{G},V}:l^2(\mathcal{V})\rightarrow l^2(\mathcal{V})$ by
\begin{equation*}
    \left(H_{\mathcal{G},V}\phi\right)(v) =  \sum_{v'\sim v} [\phi(v')-\phi(v)] + V(v)\phi(v),
\end{equation*}
where $\phi\in l^2(\mathcal{V})$ and $v\sim v'$ means that $v$ and $v'$ are neighbours.\newline
When the potential is bounded by $\lambda$ and the degree of every vertex of $\mathcal{G}$ is bounded by $k$, $H_{\mathcal{G},V}$ is a bounded self-adjoint operator and $\sigma(H_{\mathcal{G},V})\subset[-\lambda-2k,\lambda+2k]$.\newline

Schr\"odinger Operators are oftentimes Hamiltonians of quantum mechanics systems, the most prominent examples coming from condensed matter physiscs. The graph structure of $\mathcal{G}$ and the potential $V$ are modeled after the physiscal carachteristics of system in question. Every element $\phi\in l^2(\mathcal{V})$ (called wave function) represents a state of the system. To each unitary $\phi$, one associates the probability measure $\mu_{\phi}=\sum_{\bold{v}\in\mathcal{V}}\phi(\bold{v})^2 \delta_{\bold{v}}$. The time-evolution of the wave function $\phi_t$ is described by the Schr\"odinger Equation
\begin{equation}
    i\hbar \frac{d \phi_t}{dt}=H\phi_t.
\end{equation}
Delocalization of Schr\"odinger Operators can be understood in three ways: dynamical, spectral and spatial. Dynamical delocalization is an asymptotical property of the solutions $\phi_t$ of the Schr\"odinger equation. Roughly, we say that $H$ has dynamical delocalization, if, for any fixed compact $K\subset\mathcal{G}$, we have that $\mu_{\phi_t}(K)\rightarrow0$ when $t\rightarrow+\infty$. Conversely $H$ has dynamical localization, when, for any $\va>0$, there exists a compact set $K\subset\mathcal{G}$ such that $\mu_{\phi_t}(K)>1-\va$ for all times $t$.\newline

The solutions of the non-time dependent Schreodinger equation are given by $\phi_t=e^{i\hbar t H}\phi$, whose behaviour depends heavily on the spectral properies of the Hamiltonian $H$. \newline
We say that $H$ has spectral delocalization when it has absolutely continum spectrum and that it has spectral localization when it has pure point spectrum.\newline
The RAGE theorem guarantees a degree of correspondance between spectral and dynamical localization. To give an intuition, we highlight the role of the eigenvectors of $H$ in this discussion. If $\phi_0$ is linear combination of the eigenfunctions $\psi_1,\dots,\psi_k$, then, at any time $t$, ${\phi_t}$ belongs to $\text{span}\{{\psi_1},\dots,{\psi_k}\}$. This implies that the bulk of the wave function will remain in a fixed compact region, determined by $\psi_1,\dots,\psi_k$. Otherwise, the intuition suggests that $\phi_t$ spreads along some non-integrable formal solution $\psi$ of $H\psi=E\psi$.\newline
A celebrated example of spectral localization is given by the Anderson model.\newline

The previous discussion refers to Schr\"odinger operators on infinite graphs.
However, we will focus on finite, although very large, graphs. This setting has recently received attention (in \cite{anantharaman2017quantum}, \cite{anantharaman2019quantum}, \cite{avila_delocalization_2024}) for mainly two reasons. First, it is a more faithful representation of the physical reality as the infinite graphs are just an idealization of the actual matter structure. Second, the guaranteed existence of a basis of eigenvectors allows a more detailed understanding of the degree of localization or delocalization of the wave function.\newline

When $\mathcal{G}$ is finite, according to the two previous definitions, only localization can take place. Delocalization has then to be understood in terms of the shape of the eigenvectors, as the size of the graph grows. Let $(\mathcal{G}_n)_n$ be an increasing sequence of graphs, with a corrsponding sequence of potentials $V_n:\mathcal{V}(\mathcal{G})\rightarrow [-\lambda,\lambda]$. We say that the corresponding sequence of Schr\"odinger operators $H_{\mathcal{G}_n,V_n}$ has spatial delocalization if, for most of the eigenvectors $\psi$ of $H_{\mathcal{G}_n,V_n}$, $\|{\psi}\|_{\infty}$ becomes small when $n\rightarrow\infty$.\newline

We remark that $\ell^{\infty}$ delocalization of eigenvectors can be expressed in terms of spectral measures. Indeed, given $v\in\mathcal{V}(\mathcal{G})$, let $\sigma_{V,v}$ denote the spectral measure corresponding to $\delta_v\in l^2(\mathcal{V}(\mathcal{G}))$), with respect to $H_{\mathcal{G},V}$. If $\psi$ is an eigenvector of $H_{\mathcal{G},V}$ corresponding to the eigenvalue $E$, it holds that $\psi(v)^2\leq \sigma_{V,v}(E)$, with equality when $E$ is simple.\newline

Finally we introduce a crucial element for our discussion, which has a notable relevance by itself: the Integrated Density of States (or IDS).\newline
When $\mathcal{G}$ is finite we define the desity of states measure $\widetilde\sigma$ as the average of the spectral measure of every site in $\mathcal{G}$,
\begin{equation}
    \widetilde{\sigma}_V = \frac 1{\#\mathcal{V}(\mathcal{G})} \sum_{\bold{v'}\in\mathcal{V}(\mathcal{G})} \sigma_{V,\bold{v'}}.
\end{equation}
The integrated density of states is its accumulation function $\mathcal{N}_V(E)=\widetilde\sigma((-\infty,E])$. It turns out to be equal to the normalized counting function of the eigenvalues of $H_{\mathcal{G},V}$. More precisely, denoting by $E_1\leq\dots\leq E_{\#\mathcal{V}(\mathcal{G})}$ the eigenvalues of $H_{\mathcal{G},V}$ (appearing with multiplicity), we have that
\begin{equation}
    \mathcal{N}_V(E)=\frac{\#\{i:E_i\leq E)\}}{\#\mathcal{V}(\mathcal{G})}.
\end{equation}

For infinite graphs, the IDS (or the desity of states) is not always well defined. An asymptotic average of the spectral measure exists when the system has some additional homogenuity, as, for instance, in the case of ergodic Schr\"odinger operators (see \cite{damanik2022one}). Later in the paper, we will encounter $\mathbb{Z}$-periodic graphs with periodic potentials. In this context the density of states is well defined as the average of the spectral measures over a period. The general problem of its existence is studied, for example, in \cite{gruber2007uniform}, \cite{veselic2008existence}, \cite{schwarzenberger2013integrated}. For $\mathcal{G}=\mathbb{Z}^d$, Bourgain and Klein (\cite{bourgain2013bounds}) introduced the density of states outer measure, whose existence is always guaranteed. This quantity is studied and slightly generalized in \cite{hislop2021dependence}. We will be particularly interested in the regularity of the IDS, which has also been an intensive subject of study on its own (see, for example, \cite{thouless1972relation}, \cite{craig1983log}, \cite{grigorchuk1899laplace}, \cite{fillman2022irreducibility}).

\subsection{Main results}

In the same spirit of \cite{avila_delocalization_2024}, we investigate how common delocalisation is in the space of potentials of Schr\"odinger Operators on large graphs. As we will show later on, for finite dimensional Schr\"odinger Operators, delocalisation and localisation are stable properties in the space of potentials, hence one cannot hope for density of delocalisation. Therefore, we look for existence of $\delta$-nets of potentials showing delocalisation for most of the spectral measures. Here $\delta$ has to be thought to get smaller when the size of the graph increases.

\begin{definition}
    We say that a family of graphs $\mathcal{F}$ has asymptotically dense delocalization if the following holds. For any $\lambda>0$, any $\varepsilon>0$ and any $\delta>0$ there exists $N_{\mathcal{F}}(\lambda,\va,\delta)>0$, such that, given any graph $\mathcal{G}\in\mathcal{F}$ with $\#\mathcal{V}(\mathcal{G})>N_{\mathcal{F}}$, for any potential $V:\mathcal{V}(\mathcal{G})\rightarrow[-\lambda,\lambda]$, there exists a potential $W$, with $\|V-W\|_{\infty}<\delta$, for which, for at least $\delta\#\mathcal{V}(\mathcal{G})$ vertices $v\in\mathcal{V}(\mathcal{G})$, the spectral measure $\sigma_{V,v}$ has no atom of weight larger than $\va$.
\end{definition}

Our main contribution is the following theorem, which represent a criterion for asymptotically dense delocalization for families of large graphs, based on the regularity of the IDS of their Schreodinger operators. We define formally the concept of asymptoyical uniform continuity of IDS in later sections.

\begin{theorem}
    Suppose that a family of graphs $\mathcal{F}$ has asymptotically $\omega$-uniformly continuous IDSs, for some modulus of continuity $\omega$ and bounded maximum degree. Then, $\mathcal{F}$ has asymptotically dense delocalization.
\end{theorem}

The proof is inspired by the methods used in \cite{avila_delocalization_2024} to prove the density of delocalization for Schr\"odinger operators in $\mathbb{Z}^d$. In our setting an iterative procedure and a careful study of the continuity of the delocalization property are required. This works around the lack of many measure theoretic tools that are available in the infinite dimensional setting.\newline

We denote the the $d$-dimesional box graph by $\mathcal{B}_N$, where its set of verices is $[N]^d=\{1,\dots,N\}^d$ and two vertices are connected by an edge if their distance is exactly 1. We refer to the operators $H_{\mathcal{B}_N,V}$ as box Schreodinger operators. We first  present the proof in the context of box Schr\"odinger operators, in order to provide explicit estimates, and we outline, in a second moment, how this  generalizes to our actual result.\newline

As a first corollary, we recover the main result from \cite{avila_delocalization_2024}. Our approach has the advantage of providing quantitative bounds, unavailbale earlier, due to the use of a correspondance principle.

\begin{theorem}
    For any $d>0$, the sequence of graphs $\mathcal{G}_N=[N]^d$ has asymptotically dense delocalization.
\end{theorem}

Finally, we describe a more general family of $\mathbb{Z}$-invariant graphs that falls under our criterion. This includes some notable family of graphs, for instance finite range box Schr\"odinger operators (that we will define precisely in the main body of the paper).

\begin{corollary}
    The family of the $r$-range Schr\"odinger operators on $[N]^d$ has asymptotically dense delocalization for every $r,d\in\N^+$.
\end{corollary}

\section{Results and Plan of Proof}

\subsection{Producing delocalization}

We will work with a slight modification of the spectral measures. Recalling that we denote the IDS by $\mathcal{N}$, we set
\begin{equation}
    \nu_{\bold{v}} = \mathcal{N}_*\sigma_{\bold{v}}.
\end{equation}
Let us give a description of $\nu_{\bold{v}}$ in the case where $H$ has simple spectrum. Let us denote by $\psi_i\in \ell^2(\mathcal{V})$ the unitary eigenvector corresponding to $E_i$. Then $\nu_{\bold{v}}$ is supported on the points of the form $i/\#\mathcal{V}$ and $\nu_{\bold{v}}(i/\#\mathcal{V})=\psi_i(\bold{v})^2$.\newline
The maximal size of the atoms of $\nu_{\bold{v}}$ is always larger than the one of $\sigma_{\bold{v}}$. Hence, the following definitions of delocalization imply $\ell^{\infty}$-spatial delocalisation.

\begin{definition}
    Given $\varepsilon,\eta>0$, $J\subset\R$ and a Borel measure $\mu$ on $\R$, we say that $\mu$ is $(\varepsilon,\eta,J)$-delocalized if $\mu(I)\leq\varepsilon$ for every interval $I\subset J$ such that $|I|\leq \eta$.\newline
    We just say that $\mu$ is $(\varepsilon,\eta)$-delocalized if it is $(\varepsilon,\eta,\R)$-delocalized.
\end{definition}

We say that $J$ is the region where $\mu$ is delocalized and that $\va$ and $\eta$ are, respectively, the size and the scale of the delocalization. 

\begin{definition}
    Given $\varepsilon,\eta>0$, $J\subset\R$ and a potential $V:\mathcal{V}\rightarrow [-\lambda,\lambda]$ we say that $V$ is $(\varepsilon,\eta,J)$-delocalized at $\bold{v}$ if $\nu_{\bold{v}}$ is $(\varepsilon,\eta,J)$-delocalized.\newline
    Given $B\subset\mathcal{V}$, we say that $V$ is $(\varepsilon,\eta,J)$-delocalized on $B$ if it is $(\varepsilon,\eta,J)$-delocalized at every $\n\in B$.
\end{definition}

A small-scale delocalization is easier to achieve, but it is harder to preserve under perturbations.\newline

The main step in the proof of our main result is the following lemma. It states that if a potential has some delocalization, via a small perturbation, we can produce a new potential, which has a stronger delocalization on a specific target region.\newline
Although stated for the box operators, this theorem holds for more general graphs, provided some bounds on the discontinuities of the IDS of the corresponding Schreodinger Operator.

\begin{theorem}\label{thm perturbation}
    Let $V:[N]^d\rightarrow[-\lambda,\lambda]$ be a potential and suppose that any potential $W$, such that $\|V-W\|_{\infty}<\delta$ is $(\varepsilon,\eta)$-delocalized on $B$.\newline
    Fix real numbers $\va'\geq \frac23 \va$, $0<\Delta<1$, $\eta'>0$ and an interval $J$ such that $|J|\leq \eta/2$. If
    \begin{equation}
        \frac1{\eta'}\gtrsim_{\lambda,d}  \frac{\eta}{\delta\va^2\Delta^2},
        \qquad\qquad N>\frac6{\va\eta'},\frac1{\eta\eta'},
    \end{equation}\label{hypotesis}
    then there exists $V_1:[N]^d\rightarrow[-\lambda,\lambda]$, with $||V-V_1||_{\infty}<\delta$, that is $(\varepsilon',\eta',J)$-delocalized on $B'\subset B$,  where $\#(B\backslash B')/N^d\leq\Delta $.
\end{theorem}

The proof exploits the technique behind the main theorem in \cite{avila_delocalization_2024}, which directly dealt with the infinite dimensional setting. In our setting, the argument does not provide immediately the desired degree of delocalization, since we cannot exploit many tools from measure theory that are available in the non-discrete case. Instead, we will reach the desired result by an iterative application of this result.\newline

The proof goes by contraddiction. Suppose that no potential $W$ in $B(V,\delta)$ shows $(\va',\eta',J)$-delocalization. Then the following must occur: $\nu_{W,\bold{n}}(J)<\va$ and $\nu_{W,\bold{n}}(I_{W})> \va'$ for some interval $I_W\subset J$ of size $\eta'$. This means that most of the measure of $J$ is concentrated on a much shorter interval $I_W$. We highlight how this leads to a contraddiction.\newline
We fix a function $\beta:[0,1]\rightarrow\mathbb{R}$, and we define the following function, from the space of potentials to $\mathbb{R}$,
\begin{equation}
    \Phi_{\beta}(V)=\int E\beta(\mathcal{N}_V(E))d\mathcal{N}_V(E).
\end{equation}
It is easy to check that if $\beta$ oscillates fast, then $\|\Phi_{\beta}\|_{\infty}$ is close to $0$.\newline
One can also see that, thanks to the Hellmann-Feynman formula, $\Phi_{\beta}$ admits a weak gradient $\Psi_{\beta}$, that can be expressed in terms of $\nu_{\bold{n}}$, as follows

\begin{equation}
    \Psi_{\beta}(V)[\n] = \int \beta(t) d\nu_{\n}(t).
\end{equation}

To reach the desired contraddiction we make a specific choice of $\beta$. We choose it to be supported on $J$ and with oscillation just slightly slower than $\eta'$. This ensures that $\Psi_{\beta}$ sees only the mass of $\nu_{\bold{n}}$ that sits on $J$, and at the same time there is no cancellation on $I_W$, because the oscillation is too slow on that extremely small scale. Since most of the mass on $J$ resists cancelation (since it sits on $I_{W}$), we conclude that $\Psi_{\beta}$ is always large.\newline
This, whenever $r$ is large enough, concludes the argument: a function with a large gradient, on a sufficiently large disk, must have a large oscillation (which $\Phi_{\beta}$ does not have).

\subsection{Stability of Delocalization}

Along the iterative application of theorem \ref{thm perturbation}, we need to be careful not to destroy the delocalization created with the previous steps on different regions. This is made possible by the next theorem, about stability of delocalization under perturbations.

\begin{theorem}\label{stability of delocalization}
    Let $\va'>\va>0$ and $\eta>\eta'>0$ be real numbers, and $V$ a potential that is $(\va,\eta,J)$-delocalized at $\n$.\newline
    If $W$ is a potential, 
    \begin{equation}
        \|W-V\|_{\infty}<e^{-\frac{c}{(\va'-\va)(\eta-\eta')}} \quad\text{and}\quad N > e^{\frac{c}{(\va'-\va)(\eta-\eta')}},
    \end{equation}
    then $W$ is $(\va',\eta',J')$-delocalized at $\n$, where $J'=\{x\in J : d(x,\partial J) > \frac12(\eta-\eta')\}$
\end{theorem}

This essentially follows from the fact that spectral measures depend continuously on the potential. To make this more evident, let us unpack the dependance of $\nu_{V,\bold{n}}$ on the spectral measures $\mu_{V,\bold{n}}$'s. We have that
\begin{equation}
    \nu_{V,\bold{n}}=(\mathcal{N}_V)_*\sigma_{V,\bold{n}}.
\end{equation}

Hence, the continuity of $\nu_{V,\bold{n}}$ in $V$ depends on: the continuity of $\sigma_{V,\bold{n}}$ in $V$, the dependence of $\mathcal{N}_V$ on $V$, and the regularity of $\mathcal{N}_V(E)$ as a function of $E$. Equation \ref{Nv is average} ensures that the second one boils down to the first one.\newline

The continuous weak-$*$ dependance of the spectral measures on the potential is a well-known fact; see, for example, \cite{damanik2022one}. However, it is in our interest to provide quantitative and uniform bounds. This is done studying the Stjlties transform of the spectral measure, exploiting the spectral theorem and concluding with an harmonic analysis argument.\newline
Similar results are obtained for some families of infinite graph in \cite{hislop2021dependence}.\newline

Our study of the continuity of the spectral measures and formula \ref{Nv is average} apply to general graphs, while the regularity of the IDS depends heavily on the graph structure. It is, in fact, the crucial property that determines if asymptotic dense delocalization holds for a given family of graphs. Let us be more specific about the required regularity in the next section.\newline

\subsection{The general criterion}

When $\mathcal{G}$ is a finite graph the IDS cannot be continuous, since it shows jump discontinuities of size at least $1/\#\mathcal{V}(\mathcal{G})$. We describe its regularity as being uniformly close to a uniform continuous function with a a specified modulus of continuity.

If $\omega:\mathbb{R}^+\rightarrow\mathbb{R}^+$ is a modulus of continuity and $f:\R\rightarrow\R$ a function that admits $\omega$ as modulus of continuity, we say, for short, that $f$ is $\omega$-uniformly continuous.\newline

\begin{definition}
    Let $\omega:\R^+\rightarrow\R^+$ be a modulus of continuity. We say that a family $\mathcal{M}$, of functions from $\R$ to $\R$, is $\va$-quasi $\omega$-uniformly continuous if for any $f\in\mathcal{M}$ there exists $g:R\rightarrow\R$ that is $\omega$-uniformly continuous and such that $\|f-g\|_{\infty}\leq\va$.
\end{definition}

\begin{definition}
    Given $\lambda>0$, we say that a family of graphs $\mathcal{F}$ has asymptotically $\omega$-uniformly continuous IDSs if for every $\va>0$, there exists $N_{\va}$ such that the family of IDSs
 \begin{equation}
     \mathcal{M}_{\mathcal{G},N_{\va}}=\{\mathcal{N}_{V,\mathcal{G}}: \mathcal{G}\in\mathcal{F}, \#\mathcal{V}(G)\geq N_{\va}, V:\mathcal{V}(G)\rightarrow[-\lambda,\lambda]\}
 \end{equation}
is $\va$-quasi $\omega$-uniformly continuous.
\end{definition}

When the IDS has this regularity the discussion of the two previous sections works without obstructions, providing our main result, possibly with different quantitative bounds, that depend only on $\omega$, the bound on the potential $\lambda$, and the maximal degree of the vertices of the family of graphs $k$.

\begin{theorem}
    Suppose that a family of graphs $\mathcal{F}$ has asymptotically $\omega$-uniformly continuous IDSs, for some modulus of continuity $\omega$ and bounded maximum degree. Then, $\mathcal{F}$ has asymptotically dense delocalization.
\end{theorem}

\subsection{Regularity of the IDS for the box operator}

To conclude the proof of asymptotically dense delocalization for the family of box graphs, we just need to prove some regularity for the IDS of their corresponding Schr\"odinger operators. This is already achieved in \cite{bourgain2013bounds} with a direct argument, but we present anyway an alternative proof that relies on the comparison between finite graphs and some infinite extension. This follows the common intuition that the IDS of an infinite dimensional operator is the limit of some finite rank truncated operators. As this is well know for the box operator, the novelty of our argument are quantitative uniform bounds.\newline

We will deal with the following modulus of continuity. When
\begin{equation}
    \omega(x)= 
    \begin{cases}
        \frac C{-\ln|x|} & \text{if } |x|<1/2\\
        +\infty & \text{otherwise},
    \end{cases}
\end{equation}
we replace the expression '$\omega$-uniformly continuous' with '$C$-$\log$ Holder continuous'.\newline

\begin{theorem}
    The sequence of graphs $\mathcal{G}_N=[N]^d$ has asymptotically $C$-$\log$ Holder continuous IDSs, where $C=C_{\lambda,d}$ depends only on $\lambda$ and $d$.
\end{theorem}

Let us introduce the infinite extension in question.\newline
Given a box Schr\"odinger Operator $H_V=H_{V,[N]^d}:\ell^2([N]^d)\rightarrow\ell^2([N]^d)$, we consider its potential $V:[N]^d\rightarrow [-\lambda,\lambda]$ and we extend it periodically to $V^{\infty}:\mathbb{Z}^d\rightarrow[-\lambda,\lambda]$, by $V^{\infty}(\n)=V(\n')$ if $\n\equiv\n'$ (mod $N$). The operator corresponding to the new potential $H^{\infty}_V=H_{V^{\infty},\mathbb{Z}^d}:\ell^2(\mathbb{Z}^d)\rightarrow\ell^2(\mathbb{Z}^d)$, is the periodic extension of $H_V$.\newline

For convenience, we denote by $\sigma_{V,\n}^{\infty}$ the spectral measure at site $\n\in\mathbb{Z}^d$ relative to $H^{\infty}_V$.\newline
A priori the existence of a well-defined density of states is not guaranteed for infinite dimension Schr\"odinger operators, nevertheless, thanks to the periodicity, we can simply set it to be equal to

\begin{equation}
    \tilde\sigma^{\infty}_V = \frac1{N^d} \sum_{\n\in[N]^d}\sigma_{V,\n}^{\infty}.
\end{equation}

Consequently, we define the IDS as $\mathcal{N}_V^{\infty}(E)=\tilde\sigma_n^{\infty}((-\infty,E])$.\newline
Being periodic, the potential $V^{\infty}$ can be seen as dynamically generated by an ergodic system. By \cite{craig1983log}, this implies that the IDS $\mathcal{N}_V^{\infty}$ is $C$-$\log$ Holder continuous, for $C$ that depends only on $\lambda$ and $d$.\newline

The conclusion follows from the fact that, when $N$ goes to $+\infty$, most of the spectral measures $\nu_{\n}$ get close to the measures $\nu_{\n}^{\infty}$, impling that $\|\mathcal{N}_V-\mathcal{N}_V^{\infty}\|_{\infty}\rightarrow 0$.\newline
This last statement is proved by the moments method, exploiting the spectral theorem.

\subsection{A more general family of Graphs}

Finally, we present a more general class of graphs that satisfies the hypotesis of our criterion for asymptotically dense delocalization.\newline
To show that they present regular IDSs we follow the same strategy of proof that we employed for the box graphs. Here, the key step is the regularity of IDS for the infinite extensions. We describe first these ifinite graphs and we recover our family as their truncations.\newline

We say that a graph $\mathcal{G}$ is $\mathbb{Z}$-invariant if there exist a graph automorphisms $f:\mathcal{V}(\mathcal{G})\rightarrow\mathcal{V}(\mathcal{G})$ that acts freely on $\mathcal{G}$. In addition, we say that $\mathcal{G}$ has finite section if $f$ has finitely many distinct orbits.\newline
We say that $\mathcal{G}$ is propagating, if it is $\mathbb{Z}$-invariant with finite section and if there exists a partial order $\leq$ on $\mathcal{V}(\mathcal{G})$ such that $f$ is a strictly increasing order automorphism with the two following properties:
\begin{itemize}
    \item $f(v)$ and $v$ are linked for any $v\in\mathcal{V}(\mathcal{G})$;
    \item for any neighbor $w$ of $v$, it holds that $f^{-1}(v)\leq w \leq f(v)$.
\end{itemize}

We say that a potential $V:\mathcal{V}(\mathcal{G})\rightarrow [-\lambda,\lambda]$ is periodic if it is $f^p$ invariant for some $p\in\mathbb{N}$. The combined $\mathbb{Z}$-periodicity of the graph and of the potential (with possibly different periods) ensures the periodicity of the spectral measures, implying the existence of the density of states and, consequentely, of the IDS.

\begin{theorem}\label{IDS regularity graphs}
    Let $\mathcal{G}$ be a $\mathbb{Z}$-invariant finite section propagating graph of maximal degree $k$. Let $V:\mathcal{V}(\mathcal{G})\rightarrow [-\lambda,\lambda]$ be a periodic potential, then the IDS of the associated Schr\"odinger operator $H_{\mathcal{G}, V}$ is $C_{\lambda,k}$-$\log$ Holder continuous.
\end{theorem}

The proof, which follows the original approach for the same result about box operators (\cite{thouless1972relation}, \cite{craig1983log}), is based on a version of the Thouless formula. We prove that $\int_{\mathbb{R}}\log|E'-E|d\tilde\sigma(E')=\gamma(E)$ for any $E\in\mathbb{C}$ with $Im(E)\neq0$, where $\gamma(E)$ is the some fo some Lyapunov exponents of an associated cocycle over $\mathbb{Z}$. For such $E$, $\gamma(E)$ is always non negative, implying, as desired, that the density of states $\tilde\sigma$ cannot be too concentrated.\newline

We say that $X\subset \mathcal{V}(\mathcal{G})$ is a fundamental domain for the action of $f$ if any orbit of $f$ intersects $X$ in exactly one vertex. A subset $Y\subset \mathcal{V}(\mathcal{G})$ is said to be convex if for any two vertices $v,z\in Y$, and any $w$ such that $v\leq w \leq z$, we have that  $w\in Y$.\newline
We denote by $\mathcal{G}_Y$ the subgraph of $\mathcal{G}$ induced by $Y$. We say that a graph of the form $\mathcal{G}_Y$, for $\mathcal{G}$ propagating and $Y$ convex, is a truncated propagating graph. We finally define the length of $\mathcal{G}_Y$ as the maximal number of disjoint fundamental domains contained in $\mathcal{G}_Y$.

\begin{theorem}\label{main results graph}
    Let $\mathcal{G}_N$ be a sequence of truncated propagating graphs with increasing lengths and bounded maximal degree. Then $\mathcal{G}_N$ has asymptotically dense delocalization.
\end{theorem}

Notice that this result does not require any condition of the size of the section of the graph. For instance the result about box Schr\"odinger operator can be recovered as a corollary. Another notable family of operators that falls under this theorem is the one of finite range box Schr\"odinger operators, see \cite{ge2024almost}.\newline

Finally, we have to highlight that this result can be obtained also by the approach from \cite{bourgain2013bounds}, which does not rely on (nor show) the Thouless formula for such operators.

\section{Real Analysis Preliminaries}

For our argument to work it is crucial to have a uniform (and quantitative) control on the stability of delocalisation of measures. Since this property is easily captured by suitable test functions, it will be our concern to metrize the weak-$*$ topology on the space of Borel probability measures on $\mathbb{R}$.\newline
It is well-know that this is achieved by the Fortet-Mourier metric, defined as
\begin{equation}
    d_{FM}(\mu,\nu) = \sup_{\phi\in \text{Lip}^1({\mathbb{R}}), \|\phi\|_{\infty}\leq1} \left| \int_{\mathbb{R}} \phi d\mu - \int_{\mathbb{R}} \phi d\nu \right|.
\end{equation} 

The Fortet-Mourier metric does not control directly the interaction of measures with test functions that are less regualr than Lipschitz. For this reason we prefer to introduce the following one.

\begin{equation}
    d(\mu,\nu) = \inf \left\{ \delta>0 : \|\widehat{\mu-\nu}\|_{L^1([-1/\delta,1/\delta])}   < \delta     \right\}
\end{equation}

\begin{remark}
    For probability measures sharing a common bounded support, the metric $d$ is Holder equivalent to the Fortet-Mourier metric $d_{FM}$.
\end{remark}

\begin{proof}
    Suppose that $d_{FM}(\mu,\nu)<\delta$. Let $\phi:\mathbb{R}\rightarrow[-1,1]$ be an $L$-Lipschitz function. Applying the definition of $d_{FM}$ to $\phi/L$ we obtain that $\int_{\mathbb{R}}\phi d(\mu-\nu)\leq L \delta$. Noticing that $e^{-ix\xi}$ is a $|\xi|$-Lipschitz function of $x$ we get that $|(\widehat{\mu-\nu})(\xi)|\leq |\xi|\delta$. Hence $\|\widehat{\mu-\nu}\|_{L^1([-1/\va,1/\va])}\lesssim \delta/\va^2$. Picking $\delta = \va^3$ we prove the first inequality.\newline
    The following lemma \ref{fourier controls test} gives the opposite one when applyied with $\omega_{\phi}(t)=t$.
\end{proof}

As anticipated, the next lemma shows that the distance $d$ provides direct control on the integral of uniformly continuous test functions, of any modulus of continuity.\newline
In the following statements we will assume all modulus of continuity $\omega$, to satisfy $\omega(t)\geq t$.\newline

Let $f\in\mathcal{C}^2(\mathbb{R})$ be a non-negative function, supported on $[-1,1]$, with $\|f\|_{L^1(\mathbb{R})}=1$.\newline
We remark that
\begin{equation}
    |\widehat{f}(\xi)|\lesssim \frac{ 1}{|\xi|^2}.
\end{equation}
Fuerthermore we define $f_r:\R\rightarrow\R$ by $f_r(t)=\frac1r f(\frac tr)$.

\begin{lemma}\label{fourier controls test}
    Let $\mu,\nu$ be Borel probability measures on $\R$. Let $\phi:\R\rightarrow\R$ a continuous function, supported on [-M,M], with $\|\phi\|_{\infty}\leq 1$. 
    Let $\omega_{\phi}:\R^+\rightarrow\R^+$ be its modulus of continuity. Then
    \begin{equation}
        \left| \int_{\R}\phi d\mu - \int_{\R}\phi d\nu \right| \lesssim_{M}   \omega_{\phi}\left( d(\mu,\nu)^{\frac13}  \right)
    \end{equation}
\end{lemma}

\begin{proof}
Fix $\delta>d(\mu,\nu)$, set $\sigma=\mu-\nu$ and let $r>0$ be real number.\newline
By the Plancherel identity we can argue that
\begin{align}
    \left|\int_{\R} \phi d\sigma\right| &\leq \left|\int_{\R} (\phi - \phi*f_r)d\sigma\right| + \left|\int_{\R} (\phi*f_r) d\sigma\right| \\
    &\lesssim \|\phi - \phi*f_r\|_{\infty} + \int_{\R} |\widehat{\phi}(\xi)||\widehat{f_r}(\xi)||\widehat{\sigma}(\xi)|d\xi\\
    &\lesssim_{M} \omega_{\phi}(r) +  \int_{-1/\delta}^{1/\delta} |\widehat{\sigma}(\xi)|d\xi + \int_{\R\setminus[-1/\delta,1/\delta]}|\widehat{f_r}(\xi)|d\xi \\
    &\lesssim_{M} \omega_{\phi}(r) +  \int_{-h}^h |\widehat{\sigma}(\xi)|d\xi + \frac{1}{r}\int_{r/\delta}^{\infty}\frac{ 1}{|\xi|^2}d\xi.\\
    &\lesssim_{M}   \omega_{\phi}(r) + \delta + \frac{\delta}{r^{2}}.
\end{align}
Setting $r=\delta^{1/3}$ and letting $\delta\rightarrow d(\mu,\nu)$, we obtain the desired conclusion.
\end{proof}

Finally, we show that measures, which are close in the metric $d$, must have close cumulative functions, provided some regularity for at least one of two.\newline
This will be relevant as the IDS is the cumulative distribution of the averaged specral measure.

\begin{lemma}\label{Fourier controls cumulative}
    Let $\sigma_1$ and $\sigma_2$ be two probability measures supported on $[-M,M]$ and $\mathcal{N}_1$ and $\mathcal{N}_2$ their cumulative functions. Suppose that $\mathcal{N}_1$ is $\va$-almost $\omega$-uniformly continuous.
    Then,
    \begin{equation}
        \|\mathcal{N}_1 - \mathcal{N}_2\|_{\infty} \lesssim_M  \omega\left(d(\mu,\nu)^{1/6}\right) + \va.
    \end{equation}
\end{lemma}

\begin{proof}
    Fix $E\in\R$. Given $L>1$, we define  $\alpha:\R\rightarrow\R$ as 
    \begin{equation}
        \alpha(t)=
        \begin{cases}
            1 &\text{if } t\leq E\\
            1-L(t-E) &\text{if } t_0<t<E+1/L\\
            0 &\text{if } t\geq E+1/L.
        \end{cases}
    \end{equation}
    We can consider $\tilde{\alpha}:\R\rightarrow\R$ with the following properties: it coincides with $\alpha$ on $[-M,M]$, it is $L$-Lipschitz and is supported on $[-M-1,M+1]$. We can now apply lemma \ref{fourier controls test} to $\tilde{\alpha}$ and $\sigma_1$ and $\sigma_2$. We get, for $r>1$,

    \begin{align}
        \mathcal{N}_{1}(E) &\leq  \int \alpha d\sigma_{1}\\
        &\leq \int \alpha d\sigma_{2} + C_M L \cdot d(\mu,\nu)^{1/3}\\
        &\leq \mathcal{N}_{2}(E+1/L) +  C_M L \cdot d(\mu,\nu)^{1/3}\\  
        &\leq \mathcal{N}_{2}(E) + 2\va + C'_M \left( \omega(1/L) + L \cdot d(\mu,\nu)^{1/3}\right).
        \label{equation continuity}
    \end{align}
    The reverse inquality can be proved in an analogous way.\newline
    Finally, setting $L=d(\mu,\nu)^{-1/6}$ gives us the desired result. 
\end{proof}

\section{Fourier Analysis of Spectral Measures}

In this section, we prove some continuity properties of spectral measures. This will be done in terms of the metric $d$, introduced in the previous section. To be more specific, we first prove that spectral measure of box Schr\"odinger operators depend Holder-continuously on the potential. Then we prove that, in the limit $N\rightarrow\infty$, the density of states of the box operator approximates the one of its periodic extension, and the speed of convergence is again polynomial in $N$.

\begin{lemma}\label{Fourier - dependance on potential}
    Let $V,W:[N]^d\rightarrow[-\lambda,\lambda]$ be two potentials. Then, for every $\n\in[N]^d$,
    \begin{equation}
        d({\sigma}_{V,\n},{\sigma}_{W,\n}) \lesssim \|V-W\|_{\infty}^{1/4}
    \end{equation}
    and
    \begin{equation}
        d(\tilde{\sigma}_V,\tilde{\sigma}_W) \lesssim \|V-W\|_{\infty}^{1/4}.
    \end{equation}    
\end{lemma}

\begin{proof}
Let us set $\|V-W\|_{\infty}= \va$ and $\sigma_{\n}=\sigma_{V,\n}-\sigma_{W,\n}$ and denote by $g(t)=\frac1{1+t^2}$ the Cauchy function. 
Given $E=x+iy\in\mathbb{H}$, by the spectral theorem and the resolvent identity, we get
\begin{align*}
    \sigma_{\n} * g_y(x) &= -\text{Im}\left(\int_{\R} \frac 1{t-E}d\sigma_{\n}(t)\right)\\
    &\leq \left|\int_{\R} \frac 1{t-E}d\sigma_{\n}(t)\right|\\
    &=\left|\left\langle \left((H_V-E)^{-1}-(H_W-E)^{-1})\right)\delta_n,\delta_n \right\rangle\right|\\
    &=\left|\left\langle \left((H_V-E)^{-1}(H_W-H_V)(H_W-E)^{-1})\right)\delta_n,\delta_n \right\rangle\right|\\
    &\leq \va d\left(E,\sigma(H_V))^{-1}d(E,\sigma(H_W)\right)^{-1}\\
    &\leq \frac{\va}{d\left(x,[-2d-\lambda,2d+\lambda]\right)^2+y^2}.
\end{align*}

Integrating in $x$, we get

\begin{equation}
    \|\sigma_{\n} * g_y\|_1 \leq \va\left(\frac{2\lambda+4d}{y^2}+\frac{\pi}{y}\right),
\end{equation}
and it follows that

\begin{equation}
    \|\widehat{\sigma_{\n} * g_y}\|_{\infty}\leq \va\left(\frac{2\lambda+4d}{y^2}+\frac{\pi}{y}\right).
\end{equation}
Recalling that $\widehat{g_y}(\xi)=e^{-y|\xi|}$, we deduce that

\begin{equation}
    |\widehat{\sigma}_{\n}(\xi)| = |\widehat{\sigma_{\n} * g_y}(\xi)| |\widehat{g_y}(\xi)|^{-1} \leq \va\left(\frac{2\lambda+4d}{y^2}+\frac{\pi}{y}\right) e^{y|\xi|}
\end{equation}
and choosing $y=1/|\xi|$, we get that
\begin{equation}
        \left|\widehat{\sigma_{\n}}(\xi)\right| \lesssim_{\lambda,d}\va\left(|\xi|^2+|\xi|\right).
    \end{equation}
Integrating of over $[-1/\delta,1/\delta]$, we obtain
\begin{equation}
        \|\widehat{{\sigma}_{\n}}\|_{L^1([-1/\delta,1/\delta])}  \lesssim_{\lambda,d}\va \delta^{-3}.
    \end{equation}
Setting $\delta=\va^{1/4}$, the first conclusion follows from the assumptions on $\va$. Averaging over $\n\in[N]^d$, the second one follows directly.
\end{proof}

We now focus on the approximation of spectral measure with the ones of the periodicly extended operator. This will be done through the moments method.\newline
We remark that the following proof exploits the geometry of $\mathbb{Z}^d$. This happens in two ways: we rely on the fact that any box extends periodically to the full lattice, and, in second place, we use that $\mathbb{Z}^d$ is amenable (i.e. the boundary of any large box has negligible volume when compared to the box itself).\newline

Given a Borel measure $\mu$ on $\mathbb{R}$, we denote its $k$-th moment by $m_k(\mu)=\int_\mathbb{R} x^k d\mu(x)$.\newline
If the potential $V$ is sampled from $[-\lambda,\lambda]$, both in the box and in the periodic case, for every site $\n$, the spectral measure $\sigma_{V,\bold{n}}$ of $H_V$ is supported on $[-\lambda-2d, \lambda+2d]$. 

\begin{lemma}\label{Fourier - periodic vs box}
    Let $V:[N]^d\rightarrow[-\lambda,\lambda]$ be a potential. Then
    \begin{equation}
        d(\tilde{\sigma}_{V},\tilde{\sigma}_{V}^{\infty}) \lesssim N^{-1/3}.
    \end{equation}  
\end{lemma}

\begin{proof}
    Let us fix $K\in\mathbb{N}$. Let $\n\in[N]^d$ be a site, such that $\text{d}(\n,\partial[N]^d)\geq K$.\newline
    
    We start by showing that the spectral measures $\sigma_{V,\n}$ and $\sigma^{\infty}_{V,\n}$ have the same first $K$ moments. Indeed, by the spectral theorem,
    \begin{equation}
    \int_{\R}x^k d\sigma_{V,\n}(x) = \langle H_V^k\delta_{\n},\delta_{\n}\rangle = \langle H_{V^{\infty}}^k\delta_{\n},\delta_{\n}\rangle = \int_{\R}x^k d\sigma^{\infty}_{V,\n}(x),
\end{equation}
where the second equality holds because $H_V^k$ and $H_{V^{\infty}}^k$ are $k$-range operators and $\n$ is at greater distance from the boundary (which is where $H_V^k$ and $H_{V^{\infty}}^k$ disagree).\newline

We now set $\sigma_{\n}=\sigma_{V,\n}-\sigma_{V,\n}^{\infty}$. Our next goal will be to find a pointwise bound for $\widehat\sigma_{\n}$.\newline
We recall that $m_k(\sigma_{\n})=0$ for $k\leq K$, and $|m_k(\sigma_{\n})| \leq 2 \cdot (2d+\lambda)^k$ otherwise.

\begin{equation}
    \widehat{\sigma_{\n}}(\xi) = \int_{\R} e^{-i\xi x} d\sigma_{\n}(x)
    = \int_{\R} \sum_{k=0}^{\infty} \frac{(-i\xi x)^k}{k!} d\sigma_{\n}(x)
    = \sum_{k=0}^{\infty} \frac{(-i)^k m_k(\sigma_{\n})}{k!}\xi^k.
\end{equation}

It follows that
\begin{equation}
    |\widehat{\sigma_{\n}}(\xi)| \leq \sum_{k>K} \frac{|m_k(\sigma_{\n})|}{k!}|\xi|^k \leq 2 \frac{(\lambda+2d)^K}{K!}|\xi|^K e^{(\lambda+2d)|\xi|} 
\end{equation}

We now set $h = \frac K{e^3(\lambda+2d)}$ and focus on $\xi\in[-h,h]$ (in order to estimate $\|\widehat\sigma_{\n}\|_{L^1([-h,h])}$).
\begin{align}
    |\widehat{\sigma}(\xi)| &\leq 2 \left( \frac{|\xi| (\lambda+2d)e}{K} \right)^K e^{(\lambda+2d)|\xi|}\\
    &\leq 2 \left( \frac{|\xi| (\lambda+2d)e^2}{K} \right)^K \frac{e^{(\lambda+2d)|\xi|}}{e^{(\lambda+2d)h}}\\
    &\leq 2 e^{-K} e^{(\lambda+2d)(|\xi|-h)}.
\end{align}

Integrating in $\xi$ over $[-h,h]$, we get that

\begin{equation}
    \|\widehat\sigma_{\n}\|_{L^1([-h,h])} \lesssim_{\lambda,d} e^{-K} \lesssim_{\lambda,d} e^{-h}.
\end{equation}
\newline
For the sites $\n\in[N]^d$ with $\text{d}(\n,\partial[N]^d) < K$, we settle for the following estimate

\begin{equation}
    \|\widehat\sigma_{\n}\|_{L^1([-h,h])} \leq 2h \cdot \|\widehat\sigma_{\n}\|_{L^{\infty}(\mathbb{R})} \lesssim h.
\end{equation}

Recalling that the number of sites at distance less than $K$ from the boundary is less than $2K/N$ of the total number of sites, and averaging over $\n\in[N]^d$, we get that 

\begin{equation}
        \|\widehat{\tilde{\sigma}_{V} - \tilde{\sigma}_{V}^{\infty}}\|_{L^1([-h,h])} \lesssim_{\lambda,d} \frac{h^2}{N} + e^{-h}.
    \end{equation}  
Setting $h=1/\delta$, and $\delta=N^{-1/3}$, the conclusion follows.
\end{proof}

\section{Regularity of IDS}\label{Regularity of IDS}

In this section, we prove that, provided that the box size $N$ is large enough, the IDS of the box operator has a certain degree of regularity. In particular it is uniformly close to a uniformly continuous function. The modulus of continuity is universal and depends only on the dimension $d$ and the bound on the potential $\lambda$.\newline

The first key element of the proof is the regualrity of the IDS of the priodic extension operator. Since $V^{\infty}$ is dynamically generated by an ergodic system, the IDS of $H_{V^{\infty}}$ is log-Holder continuous, as proved in \cite{craig1983log}.

\begin{theorem}[Craig-Simon,1983]\label{log continuity}
    For any $\lambda>0$ and $d\in\mathbb{N}^+$ there exist $C=C_{\lambda,d}$ such that, for any potential $V:[N]^d\rightarrow[-\lambda,\lambda]$ then $\mathcal{N}_V^{\infty}$ is $\omega_C$-uniformly continuous.
\end{theorem}

The second and final step of the proof is the fact that the IDS of the box operator gets uniformly close to the IDS of its periodic extension as the box size $N$ increases. This is a direct consequence of the closness of the spectral measure of the two operators. 

\begin{theorem}\label{L infty ids}
    There exists a constant $c=c_{\va,d}$ such that, for any $0<
    \delta<1$, if
    \begin{equation}
        N \geq e^{{c}/{\va}},
    \end{equation}
    then for any potential $V:[N]^d\rightarrow[-\lambda,\lambda]$, we have that $\|\mathcal{N}^{\infty}_V-\mathcal{N}_V\|_{\infty}<\va$.
\end{theorem}

\begin{proof}
    It follows from Lemma \ref{Fourier - periodic vs box} and Lemma \ref{Fourier controls cumulative}, exploiting the regularity of $\mathcal{N}^{\infty}_V$, guaranteed by theorem \ref{L infty ids}.
\end{proof}

\begin{remark}\label{almost UC}
    In more general terms, we just proved that, whenever $N\geq e^{{c}/{\va}}$, $\mathcal{N}_V$ is $\va$-almost $\omega_C$-uniformly continuous.
\end{remark}

\section{Stability of delocalization}\label{Proof of stability of delocalization}

We are now set to prove that the stability of delocalization of spectral measures is a stable property under perturbations of the potential. Delocalization can be detected by Lipschitz test function, where the Lipschitz constant is larger, the smaller the scale of the delocalization is. The main theorem of the section will therefore follow dirctly from the next lemma, which states that the measures $\nu_{\n}$ depend continuosly in the potential (where the continuity is understood in terms of testing against $L$-Lipschitz functions).

\begin{lemma}\label{stability tested against lip}
    Fix $\lambda>0$ and $d\in N^+$. There exists $c=c_{\lambda,d}$ such that, given real numbers $E>0$ and $L>1$, if
    \begin{equation}
        N > e^{c \frac{L}{E}}
    \end{equation}
    the following holds. For any $L$-Lipschitz function $\beta:[0,1]\rightarrow\R$ with $\|\beta\|_{\infty}\leq 1$, and any two potentials $V,W:[N]^d\rightarrow[-\lambda,\lambda]$ with
    \begin{equation}
        \|W-V\|_{\infty}<e^{-c \frac{L}{E}},
    \end{equation}
    for any $\n\in[N]^d$ we have 
    \begin{equation}
        \left| \int \beta d\nu_{V,\n} - \int \beta d\nu_{W,\n} \right| < E.
    \end{equation}
\end{lemma}

\begin{proof}

Remark \ref{almost UC} (essentially theorem \ref{L infty ids}), if $c$ is large enough, ensures the existence of a $\omega_C$-uniformly continuous function $f_W$ such that $\|\mathcal{N}_W-f_W\|_{\infty}<E/6L$.\newline

Recalling that $\int \beta d\nu_{V,\n}=\int \beta\circ \mathcal{N}_V d\sigma_{V,\n}$, we can write

\begin{align}\label{Breaking down the estimate}
    \int \beta d\nu_{V,\n} - \int \beta d\nu_{W,\n} &= \int \beta\circ \mathcal{N}_V d\sigma_{V,\n} - \int \beta\circ \mathcal{N}_W d\sigma_{V,\n}\\
    &+ \int \beta\circ \mathcal{N}_W d\sigma_{V,\n} - \int \beta\circ f_W d\sigma_{V,\n}\\
    &+\int \beta\circ f_W d\sigma_{V,\n} - \int \beta\circ f_W d\sigma_{W,\n}\\
    &+ \int \beta\circ f_W d\sigma_{W,\n} - \int \beta\circ \mathcal{N}_W d\sigma_{W,\n}.
\end{align}

Since $\|\mathcal{N}_W-f_W\|_{\infty}<E/6L$ and $\beta$ is $L$-Lipschitz, we have that $\|\beta\circ\mathcal{N}_W-\beta\circ f_W\|_{\infty}<E/6$. Given that $\sigma_{V,\n}$ and $\sigma_{W,\n}$ are probability measures, we immediately deduce that

\begin{equation}\label{ineq1}
    \left|\int \beta\circ \mathcal{N}_W d\sigma_{V,\n} - \int \beta\circ f_W d\sigma_{V,\n}\right|,
    \left|\int \beta\circ f_W d\sigma_{W,\n} - \int \beta\circ \mathcal{N}_W d\sigma_{W,\n} \right|
    < E/6.
\end{equation}

Now we focus on the first term. By theorem \ref{Fourier - dependance on potential} and provided that $c$ was chosen large enough, we get that $d(\tilde\sigma_V,\tilde\sigma_W)<e^{-C\cdot6L/E}$. Recalling that $\mathcal{N}_W$ is $(C\cdot E/6L)$-almost $\omega_C$-uniformly continuous, lemma \ref{Fourier controls cumulative} ensures that $\|\mathcal{N}_V-\mathcal{N}_W\|_{\infty}<E/2L$. Proceeding just as above, we get that

\begin{equation}\label{ineq2}
    \left| \int \beta\circ \mathcal{N}_V d\sigma_{V,\n} - \int \beta\circ \mathcal{N}_W d\sigma_{V,\n} \right| < E/2.
\end{equation}

We are then left to deal with the third term.\newline
Again, theorem \ref{Fourier - dependance on potential} tells us that 
$d(\tilde\sigma_{V,\n},\tilde\sigma_{W,\n}) < e^{-C\cdot6L/E}$. Noticing that $\beta\circ f_W$ is $L\cdot\omega_C$-uniformly continuous, lemma \ref{fourier controls test} allows us to conclude that

\begin{equation}\label{ineq3}
    \left|\int \beta\circ f_W d\sigma_{V,\n} - \int \beta\circ f_W d\sigma_{W,\n}\right| < E/6.
\end{equation}

The combination of the inequalities \ref{ineq1}, \ref{ineq2} and \ref{ineq3} concludes the proof.

\end{proof}

\begin{theorem}\label{stability of delocalization}
    Fix $\lambda>0$ and $d\in N^+$. There exists $c=c_{\lambda,d}>0$ such that, given real numbers $\va'>\va>0$ and $\eta>\eta'>0$, if
    \begin{equation}
        N > e^{\frac{c}{(\va'-\va)(\eta-\eta')}},
    \end{equation}
    the following holds.
    If $V:[N]^d\rightarrow[-\lambda,\lambda]$ is $(\va,\eta,J)$-delocalized at $\n$, then any $W:[N]^d\rightarrow[-\lambda,\lambda]$, such that
    \begin{equation}
        \|W-V\|_{\infty}<e^{-\frac{c}{(\va'-\va)(\eta-\eta')}},
    \end{equation}
    is $(\va',\eta',J')$-delocalized at $\n$, where $J'=\{x\in J : d(x,\partial J) > \frac12(\eta-\eta')\}$
\end{theorem}

\begin{proof}
    Let $I\subset J$ be an interval such that $|I|\leq\eta$. Let us set $I'=\{x\in I : d(x,\partial I) > \frac12(\eta-\eta')\}$. We then just need to prove that $\nu_{W,\n}(I')\leq \va'$.\newline
    Let us define $\beta:[0,1]\rightarrow\R$ by
    \begin{equation}
        \beta(t)=
        \begin{cases}
            0 & \text{if } t\notin I,\\
            1-\frac{2d(t,I')}{\eta-\eta'}  & \text{if } t\in I.
        \end{cases}
    \end{equation}
    Notice that $\beta$ is $2/(\eta-\eta')$-Lipschitz. We can therefore apply lemma \ref{stability tested against lip} to $\beta$ and $\nu_{V,\n}-\nu_{W,\n}$, with $L=2/(\eta-\eta')$ and $E=\va'-\va$. As a consequence, we get that
    \begin{equation}
        \nu_{W,\n}(I') \leq \int \beta d\nu_{W,\n} \leq \int \beta d\nu_{V,\n} + (\va'-\va) \leq \va'.
    \end{equation}
\end{proof}

\section{Perturbations producing delocalization}\label{Perturbations producing delocalization}

In this section, we prove theorem \ref{thm perturbation}, which is the core of our whole argument. It will enable us to upgrade a weaker global delocalization to a stronger delocalization on a specific target region (of a small enough size), by perturbing the potential slightly.\newline

The proof relies on two ausiliary function, that we proceed to define.\newline
For $\beta:[0,1]\rightarrow\R$ continuous function, we define

\begin{gather}
    \Phi_{\beta}(V)=\int E\beta(\mathcal{N}_V(E))d\mathcal{N}_V(E),\\
    \Psi_{\beta}(V)[\n] = \int \beta(\mathcal{N}_V(E))d\sigma_{\n}(E) = \int \beta(t) d\nu_{\n}(t).
\end{gather}

Let us denote by $\Omega_N\subset [-\lambda,\lambda]^{[N]^d}$ the set of potentials whose correspondig Schr\"odinger box operators have simple spectrum. This is an open and dense full measure set, complement of an algebraic one. Any line is either disjoint from $\Omega_N$, or its intersection with $\Omega_N$ has full one-dimensional measure. \newline
For $V\in\Omega_N$, we have a more concrete description of $\Phi_{\beta}(V)$ and $\Psi_{\beta}(V)$, as they respectively coincide with

\begin{gather}
    \bar{\Phi}_{\beta}(V)=\frac1{N^d}\sum_{i=1}^{N^d}E_i\beta(i/{N^d}),\\
    \bar{\Psi}_{\beta}(V)[\n] = \sum_{i=1}^{N^d} \beta(i/N^d)\sigma_{\n}(E_i) = \sum_{i=1}^{N^d} \beta(i/N^d)\nu_{\n}(i/N^d).
\end{gather}
The function $\bar{\Phi}_{\beta}$ is well defined even for $V\notin\Omega_N$, differently from $\bar{\Psi}_{\beta}$. Moreover, $\bar{\Phi}_{\beta}$ is continuous on the whole space of potentials $[-\lambda,\lambda]^{[N]^d}$.\newline

Critically, it turns out that $\Psi_{\beta}$ is, up to a factor $1/N^d$, the weak gradient of $\Phi_{\beta}$.

\begin{lemma}\label{phi differentiability}
    Let be $V,W:[N]^d\rightarrow [-\lambda,\lambda]$ two potentials and $\beta:[0,1]\rightarrow\R$ a continuous function. Suppose that $V\in\Omega_N$ 
    Then
    \begin{equation}\label{Hellman-Feyman}
        \bar\Phi_{\beta}(V+W)-\bar\Phi_{\beta}(V) = \int_0^1 \frac 1{N^d} \sum_{\n\in[N]^d} \Psi_{\beta}(V+tW)[\n]W[\n]dt.
    \end{equation}
\end{lemma}

\begin{proof}
    If both $V$ and $V+W$ belong to $\Omega_N$ It follows from the Hellmann-Fayman formula (see \cite{avila_delocalization_2024}). If not, let $s_n\rightarrow0$ be a sequence of real numbers such that $V+s_nW,V+(1+s_n)W\in\Omega_N$. It exists since $\{t\in\mathbb{R}: V+tW\in \Omega_N \}$ is non empty, and therefore has full measure. Letting $n$ go to $\infty$, by continuity of $\bar\Phi_{\beta}$ and of the right hand side (since $\Psi_{\beta}$ is bounded), we conclude.
\end{proof}


Let us define $\rho_{\beta}(s)=\sup_{|t-t'|<s}|\beta(t)-\beta(t')|$.\newline

As discussed in the introduction, we prove theorem \ref{thm perturbation} by contraddiction. If it didn't hold it would be possible to find a $\beta$, for which $\|\Phi_\beta\|_{\infty}$ is small, while, for every $W$ in a neighborhood of $V$, most coordinates of $\Psi_{\beta}(W)$ are big. This leads to the desired contraddiction, thanks to a "gradient descent" argument.\newline
This last step works only if $\Psi_{\beta}$, the gradient of $\Phi_{\beta}$, is regular enough, and the needed regularity is provided by the next lemma.

\begin{lemma}\label{psi continuity}
    Let $\beta:[0,1]\rightarrow\R$ be a continuous function. Let us set $\va_0=\rho_{\beta}(1/N)$. Then there exists $\delta_0>0$ such that $|\Psi_{\beta}(V)[\n]-\Psi_{\beta}(V')[\n]|<\va_0$ for any two potentials such that $||V-V'||<\delta_0$.
\end{lemma}

\begin{proof}
    The pointwise statement is proved in \cite{avila_delocalization_2024}. Our uniform version follows from the compactness of the space of potentials.
\end{proof}

The last preliminary to the proof is the following lemma, that states that the average of the measures $\nu_{\n}$ is not too far from the Lebesgue measure on $[0,1]$.

\begin{lemma}\label{average nu similar to Lebesgue}
For any interval $I\subset\mathbb{R}$,
    \begin{equation}
        \frac1{N}^d \sum_{\n\in[N]^d} \nu_{\n}(I) \leq  |I| + 1/N.
    \end{equation}
\end{lemma}

\begin{proof}
    The measure $1/N^d \sum_{\n\in[N]^d} \nu_{\n}$ has the form $\sum_{i=1}^{k} (x_i-x_{i-1})\delta_{x_i}$ where $x_0=0$ and $x_i-x_{i-1}\in(0,1/N]$. In particular $(x_i-x_{i-1})N^d$ is the dimension of the $i$-th eigenspace of $H$, which we know being always bounded by $N^{d-1}$.
\end{proof}

We can now prove the main result of this section.

\begin{remark}
    In the following statement the condition $\va'\geq \frac23 \va$ can be improved to $\va'\geq \alpha \va$ for any $\alpha>\frac12$. Moreover one has that $C_{\lambda,\delta}\rightarrow\infty$, as $\alpha\rightarrow\frac12$.
\end{remark}

\begin{theorem}\label{thm perturbation}
    Fix $\lambda>0$ and $d,N\in\mathbb{N}^+$. There exists a constant $C=C_{\lambda,d}$, such that the following holds. Let $V:[N]^d\rightarrow[-\lambda,\lambda]$ be a potential in $\Omega_N$. Let $\varepsilon>0$ and $\eta>0$ be real numbers. Let $B\subset[N]^d$ and $\delta>0$. Suppose that $V':[N]^d\rightarrow[-\lambda,\lambda]$ is $(\varepsilon,\eta)$-delocalized on $B$, whenever $||V-V'||_{\infty}<\delta$. Fix a real numbers $\va'\geq \frac23 \va$, $0<\Delta<1$ and an interval $J$ such that $|J|\leq \eta/2$.\newline
    Let $\eta'>0$. If
    \begin{equation}
        \frac1{\eta'}\geq C  \frac{\eta}{\delta\va^2\Delta^2},
        \qquad\qquad N>\frac6{\va\eta'},\frac1{\eta\eta'}.
    \end{equation}\label{hypotesis}
    Then there exists $V':[N]^d\rightarrow[-\lambda,\lambda]$ such that $||V-V'||_{\infty}<\delta$ and that $V'$ is $(\varepsilon',\eta',J)$-delocalized on a suitbale $B'\subset B$,  such that $\#(B\backslash B')/N^d\leq\Delta $.
\end{theorem}

\begin{proof}[Proof]
The proof goes by contraddiction. We therefore suppose that for any potential $V'$ such that $||V-V'||_\infty<\delta$ there exists a set of sites $C_{V'}\subset B$, with $\#C_W\geq \Delta N^d$, such that for every $\n\in C$ there exists an interval $I_{V',\n}\subset J$ such that $|I_{V',\n}|\leq\eta'$ and $\nu_{\n}(I_{V',\n})\geq \va'$. Recall that $\nu_{\n}(J)\leq\va$ for any $\n\in C_{V'}$ (and we can assume $|J|=\eta/2$). \newline

We now construct the function $\beta$ for which the ausiliary function $\Phi_{\beta}$ and $\Psi_{\beta}$ will make the contradiction evident. $\beta$ will be supported on the target region for delocalization and its frequency of oscillation has to be chosen carefully.\newline

We pick an integer $k$ such that
\begin{equation}
   16 \frac{\eta\eta'}{\va\Delta} < \frac k{N^d} <  32 \frac{\eta\eta'}{\va\Delta}
\end{equation}

We cover $J$ with $2M$ consecutive intevals of  length $k/N^d$ and we label them as $J_i$ for $i=1,\dots,2M$. We take $M$ to be minimal. Let $J'=\bigcup_{i=1}^{M}J_i$. Notice that $|J'|<\eta$, and, therefore,

\begin{equation}
    M<\frac12 \frac {\va\Delta}{16\eta'}.
\end{equation}
Now we define
\begin{equation}
\beta(t)=
\begin{cases}
    (-1)^i\min\{d(t,\partial J_i)/\eta',1\} & \text{for } t\in J_i,\\
    0 & \text{otherwise}.
\end{cases}
\end{equation}

We can immediately notice that $\beta$ oscillates fast enough to force $\Psi_{\beta}$ to be small.\newline
Notice that the support of $\beta$ is $J'$ and that $\beta$ is $2k/N^d$-periodic. In particular, $\beta(t+k/N^d)=-\beta(t)$. Let us also recall that $|E_i|\leq 2d+\lambda$. Combining the two, we get
\begin{equation}\label{phi bound}
    |\Phi_{\beta}(V')|\leq \frac k{N^d}  (4d +2\lambda)
\end{equation}
for any potential $V'$.\newline

We now want to prove that, on the contrary, $\beta$ oscillates slow enough to have $\Psi_{\beta}$ big. More precisely, we will produce a lower bound on $|\Psi_{\beta}(V')[\n]|$ for most part of $\n\in C_{V'}$.\newline

First of all, we estimate the number of sites $\bold{n}$ for which the concentration region $I_{V',\bold{n}}$ lies where $\beta$ is not oscillating.\newline
For $X\subset\R$ and $r>0$ we donote by $X_r$ the $r$-thickening of $X$.\newline 
Let us denote by $D_{V'}\subset C_{V'}$ the set of $\n\in C_{V'}$ for which $d(\partial J_i, I_{V',\n})>\eta'$ for every $i$. If $\n\in C_{V'}\setminus D_{V'}$ then $I_{V',\n}\subset (\partial J_i)_{2\eta'}$ for some $i$. Therefore

\begin{align*}
    \frac 1{N^d}\va'\cdot\#(C_{V'}\setminus D_{V'}) &\leq \frac 1{N^d} \sum_{\n\in C_{V'}\setminus D_{V'}} \nu_{\n}(I_{V',\n})\\
    &\leq \frac 1{N^d} \sum_{\n \in C_{V'}\setminus D_{V'}} \nu_{\n}\left( \bigcup_{i=1}^{2M} (\partial J_i)_{2\eta'} \right) \\
    &\leq \frac 1{N^d} \sum_{\n\in[N]^d} \nu_{\n}\left( \bigcup_{i=1}^{2M} (\partial J_i)_{2\eta'} \right) \\ 
    &\leq \sum_{i=1}^{2M} \Big| (\partial J_i)_{2\eta'}\Big| + \frac{8M}N = 16M\eta'.
\end{align*}

It follows that $\#D_{V'}\geq N^d\Delta - 16MN^d\eta'/\va'\geq  N^d\Delta/2$.\newline

Since $|\beta|\leq1$ and $\beta$ is constantly equal to 1 or $-1$ on $I_{V',\n}$, for $\n\in D_{V'}$ we have
\begin{align*}
    |\Psi_{\beta}(V')[\n]| &\geq \nu_{\n}(I_{V',\n}) - \int_{J'\setminus I_{V',\n}}\beta(t)d\nu_{\n}(t)\\
    &\geq \nu_{\n}(I_{V',\n}) - \nu_{\n}(J'\setminus I_{V',\n})\\
    &\geq \va' - (\va - \va')\\
    &= 2\va' -\va \geq \va/3.
\end{align*}

Now that we proved that $\Phi_{\beta}$ is small and most of the coordinates of $\Psi_{\beta}$ are big we conclude the proof with a "gradient descent" argument. We define a piecewise-linear flow along whose direction $\Psi_{\beta}$ is always big. This would produce a big variation of $\Phi_{\beta}$ along the flow, which should not be possible, given equation \ref{phi bound}.\newline

For our choice of $\beta$, $\rho_{\beta}(1/N)=1/N\eta'<\va/6$. It follows from lemma \ref{psi continuity} that there exists $\delta_0>0$ such that for any $\|V''-V'\|<\delta_0$ and $\n\in D_{V'}$ we have

\begin{equation}
\left|\Psi_{\beta}(V'')[\n]\right| \geq \va/6.
\end{equation}

Let us now define $W_{V'}:[N]^d\rightarrow\R$ by
\begin{equation}
    W_{V'}[\n]=
    \begin{cases}
        \text{sgn}\left(\Psi_{\beta}(V')[\n]\right) &\text{ for } \n \in D_{V'},\\
        0 &\text{ otherwise}.
    \end{cases}
\end{equation}
It follows that, for any $\|V''-V'\|<\delta_0$
\begin{equation}
    \frac 1{N^d} \sum_{\n\in[N]^d} \Psi_{\beta}(V'')[\n]W_{V'}[\n] \geq \frac{\Delta\va}{12}.
\end{equation}
Now, for $0\leq t\leq \delta$, define the flow $V(t)$ in the space of potentials and the sequence of times $T_0=0, T_1, \dots T_m$ so that: $T_{k+1}-T_k\in(0,\delta_0)$, $\delta-T_m\in(0,\delta_0)$, $V(T_k)\in \Omega_N$ for every $k=0,\ldots,m$ and

\begin{equation}
    V(t)=
    \begin{cases}
        V   & \text{if } t=0,\\
        V(T_k) + (t-T_k)W_{V(T_k)} & \text{if } T_k <t\leq T_{k+1}.
    \end{cases}
\end{equation}
This is a piecewise linear flow and its speed, when defined, is unitary (since $W_{V(T_k)}$ is unitary for any $k$). It follows that for any $t\leq\delta$ we have $\|V(t)-V\|\leq\delta$ and $\|V(t)-V(T_k)\|\leq\delta_0$ whenever $T_k<t\leq T_{k+1}$. Then we can exploit all previous inequalities, and deduce, thanks to \ref{phi differentiability}, that
\begin{equation}
    \Phi_{\beta}(V(\delta))-\Phi_{\beta}(V(0)) \geq \delta (\Delta - 16M\eta'/\va')(\va'-\va/2) \geq \frac{\delta\Delta\va}{12}
\end{equation}
We also know that
\begin{equation}
    |\Phi_{\beta}(V(\delta)|,|\Phi_{\beta}(V(0)|\leq \frac k{N^d} (4d +2\lambda)\leq 32(4d +2\lambda)\frac{\eta\eta'}{\va\Delta},
\end{equation}
which gives us, by the triangular inequality,
\begin{equation}
    \frac{\delta\Delta\va}{12} \leq 2\cdot 32(4d +2\lambda)\frac{\eta\eta'}{\va\Delta}
\end{equation}
that is a contaddiction, provided that we set $C=C_{\lambda,d}>12\cdot64\cdot(4d+2\lambda)$, recalling inequality \ref{hypotesis} imposed on $\eta'$.

\end{proof}

\section{The Iterative Scheme}\label{Induction}

Thanks to a double induction we are now ready to prove our main result. The first step, theorem \ref{main theorem}, is the crucial one: given a potential which is $(\va,\eta)$-delocalized we produce a perturbed potential that is $(\frac34\va,\eta')$-delocalized, provided that the size box $N$ is large enough. Then, the second step is a straight-forward repeated iteration of the first one. \newline

It will turn out to be convenient to reformulate theorem \ref{stability of delocalization} in the following way, as in the first induction the loss in $\va$ and $\eta$ will always be proportional to their original sizes.

\begin{theorem}\label{stability}
    There exists $c=c'_{\lambda,d}$ such that the following holds. Given $0<\va,\eta\leq1$, whenever $N>e^{\frac {c}{\va\eta}}$ and $\bar{\va}\geq\frac1{\sqrt2}\va$, given two potentials $V,W:[N]^d\rightarrow[-\lambda,\lambda]$, the following holds. If $V$ is $(\bar{\va},\eta,J)$-delocalized at $\n$ and $\|W-V\|_{\infty}<e^{-\frac {c}{\va\eta}}$, then $W$ is $\left(\frac3{2\sqrt2}\bar{\va},\eta/2,J_{(-\eta/2)}\right)$-delocalized at $\n$.
\end{theorem}

\begin{remark}
    Here, given an inteval $J=[j_1,j_2]$ and $r>0$, we denote by $J_{(-r)}$ the inteval $[j_1+r,j_2-r]$. In other words $J_{(-r)}$ is the interval obtained by shortening $J$ by $r$ at both endpoints.
\end{remark}

Let us recall the constant $c$ from theorem \ref{stability}, and $C$ from theorem \ref{thm perturbation}.\newline
Let $a$ be such that
\begin{equation}
    e^{a/5}\geq 64C,\qquad a\geq c, \qquad a\geq32.
\end{equation}
We remark that $a$ depends only on $\lambda$ and $d$.\newline
We define $h_{\va}:\mathbb{R}\rightarrow\mathbb{R}$ and a tower function $T_{\va}:\mathbb{N}\times\mathbb{R}\rightarrow\mathbb{R}$, by

\begin{equation}
    h_{\va}(x)=\left(e^{\frac{a}{\va}}\right)^x,\qquad\qquad T_{\va}(n,x)=h^{(n)}_{\va}(x)=\underbrace{\left(e^{\frac{a}{\va}}\right)^{\left(e^{\frac{a}{\va}}\right)^{\:\cdot^{\:\:\cdot^{\:\:\cdot^{\left(e^{\frac{a}{\va}}\right)^x}}}}}}_{n \text{ iterations}}
\end{equation}
When the value of $\va$ is clear by the context (as in the proof of the following theorem) we will drop the subscript $\va$ and refer to $h_{\va}$ and $T_{\va}$ just as $h$ and $T$.\newline
We remark that, for $\va\leq1$, (and since $a\geq32$) we have that $h_{\va}(x)\geq 32 \frac1{\va}x$ and $h_{\va}(x)\geq x^2$.

\begin{theorem}\label{main theorem}
    Fix $\lambda>0, d\in\mathbb{N}^+$ and let $0<\va,\eta\leq 1$ be real numbers.\newline
    Suppose that $V:[N]^d\rightarrow[-\lambda,\lambda]$ is $(\va,\eta)$-delocalized on $B\subset [N]^d$. If
    \begin{equation}
        N> T_{{\va}}\left(\left\lceil\frac{8}{\eta}\right\rceil+2,\frac1{\eta}\right),
    \end{equation}
    then there exists $W:[N]^d\rightarrow[-\lambda,\lambda]$ that is $(\frac34 \va, \eta')$-delocalized on a suitable $B'\subset B$ and such that $\|W-V\|_{\infty}<\delta$, where
    \begin{equation}
        \frac1{\eta'}= T_{{\va}}\left(\left\lceil\frac{8}{\eta}\right\rceil+2,\frac1{\eta}\right), \qquad \frac{\#(B\setminus B')}{N^d}\leq\Delta=\frac12 e^{-\frac a{5\va} \frac1{\eta}}, \qquad \delta=\frac12 e^{-\frac a{5\va} \frac1{\eta}}.
    \end{equation}
\end{theorem}

\begin{proof} 
Let us cover $[-\eta/32, 1+\eta/32]$ by closed intervals $J_k$ (for $k=1,\dots,K$), with $|J_k|=\eta/4$, in such a way that $|J_k\cap J_{k+1}|\geq \eta/8$. In order to achieve this we can take $ K=\lceil8/\eta\rceil+1$.\newline

We now set $V_0=V$, $\eta_0=\eta$ and $d_0=e^{-\frac{a}{5\va}\frac1{\eta_0}}$.\newline
By theorem \ref{stability}, we notice that any potential $W:[N]^d\rightarrow[-\lambda,\lambda]$ such that $\|W-V_0\|_{\infty}<d_0$ is $\left(\frac3{2\sqrt2}{\va},\eta/2\right)$-delocalized on $B$.\newline

We now want to define, for $k=1,\dots,K$, a sequence of potentials $V_k:[N]^d\rightarrow[-\lambda,\lambda]$ and sequences of real numbers $\eta_k$, $d_k$ and $\Delta_k$, such that:
\begin{enumerate}
    \item $\frac1{\eta_{k}}=e^{\frac{a}{\va}\frac1{\eta_{k-1}}}=T\left(k,\frac1{\eta}\right)$ (notice that $\eta_k$ is a decreasing sequence);
    
    \item $d_k= \eta_{k}^{1/5} =e^{-\frac{a}{5\va}\frac1{\eta_{k}}}\leq e^{-\frac {c}{\va\eta_k}}$ (notice that $d_k\leq d_{k-1}/2$);
    
    \item $\Delta_k=d_{k-1}/2$ (notice that $\Delta_k\leq\Delta_{k-1}/2$ for $k>1$);
    
    \item $\|V_k-V_l\| < \left(1-\frac1{2^{k-l}}\right)d_l$ for $0\leq l \leq k$;
    
    \item $V_k$ is $\left(\frac1{\sqrt2}\va, \eta_k, J_k \right)$-delocalized on a suitable $B_k\subset B$, with $\frac{\#(B\setminus B_k)}{N^d}\leq \Delta_k$.\newline
\end{enumerate}

The main goal of this construction is to produce delocalization on $J_k$ when constructing $V_k$ (condition (5)), while at the same time staying close enough to the previous potentials $V_i$ (less the $d_i$ respectively) in order to retain the previusly produced delocalization on $J_i$ for $i<k$ (condition (4)).\newline
At each step the smallness of radius of stability of delocalization $d_k$ is imposed by the smallness of the current scale of delocalization $\eta_k$. When producing the next potential $V_k$, since we are forced to find it at distance at most $d_{k}$, we are forced to impose a smaller scale of delocalization $\eta_{k+1}$, according to theorem \ref{thm perturbation}.\newline
Along this procedure we also keep account of how many sites we lose control on at each step, i.e. $\Delta_kN^d$.\newline

We show that this is possible by induction on $k$.\newline
Suppose that the construction works until $V_k$. Set $\delta_{k+1}=d_k/2$ (notice that $\delta_{k+1}=\Delta_{k+1}$). Pick $V'_k\in\Omega_N$ with $\|V'_k-V_k\|<\delta_{k+1}/2$. By condition $(2)$ and $(4)$ and the triangular inequality, we know that any potential $W$ such that $\|W-V_k'\|_{\infty}\leq\delta_{k+1}/2$ is $\left(\frac3{2\sqrt2}{\va},\eta_0/2\right)$-delocalized on $B$. We set $\eta_{k+1}$ and $\Delta_{k+1}$ as specified by conditions $(1)$ and $(3)$. Theorem \ref{thm perturbation} will now provide the existence of a $V_{k+1}$, with $\|V_{k+1}-V_k'\|_{\infty} < \delta_{k+1}/2$, that satisfies $(5)$, provided that
\begin{equation}
    \frac1{\eta_{k+1}}\geq C\frac{\eta_0}2 \frac2{\delta_{k+1}}\frac8{9\va^2}\frac1{\Delta_{k+1}^2}, \qquad\qquad N> \frac1{\eta_0\eta_{k+1}},\frac{6}{\va\eta_{k+1}}.
\end{equation}

Noticing that $6/\va, 1/\eta_0,1/\eta_{k+1} \leq T(\lceil 8/\eta\rceil+1,1/\eta)$ and that $N>T(\lceil 8/\eta\rceil+2,1/\eta)>T(\lceil 8/\eta\rceil+1,1/\eta)^2$, the latter inequality is directly implied by our assumptions on $N$. On the other hand, the former follows from next computation.
\begin{equation}
    \frac1{\eta_{k+1}} = \frac1{d_k^5} \geq e^{\frac a5} \frac1{d_{k}} \frac1{d_k^2} \left(e^{\frac1{\va}}\right)^{\frac a5}\geq 64C \frac1{\delta_{k+1}} \frac1{\Delta_{k+1}^2}\left(\frac1{\va}\right)^2 \frac{\eta_0}2.
\end{equation}
The condition on $\delta_{k+1}$, and the triangular inequality on $V_k$, $V_k'$ and $V_{k+1}$, implies condition $(4)$ for $k+1$. Finally, just set $d_{k+1}$ according to condition $(2)$. This concludes the induction.\newline

We need to take into consideration that the regions of delocalization shrink slightly under perturbation.
Given an inteval $J=[j_1,j_2]$ and $r>0$, we denote by $J_{(-r)}$ the inteval $[j_1+r,j_2-r]$. In other words $J_{(-r)}$ is the interval obtained by shortening $J$ by $r$ at both endpoints.\newline

By condition $(4)$ it holds that $\|V_K-V_k\|<d_k$. This, condition $(5)$ and theorem \ref{stability} allow us to conclude that $V_K$ is $\left(\frac34\va,\frac{\eta_k}2,(J_k)_{(-\eta_k/2)}\right)$-delocalized on $B_k$.\newline
Let us denote by $\widetilde{J_k}$ the shortened interval $(J_k)_{(-\eta_k/2)}$. Then we can say that $V_K$ is $\left(\frac34\va,\frac{\eta_K}2,\widetilde{J_k}\right)$-delocalized on $B_k$.\newline

Let us set $\eta'=T\left(\left\lceil\frac{8}{\eta}\right\rceil+2,\frac1{\eta}\right)$, $\Delta=2\Delta_1$, $\delta=d_0$ and $B'=\bigcap_{k=1}^K B_k $.\newline
Let us point out that $\eta'\leq\eta_K/2$, and therefore that $V_K$ is $(\frac34\va,\eta',\widetilde{J_k})$-delocalized on $B'$, $\forall k$.\newline
Notice that, for $k\geq1$, $\eta_k\leq \eta_1=e^{-\frac a{\va}e^{\frac a{\va\eta}}}\leq \eta/32$, since $a\geq32$. Therefore the intervals $\tilde{J_k}$ are still covering $[0,1]$, and $\widetilde{J_k}$ and $\widetilde{J}_{k+1}$ are still overlapping on an interval of length $\eta/16$, for any $k$. Since $\eta_K\leq\eta/16$, and that $\nu_{V_K,\n}$ are supported on $[0,1]$, we can finally deduce that $V_K$ is $\left(\frac34\va,\eta' \right)$-delocalized on $B'$, and $\|V_K-V\|_{\infty}<\delta$.\newline

We finally check, thanks to condition $(3)$, that
\begin{equation}\label{computing Delta}
    \frac{\#(B\setminus B')}{N^d} \leq  \sum_{k=1}^K\Delta_k \leq \sum_{k=1}^K \frac{\Delta_1} {2^{k-1}}  \leq 2 \Delta_1  = \frac 12 e^{-\frac a{5\va} e^{\frac a{\va}\frac1{\eta}}}.
\end{equation}
\end{proof}

\begin{remark}\label{remark thm}
    $T_{\va}$ can be replaced by $T_{\tilde{\va}}$ for any $0<\tilde{\va}\leq\va$, beacuse of its monotonicity in $\va$.
\end{remark}

We can final prove our main result, by a simple iteration of theorem \ref{main theorem}.\newline

Let us now define $H_{\va}:\R\rightarrow \R$ and $G_{\va}:\N\times\R \rightarrow \R$ as follow
\begin{equation}
    H_{\va}(x)=T_{\va}\left(\lceil 8x\rceil+2, x \right), \qquad \qquad G_{\va}(n,x)=H_{\va}^{(n)}(x)=\underbrace{H_{\va}\circ\dots\circ H_{\va}}_{n \text{ iterations}} (x)
\end{equation}

\begin{theorem}\label{final theorem}
    Fix $\lambda>0$ and $d,N\in\N^+$. Let $\va>0$ and $\delta>0$.\newline
    Let $V:[N]^d\rightarrow[-\lambda,\lambda]$ a potential and $H_V$ its associated Schr\"odinger box operator on $l^2([N]^d)$. If
    \begin{equation}
        N > G_{\va}\left(\left\lceil \log_{\frac34}(\va)\right\rceil, \frac{5\va}{a}\log\left(\frac1{\delta}\right)\right)
    \end{equation}
    there exists a potential $W:[N]^d\rightarrow[-\lambda,\lambda]$, with $\|W-V\|_{\infty}<\delta$ such that, for all, except at most $\delta N^d$, $\n\in[N]^d$, the spectral measure $\sigma_{H_W,\n}$ has no atoms of weight more that $\va$.
\end{theorem}

\begin{proof}
    Set $K=\left\lceil \log_{\frac34}(\va)\right\rceil$, $1/\eta_0=\frac{5\va}{a}\log(\frac1{\delta})$, $V_0=V$, $B_0=[N]^d$.\newline
    Notice that $V$ is always $(1,\eta_0)$-delocalized on $[N]^d$.\newline
    For $k=1,\dots,K$ we define a sequence of potentials $V_k:[N]^d\rightarrow[-\lambda,\lambda]$ and sequences of real numbers $\eta_k$, $\delta_k$ such that:
    \begin{enumerate}
        \item $1/\eta_{k+1}= T_{\va}(\lceil 8/\eta_k\rceil+2,1/\eta_k)=G_{\va}(k,1/\eta_0)$;
        \item $\delta_{k+1}=\frac12 e^{-\frac a{5\va} \frac1{\eta_k}}$;
        \item $V_k$ is $\big(\left(\frac34\right)^k, \eta_k\big)$-delocalized on a suitable $B_k\subset B_{k-1}$ with $\frac{\#(B_{k-1}\setminus B_k)}{N^d}\leq \delta_k$;
        \item $\|V_k-V_{k-1}\|_{\infty}\leq \delta_k$.
    \end{enumerate}
    We proceed again by induction, and we only have to check that it is possible to fulfill contitions $(3)$ and $(4)$. This follows by an direct application of theorem \ref{main theorem}. We are able to apply the theorem to $V_k$, with the required parameters, since, for $k<K$, we have that $(3/4)^k<\va$ and that $N,1/\eta_{k+1}\geq H_{\va}(1/\eta_k)$.\newline

    Now we just notice that, since $\eta_{k+1}\leq\frac12\eta_k$ we also have that $\delta_{k+1}\leq\frac12\delta_k$, giving us that 
    \begin{equation}
        \sum_{k=1}^K \delta_k \leq \sum_{k=1}^K \frac1{2^{k-1}}\delta_1 \leq 2\delta_1 = e^{-\frac a{5\va} e^{\frac a{\va}\frac1{\eta_0}}} = \delta.
    \end{equation}

    We now set $W=V_K$ and remark that for all $\n\in B_K$, $\nu_{H,\n}$ is $(\va,\eta_K)-$delocalized, since $\va\geq(3/4)^K$. It follows that no atom of $\nu_{H,\n}$ can have weight bigger than $\va$. Since atoms of $\nu_{H,\n}$ have weight bigger or equal than the ones of $\sigma_{H_W,\n}$, the same conclusion holds for those measures. To conclude we then just need the two following inequalities:

    \begin{equation}
        \|W-V\|_{\infty} \leq \sum_{k=1}^K \|V_k-V_{k-1}\|_{\infty} \leq \sum_{k=1}^K \delta_k < \delta,
    \end{equation}
    and
    \begin{equation}
        \frac{\#([N]^d\setminus B_K)}{N^d} \leq \sum_{k=1}^K \frac{\#(B_{k-1}\setminus B_k)}{N^d} \leq \sum_{k=1}^K \delta_k < \delta.
    \end{equation}
\end{proof}

\section{Schr\"odinger Operators on Graphs}

In this section we briefly discussed how our proof can be generalized. The result is the following criterion for asymptotically dense delocalization.

\begin{theorem}\label{IDS regularity implies delocalization}
    Suppose that a family of graphs $\mathcal{F}$ has asymptotically $\omega$-uniformly continuous IDSs, for some modulus of continuity $\omega:\R^+\rightarrow\R^+$. Then $\mathcal{F}$ admits asymptotically dense delocalization.
\end{theorem}

\begin{proof}
We need to show that any time that the specific graph structure of  $[N]^d$ plays a role in the proof, the only property exploited is the asymptotical $log$-Holder continuity of the IDSs.\newline

In lemma \ref{Fourier - periodic vs box}, we used the amenability of $\mathbb{Z}^d$, and theorem \ref{log continuity} holds specifically for ergodic Schredinger operators on $\mathbb{Z}^d$. Anyway, both of them were only used to prove theorem \ref{L infty ids}, which is exactly the asymptotical $log$-Holder continuity of the IDSs for the bos Operators.\newline

We used specific properties of Schredinger opertors on $[N]^d$ in only two more istances: the proofs of lemma \ref{psi continuity} and lemma \ref{average nu similar to Lebesgue}. They both rely on the same fact: that the maximal size of discontinuities of the IDSs goes to zero when the graph size $N$ goes to $\infty$, that is a weaker property than the asymptotical $\log$-Holder continuity of the IDSs.
\end{proof}

We end the section we a list of few examples, showing that the asymptotical uniform continuity of IDSs depends subtly on the geometry of the graphs.

\begin{example}
    Set $\mathcal{F}_1=\{G_n:n\in\mathbb{N}\}$, where $\mathcal{G}_n$ is the $n$-diamonds chain graph, shaped as explained in the next picture.\newline
    For $V=0$ the eigenspace of $0$ has dimension at least $n$, while $\#\mathcal{V}(\mathcal{G}_n)=3n+1$. Therefore the IDS has a discontinuity of size approximately 1/3 at 0, and $\mathcal{F}_1$ does not have asymptotycally $\omega$-uniformly continuous IDSs for any modulus of continuity $\omega$.
\end{example}

\begin{figure}[ht]
\centering
\begin{tikzpicture}[every node/.style={circle, draw, inner sep=1.5pt}, node distance=1.2cm and 1.2cm]

\node (a1) at (0,0) {};
\node (a2) [above right=of a1] {};
\node (a3) [below right=of a2] {};
\node (a4) [below right=of a1] {};

\draw (a1) -- (a2) -- (a3) -- (a4) -- (a1);

\node (b2) [above right=of a3] {};
\node (b3) [below right=of b2] {};
\node (b4) [below right=of a3] {};

\draw (a3) -- (b2) -- (b3) -- (b4) -- (a3);

\node (c2) [above right=of b3] {};
\node (c3) [below right=of c2] {};
\node (c4) [below right=of b3] {};

\draw (b3) -- (c2) -- (c3) -- (c4) -- (b3);

\node (d2) [above right=of c3] {};
\node (d3) [below right=of d2] {};
\node (d4) [below right=of c3] {};

\draw (c3) -- (d2) -- (d3) -- (d4) -- (c3);

\end{tikzpicture}
\caption{The 4-diamonds chain graph}
\end{figure}
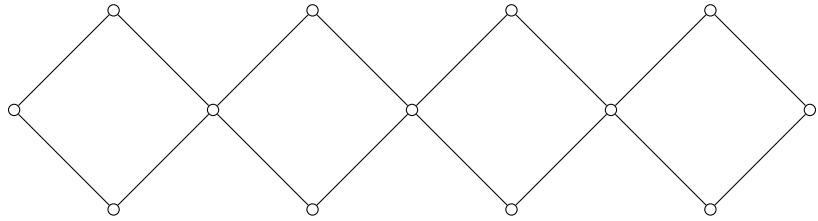

\begin{example}
    Set $\mathcal{F}_2=\{G_n:n\in\mathbb{N}\}$, $\mathcal{V}(\mathcal{G}_n)=[N]\times[2]$, and two vertices are connected if and only if their $x$-coordinates differ exactly by 1.\newline
    Here, again for $V=0$, the eigenspace of $0$ has dimension has dimension at least $n$, while $\#\mathcal{V}(\mathcal{G}_n)=2n$. Hence the situation is the same one described in the previous example.
\end{example}

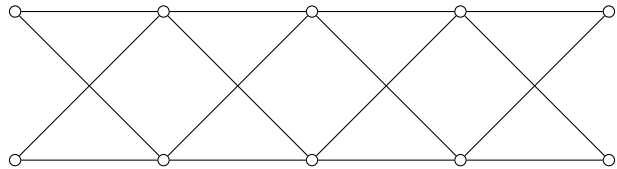
\begin{figure}[ht]
\centering
\begin{tikzpicture}[every node/.style={circle, draw, inner sep=1.5pt},
                    node distance=1.8cm]

\node (a1) at (0,0) {};
\node (a2) [right=of a1] {};
\node (a3) [below=of a2] {};
\node (a4) [below=of a1] {};

\draw (a1) -- (a2);
\draw (a4) -- (a3);
\draw (a1) -- (a3);
\draw (a4) -- (a2);

\node (b2) [right=of a2] {};
\node (b3) [below=of b2] {};

\draw (a3) -- (b2);
\draw (a2) -- (b3);
\draw (a3) -- (b3);
\draw (a2) -- (b2);

\node (c2) [right=of b2] {};
\node (c3) [below=of c2] {};

\draw (c3) -- (b2);
\draw (c2) -- (b3);
\draw (c3) -- (b3);
\draw (c2) -- (b2);

\node (d2) [right=of c2] {};
\node (d3) [below=of d2] {};

\draw (c3) -- (d2);
\draw (c2) -- (d3);
\draw (c3) -- (d3);
\draw (c2) -- (d2);

\end{tikzpicture}
\caption{The graph $G_5$ from the family $\mathcal{F}_2$}
\end{figure}

\begin{example}
    Set $\mathcal{F}_3=\{G_n:n\in\mathbb{N}\}$, where $\mathcal{V}(\mathcal{G}_n)=[2]\times[N]$ and two vertices are connected if and only if they are at distance 1.\newline 
    This family of graphs, has asymptotically $\log$-Holder continuous IDSs, as we will show in the next section.
\end{example}

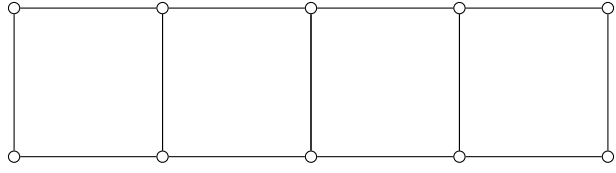
\begin{figure}[ht]
\centering
\begin{tikzpicture}[every node/.style={circle, draw, inner sep=1.5pt},
                    node distance=1.8cm]

\node (a1) at (0,0) {};
\node (a2) [right=of a1] {};
\node (a3) [below=of a2] {};
\node (a4) [below=of a1] {};

\draw (a1) -- (a2) -- (a3) -- (a4) -- (a1);

\node (b2) [right=of a2] {};
\node (b3) [below=of b2] {};

\draw (a2) -- (b2) -- (b3) -- (a3) -- (a2);

\node (c2) [right=of b2] {};
\node (c3) [below=of c2] {};

\draw (b2) -- (c2) -- (c3) -- (b3) -- (b2);

\node (d2) [right=of c2] {};
\node (d3) [below=of d2] {};

\draw (c2) -- (d2) -- (d3) -- (c3) -- (c2);

\end{tikzpicture}
\caption{The graph $G_5$ from the family $\mathcal{F}_3$}
\end{figure}

\section{Continuity of the IDS for certain Graph Operators}

\begin{theorem}\label{IDS regularity graphs}
    Let $\mathcal{G}$ be a $\mathbb{Z}$-invariant finite section propagating graph of maximal degree $k$. Let $V:\mathcal{V}(\mathcal{G})\rightarrow [-\lambda,\lambda]$ be a periodic potential, then the IDS of the associated Schr\"odinger operator $H_{\mathcal{G}, V}$ is $C_{\lambda,k}$-$\log$ Holder continuous.
\end{theorem}

Before proving the theorem let us state some corollaries.

\begin{theorem}\label{main results graph}
    Let $\mathcal{G}_N$ be a sequence of truncated propagating graphs with increasing lengths and bounded maximal degree. Then $\mathcal{G}_N$ has asymptotically dense delocalization.
\end{theorem}

\begin{proof}
    In view of theorem \ref{IDS regularity implies delocalization}, it will be sufficient to show that the IDSs of Schr\"odinger Operators on $\mathcal{G}_N$ are asymptotically $C$-$\log$ Holder continuous, where $C$ depends only on the bound on the degrees $k$ and a fixed bound on the potentials $\lambda$.\newline

    The regularity of the IDSs is proved in the same way as theorem \ref{L infty ids}.\newline 
    Theorem \ref{IDS regularity graphs} has here the same role played by theorem \ref{log continuity}, providing the information that the IDSs of the infinitely extended graphs are indeed $\log$-Holder continuous.\newline
    Finally one can prove that, as the length increases, the IDSs of operators on $\mathcal{G}_N$ converge in $\ell^{\infty}$ to the ones of the extended operators, following the proof of lemma \ref{Fourier - periodic vs box} and concluding thanks to lemma \ref{Fourier controls cumulative}. We remark that the speed of convergence depends only on the length.
\end{proof}

We now turn our attention to the proof of theorem \ref{IDS regularity graphs}, which follows \cite{craig1983log}.\newline

The main idea is to prove a Thouless formula for $H_{\mathcal{G},V}$.\newline

The first step is to reduce the operator to a long (but finite) range operator on $\ell^2(\mathbb{Z})$. We extend the partial order given on $\mathcal{V}(\mathcal{G})$ to a total one, that we again denote by $\leq$. It easy to check that $\mathcal{G}$ is propagating also with respect to the total order.\newline
We fix a convex fundamental domain $X$, and we denote by $r$ its cardinality. Since the set of vertices $\mathcal{V}(\mathcal{G})$ is totally ordered, countable and discrete, we can identify it with $\mathbb{Z}$ (and $X$ with $[1,r]$). We can therefore assume $H$ to be an operator from $\ell^2(\mathbb{Z})$ to itself, and to be defined by
\begin{equation}
    Hu(n) = V(n)u(n) +  \sum_{i=-r}^{r} A_{n,n+i}u(n+i).
\end{equation}
We exploited that $N(v)\subset[n-r,n+r]$.\newline
When $H$ is a Schr\"odinger opertaor $A_{n,n+i}$ are the entries of the adjacency matrix of $\mathcal{G}$, but in the following it will be sometimes convenient to work with a more general class of operators, that we will refer to as $(2d+1)$-diagonal operators. In that generality, we will only assume the $A_{i,j}$'s to be bounded and that $A_{n,n+r}=A_{n,n-r}=1$ for any $n$. (Symmetry is not be required).\newline
Suppose now that $u$ is a formal solution to $Hu=Eu$, for $E\in\mathbb{C}$. Insisting on the fact that $A_{n,n+r}=1$ for any $n$, we get that
\begin{equation}
    u(n+r) = (E-V(n))u(n) - \sum_{i=-r}^{r-1} A_{n,n+i}u(n+i).
\end{equation}
We denote the vector $(u(a),u(a+1),\dots,u(b))^t$ by $u([a,b])$. The previous formula can be interpreted in terms of transfer matrices. Indeed,
\begin{equation}
    u[(n+1,n+2r)] = A_n(E) u([n,n+2r-1]),
\end{equation}
where

\begin{equation}
A_n(E)=
    \begin{tikzpicture}[baseline=(current bounding box.center)]
\matrix (m) [matrix of math nodes,nodes in empty cells,right delimiter={]},left delimiter={[} ]{
0  & 1 & & &  \\
\phantom{a_{n+2r-1,n}}  & \phantom{a_{n+2r-1,n}} & \phantom{a_{n+2r-1,n}} & \phantom{a_{n+2r-1,n}} &  \\
 & & & & \\
  & & & & \\
  & & & & \\
  & & & & \\
 & & & 0 & 1\\
a_{n,n+2r}  & & & & a_{n+2r-1,n+2r}\\
} ;
\draw[loosely dotted] (m-1-1)-- (m-7-4);
\draw[loosely dotted] (m-1-2)-- (m-7-5);
\draw[loosely dotted] (m-8-1)-- (m-8-4);
\end{tikzpicture},
\end{equation}

and we set (for later convenience) $a_{m,n+2r}=-A_{n+r,m}$ for $i\neq n$ and $a_{n+r,n+2r}=E - V(n+r) - A_{n+r,n+r}$.

\begin{remark}\label{degree of entries}
    For every $n$, $a_{n,n+r}$ is a monic monomial in $E$, while $a_{n,m}$ are constant in $E$, whenever $m\neq n+r$.
\end{remark}

We set $T_n(E)=A_n(E)\dots A_1(E)$, in such a way that $u([n+1,n+2r])=T_n(E)u([1,2r])$. We also set $T_{a,b}(E)=A_{b}(E)\cdots A_{a+1}(E)$. The periodicity of graph structure and of the potential (with possibly different periods), imply the periodicity of the sequence of matrices $A_n(E)$ for every fixed $E$. This guarantees the existence of Lyapunov exponents $\gamma_1(E)\geq\dots\geq\gamma_{2r}(E)$, that can be defined through the formula
\begin{equation}
    \gamma_1(E)+\dots+\gamma_{k}(E) = \lim_{n\rightarrow\infty} \frac1n \log \left\|{\bigwedge} ^k T_E(n) \right\|.
\end{equation}

Since $a_{n,n-r}=-1$, we have that $\det(A_n(E))=1$ for every $n$ and $E$. This implies that $\gamma_1(E)+\dots+\gamma_{2r}(E)=0$. The matrices $A_n(E)$ preserve  each a symplectic form, but not a common one. Therefore, $T_n(E)$ needs not to be symplectic.\newline

We define the growth rate, as
\begin{equation}
    \gamma(E) = \gamma_1(E)+\dots+\gamma_{r}(E),
\end{equation}
and we remark that $\gamma(E)\geq0$ for every $E$. We can now state the Tohuless formula.

\begin{lemma}\label{Thouless formula}
    For any Schr\"odinger operator $H$ on a propagating graph $\mathcal{G}$, with periodic potential $V$, and for any $E\in\mathbb{C}$ with $Im(E)\neq 0$, we have that
    \begin{equation}
        \int_{\mathbb{R}}\log|E'-E|d\tilde\sigma(E') = {\gamma(E)}.
    \end{equation}
\end{lemma}

The proof of lemma \ref{Thouless formula} follows \cite{craig1983log}, with minor changes, that we explain in detail in the following section \ref{Proof of Thouless formula}.\newline

Since $\gamma(E)\geq0$ for any $E$ such that $Im(E)\neq0$, lemma \ref{IDS regularity graphs} follows immediately from the following statement (see lemma 1.5 from \cite{craig1983log} for the proof). 

\begin{lemma}
    Let $\mu$ be a positive measure on $\mathbb{R}$, supported on $[-M,M]$. Suppose that for any complex number $E$, with $Im(E)\neq 0$ we have that
    \begin{equation}
        \int_{\mathbb{R}}\log|E'-E|d\mu(E') \geq 0.
    \end{equation}
    Then, the cumulative function $F(t)=\mu((-\infty,t])$ is $C_M$-$\log$ Holder continuos in $t$.
\end{lemma}

\section{Proof of the Thouless formula}\label{Proof of Thouless formula}

Recalling that $\mathcal{V}(\mathcal{G})$ is identified with $\mathbb{Z}$, we denote by $\mathcal{G}_n$ the subgraph of $\mathcal{G}$ induced by $\{1+r,\dots,n+r\}$. Its length is $\lfloor n/r \rfloor$, and we denote it by $l$. We also denote by $H_n$ the Schr\"odinger operator induced by $\mathcal{G}_n$ and $V|_{\{r+1,\dots,n+r\}}$. The eigenvectors, with eigenvalue $E$, of the truncated operator $H_n$ are formal solution of the equation $Hu=Eu$ that, in addition, satisfy the boundary condition $u([1,r])=u([n+r+1,n+2r])=0$, restricted to $[r+1,n+r]$.\newline
Let us denote by $\tilde\sigma_{(n)}$ the density of states of $H_n$. As already stated in the proof of theorem \ref{main results graph}, $\tilde\sigma_{(n)}\rightarrow\tilde\sigma$ weakly, as $l\rightarrow\infty$.\newline

Our strategy is to prove a version of Thouless formula for the truncated operators and pass to the limit. We therefore focus on the operator ${\bigwedge} ^r T_n(E)$, trying to relate its norm to the eigenvalues of $H_n$.\newline

The next lemma states that, while testing with a specific pair of multi-vectors we observe the desired rate of growth. It is proved by the Thouless argument. We need some preliminary considerations. \newline

We recall that $T_n(E)$ is a linear map from $\mathbb{R}^{2r}$ to itself and we denote the standard basis of $\mathbb{R}^{2r}$ by $\{e_1,\dots,e_r,d_1,\dots,d_r\}$. We provide a description of the matrix entries of $T_n(E)$.

\begin{equation}\label{formula matrix entries}
    u(n+j) =  \sum_{i=1}^{2r} u(i) \sum_{m_0=i} a_{m_0,m_1}\cdots a_{m_{k-1},m_k}, 
\end{equation}
where we have that $i=m_0<m_1<\dots<m_k=n+j$ and the difference between any two consecutive terms of the previous inequality is at most $2r$. As a consequence of the previous formula, and remark \ref{degree of entries}, we have the following.

\begin{remark}\label{matrix entries polynomials}
    $\big(T_n(E)\big)_{i,j}$ is a polynomial in $E$, of degree at most $\lfloor(n+j-i)/r\rfloor$. Furthermore, the coefficient of degree $\lfloor(n+j-i)/r\rfloor$ is $1$ when $r$ divides $n+j-i$.
\end{remark}

\begin{lemma}\label{Thouless formula step 1}
For any $(2r+1)$-diagonal operator, we have that
    \begin{equation}\label{multi-vector test}
        \frac1{n}\log\left(\langle d_1\wedge\dots\wedge d_r, {\bigwedge} ^r T_n(E) (d_1\wedge\dots\wedge d_r)\rangle \right) = \int_{\mathbb{R}}\log|E'-E|d\tilde\sigma_{(n)}(E').
    \end{equation}
    In particular,
    \begin{equation}
        \int_{\mathbb{R}}\log|E'-E|d\tilde\sigma(E') \leq \gamma(E).
    \end{equation}
\end{lemma}

\begin{proof}
    We set $P(E) = \langle d_1\wedge\dots\wedge d_r, {\bigwedge} ^r T_n(E) (d_1\wedge\dots\wedge d_r)\rangle$. Let $Q$ be the projection from $\mathbb{R}^{2r}$ to $span\{d_1,\dots,d_r\}$. One can observe that $P(E)$ is the detrminant of $QT_{n}(E)$, seen as an operator from $Q\mathbb{R}^{2r}$ to itself.\newline
    We claim that $P(E)$ is a monic polynomial in $E$ of degree exactly $n$.\newline
    Let $(R_{i,j})_{i,j}$ be the matrix representation of $QT_n(E)$ with respect to the basis $\{d_1,\dots,d_r\}$. Then,
    \begin{equation}
        \det(QT_n(E)) = \sum_{\sigma\in S_r} \text{sgn}(\sigma) \prod_{i=1}^r R_{i,\sigma(i)}.
    \end{equation}
    In view of remark \ref{matrix entries polynomials}, $R_{i,\sigma(i)}$ is a polynomial in $E$ and its degree is $\lfloor (n+i-\sigma(i))/r\rfloor$. We have that $\sum_{i=1}^r (n+i-\sigma(i))/r = n$, hence the degree of every term of the previous sum is less or equal than $n$. Equality holds if and only if $r$ divides $n+i+\sigma(i)$ for every $i$, and this is true for exactly one permutation, that we donote by $\bar\sigma$. Finally we conclude that $\prod_{i=1}^r R_{i,\bar\sigma(i)}$ is in fact monic, since each of its factors is, again, thanks to remark \ref{matrix entries polynomials} and the fact that $r$ divides $n+i+\sigma(i)$.\newline

    We claim that if $E$ is an eigenvector of $H_{l}$, the $P(E)=0$.\newline
    If $E$ is an eigenvalue of $H_n$, there exists a solution $u$ of $Hu=Eu$, satisfying $u([1,r])=u([n+1+r, n+2r])=0$. This implies that $u[(1,2r)]\in Q\mathbb{R}^{2r}$. Then, $QT_{n}(E) u[(1,2r)] = Qu[(n+1,n + 2r)] = u([n+1+r, n+2r])=0$. Proving that $P(E)=0$ since $QT_{n}(E)$ is not invertible from $Q\mathbb{R}^{2r}$ to itself. \newline
    
    Suppose that $H_n$ has simple spectrum, in this case the number of distinct eigenvalues matches the number of roots of $P$, determining them uniquely. Hence $P(E)$ has to satisfying equation \ref{multi-vector test}.\newline
    Simplicity of the spectrum is a dense condition in the space of potentials and both terms of the equality \ref{multi-vector test} are continuous with respect to the potential. This concludes the proof.
\end{proof}

To control the norm of ${\bigwedge} ^r T_n(E)$ we need to test it agaist a spanning set of pairs of $r$-vectors. For the following family of $r$-vectors it is possible to reduce this computation to an application of lemma \ref{Thouless formula step 1}. This is done by extending the operator $H_{n}$ close to the boundary with a 'fake potential'.\newline
We define the family of $r$-vectors
\begin{equation}
\tilde S=\left\{  (e_1+\tilde d_1)\wedge\dots\wedge (e_r+\tilde d_r) : \tilde d_i\in span\{d_1,\dots,d_r\} \text{ for } i=1,\dots,r   \right\}.
\end{equation}
One can check that $\tilde S$ is a generating set of $\bigwedge^r(\mathbb{R}^{2r})$.\newline
We extract a finite generating set $S$ out of $\tilde S$.\newline

\begin{lemma}
    Let $E$ be a complex number with $Im(E)\neq0$ and $s_1,s_2$ be two $r$-vectors in $S$. Then, there exists a sequence of measures $\tilde\sigma_{(n),s_1,s_2}$ on $\mathbb{C}$, sharing a common bounded support and weakly converging to $\tilde\sigma$, such that
    \begin{equation}\label{multi-vector test 2}
        \frac1{n+2r}\log\Big(\langle s_1, {\bigwedge} ^r T_n(E) s_2\rangle \Big) =  \int_{\mathbb{R}}\log|E'-E|d\tilde\sigma_{(n),s_1,s_2}(E').
    \end{equation}
\end{lemma}

\begin{proof}
We set $s_1=(e_1+d_1^{b})\wedge\dots\wedge (e_r+d_r^{b})$ and $s_2=(e_1+d_1^{\#})\wedge\dots\wedge (e_r+d_r^{\#})$, for $d^b_i,d^{\#}_i\in span\{d_1,\dots,d_r\}$ for $i=1,\dots,r$.\newline

We start by defining a modified truncated operator $\tilde H_n$. First we extend the truncated graphs by two copies of the fundamental domain: we set $\tilde G_n$ to be the graph induced by $\{1, n+2r\}$. And we define the operator $\tilde H_n$ again by the formula
\begin{equation}
    \widetilde H_l u(m) = \widetilde V(m)u(m) +  \sum_{i=-r}^r \widetilde A_{m,m+i}u(m+i),
\end{equation}
where $\widetilde V(m)$ coincides with $V(m)$ for $m\in[1+r,n+r]$ and $\tilde A_{m,m+i}=A_{m,m+i}$ when $m,m+i\in[1+r,n+r]$. Out od $[1+r,n+r]$ we define $\widetilde A_{m,m+i}$ and $\widetilde V(n)$ artificially: this is what we called boundary fake potential.\newline
We define the transfer matrices $\widetilde T_{a,b}(E)$ for $\widetilde H_n$, and the corresponding entries $\tilde a_{i,m}$. We have that $\tilde a_{i,m} = a_{i,m}$ for any $m\in[1+2r, n+2r]$. Defining $\tilde a_{i,m}$ for $m\in[1+r,2r]$ and $m\in[n+1+2r,n+3r]$, implicitly extends the definition of $A$ and $V$ to $\widetilde A$ and $\widetilde V$.\newline
In particular, for $j \in [1,r]$ we define $\tilde a_{i,r+j} = d_{i}^b(j)$ for $i\in[1,r]$, and $\tilde a_{i,r+j} = 0$ for $i\not\in [1,r]$. Equation \ref{formula matrix entries} ensure that $\widetilde T_{-r,0}(E)$ sends $d_i$ to $e_i+d_i^b$ for any $i=1,\dots,r$. Similarly, we can define $\tilde a_{i,m}$ for $m\in [n+1+2r,n+3r]$ in such a way that $\widetilde T_{n,n+r}(E)^t$ sends $d_i$ to $e_i+d_i^{\#}$ for any $i=1,\dots,r$.\newline

We remark that $\tilde H_n$ is not a self-adjoint operator, hence its spectrum does not necessarily lie on $\mathbb{R}$. Nevertheless its spectral radius is bounded by a constant that depends only on $\lambda$, $r$, $E$ and the fixed vectors $d_1^{\#}, \dots,d_r^{\#},d_1^b,\dots,d^b_r$.\newline
We define $\tilde\sigma_{(n),*}$ to be the density of states of $\tilde H_n$. This measures, in view of the previous discussion, share a common bounded support in $\mathbb{C}$. They converge weakly to $\tilde\sigma$ in view of lemma 3.3 from \cite{craig1983log} (the proof works unchannged for our setting).\newline

We are left to prove equation \ref{multi-vector test 2}.\newline

\begin{align}
    &\frac1{n+2r}\log\Big(\left\langle s_1, {\bigwedge} ^r T_n(E) s_2\right\rangle \Big)\\
    =&\frac1{n+2r}\log\Big(\left\langle (e_1+d_1^{\#})\wedge\dots\wedge (e_r+d_r^{\#}),{\bigwedge} ^r \widetilde T_n(E) ((e_1+d_1^b)\wedge\dots\wedge (e_r+d_r^b))\right\rangle \Big) \\
    =& \frac1{n+2r} \log\Big(\left\langle (d_1\wedge\dots\wedge d_r),{\bigwedge} ^r  \widetilde T_{-r,n+r}(E) (d_1\wedge\dots\wedge d_r)\right\rangle \Big)\\
    =& \int_{\mathbb{R}}\log|E'-E|d\tilde\sigma_{(n),s_1,s_2}(E').
\end{align}
For the last equality we applied lemma \ref{Thouless formula step 1}, to the operator $\tilde H_n$.
\end{proof}

Since $S$ is a finite spanning set, we have that
\begin{equation}
    \left\| {\bigwedge} ^r T_n(E) \right\| \leq C_S \sup_{s_1,s_2\in S} \left|\left\langle s_1, {\bigwedge} ^r T_n(E) s_2\right\rangle \right|.
\end{equation}
Hence,

\begin{align}
     {\gamma(E)} &\leq \lim_{n\rightarrow\infty} \frac1{n} \log \left\| {\bigwedge} ^r T_n(E) \right\|\\
    &\leq \lim_{n\rightarrow\infty} \frac{\log(C_S)}{n} + \limsup_{n\rightarrow\infty} \frac1{n} \sup_{s_1,s_2\in S} \left|\left\langle s_1, {\bigwedge} ^r T_n(E) s_2\right\rangle \right|\\
    &\leq  \sup_{s_1,s_2\in S}   \lim_{n\rightarrow\infty}  \frac{n+2r}n\int_{\mathbb{R}}\log|E'-E|d\tilde\sigma_{(n),s_1,s_2}(E')\\
    &= \int_{\mathbb{R}}\log|E'-E|d\tilde\sigma(E').
\end{align}

Where the last equality follows from the weak convergence of $\tilde\sigma_{(n),s_1,s_2}$ to $\tilde\sigma$. We remark that $\log(\cdot-E)$ is not a bounded function on $\mathbb{C}$, but the convergence of the integral follows from an approximation argument, see lemma 3.4 from \cite{craig1983log}.

\bibliographystyle{siam}
\bibliography{Bibliography}


\end{document}